\documentclass[11pt,reqno,a4paper]{amsart}
\usepackage{amsmath,amssymb,amsthm,amsfonts}
\usepackage[margin=1in]{geometry}
\usepackage{color}
\usepackage[colorlinks=true]{hyperref}
\usepackage{xcolor}
\usepackage{enumitem}
\usepackage{hyperref}
\usepackage[all]{hypcap}[2006/02/20]
\usepackage{mathrsfs}
\usepackage{epsfig, subcaption}

\newcommand{\bb}[1]{\mathbb{#1}}

\newcommand{\pard}[2]{\frac{\partial #1}{\partial #2}}

\newcommand{\ho}{\left(\frac{d}{dt} -\Delta \right)}
\newcommand{\ddt}[1]{\frac{ d #1}{dt}}
\newcommand{\ip}[2]{\left \langle #1 , #2 \right\rangle}

\newcommand{\n}{\nabla}
\newcommand{\e}{\epsilon}

\renewcommand{\l}{\lambda}
\renewcommand{\d}{\delta}
\renewcommand{\a}{\alpha}

\newcommand{\ra}{\rightarrow}
\newcommand{\ov}{\overline}

\newcommand{\fr}[2]{\frac{#1}{#2}}

\newcommand{\veps}{\varepsilon}

\newcommand{\II}{I\!I}

\newcommand{\IIM}{I\!I}
\newcommand{\IIS}{\widetilde{I\!I}}
\newcommand{\up}[1]{{}^{#1}\!}
\newcommand{\CY}{\mathcal{Y}}

\newcommand{\CE}[3]{{C^{#1,#2}(#3)}}
\newcommand{\CP}[2]{{C^{#1;\frac{#1}{2}}(#2)}}

\newcommand{\CPN}[1]{{\Gamma^{#1;\frac{#1}{2}}_T}}
\newcommand{\CPNn}[1]{{#1;T}}
\newcommand{\CPNt}[2]{{\Gamma^{#1;\frac{#1}{2}}_{#2}}}
\newcommand{\CPNtn}[2]{{#1;#2}}
\newcommand{\CPNb}[1]{{\Gamma^{#1;\frac{#1}{2}}_{\partial, T}}}
\newcommand{\CPNbn}[1]{{#1;\partial;T}}

\newcommand{\CPNbtn}[2]{{#1;\partial; #2}}
\newcommand{\CEN}[2]{\CE{#1}{#2}{N\ov{M}_{0}}}
\newcommand{\CENn}[2]{{#1, #2}}

\renewcommand{\b}{\beta}
\newcommand{\vn}{\tilde{v}}
\newcommand{\np}{\widetilde{\n}}
\newcommand{\hatG}{G}
\newcommand{\Voo}{\nu_0}
\newcommand{\Aaa}{A^{\delta,\tau,\Theta}}
\newcommand{\p}{\partial}

\newcommand{\hide}[1]{}%{\textcolor{red}{#1}}

\theoremstyle{plain}
\newtheorem{theorem}{Theorem}
\newtheorem{lemma}[theorem]{Lemma}

\newtheorem{proposition}[theorem]{Proposition}
\newtheorem{cor}[theorem]{Corollary}
\theoremstyle{definition}

\newtheorem{example}{Example}
\theoremstyle{remark}
\newtheorem{remark}[theorem]{Remark}
\newtheorem*{ack}{Acknowledgement}

\title{Lagrangian mean curvature flow with boundary}
\author{Christopher G. Evans}
\author{Ben Lambert}
\author{Albert Wood}

\address{Department of Mathematics, University College London, Gower Street, London, WC1E 6BT, United Kingdom}

\email{christopher.evans.13@ucl.ac.uk, b.lambert@ucl.ac.uk,  albert.wood.16@ucl.ac.uk}

\begin{document}
	
	\begin{abstract}
	We introduce Lagrangian mean curvature flow with boundary in Calabi--Yau manifolds by defining a natural mixed Dirichlet-Neumann boundary condition, and prove that under this flow, the Lagrangian condition is preserved. We also study in detail the flow of equivariant Lagrangian discs with boundary on the Lawlor neck and the self-shrinking Clifford torus, and demonstrate long-time existence and convergence of the flow in the first instance and of the rescaled flow in the second. 
	\end{abstract}
	
	\maketitle
	
	\section{Introduction}\label{sec-intro}
	
	The foundational result of Lagrangian mean curvature flow is that in Calabi--Yau manifolds, mean curvature flow preserves closed Lagrangian submanifolds (see the work of Smoczyk \cite{Smoczyk1996}). It is natural then to ask whether this can be generalised to submanifolds with boundary. Equivalently, what is a well-defined boundary condition for Lagrangian mean curvature flow? In this paper we answer this question, and show that the resulting flow exhibits good behaviour in some model situations.
	
	The Thomas--Yau conjecture \cite{Thomas2002} proposes that any graded Lagrangian $L^n$ in a Calabi--Yau manifold $\mathcal Y^{2n}$ satisfying a stability condition flows to the unique special Lagrangian in its Hamiltonian isotopy class. The counter-example of Neves \cite{Neves2013} makes it clear that singularities can occur in general, however these constructions are not almost-calibrated (and therefore not stable). Updated versions of the conjecture were presented by Joyce in \cite{joyce_2015}. Joyce suggests working in an isomorphism class of a conjectural enlarged version of the derived Fukaya category $D^b \mathscr{F} (M)$ rather than the Hamiltonian isotopy class of $L$. In particular, the standard derived Fukaya category (as developed by Fukaya--Oh--Ohta--Ono \cite{fukaya_oh_ohta_ono_2010} and Seidel \cite{seidel_2008}) should be expanded to include immersed and singular Lagrangians. 
	
	In order to work within this category, it is necessary to work with a larger class of Lagrangian mean curvature flows than have been previously considered. A full generalisation would include flows of Lagrangian networks (see for instance \cite{mantegazza_novaga_alessandra_schulze_2018} for a 1-dimensional version of this phenomenon). In this paper, we focus on one initial direction for this generalisation, namely by specifying a boundary condition for a Lagrangian mean curvature flow $L_t$ on another Lagrangian mean curvature flow $\Sigma_t$; this corresponds to the network case where one of the angles is $\pi$.

Boundary conditions which preserve the Lagrangian condition are exceptional; standard Dirichlet and Neumann conditions do not have this property. One might be tempted to consider instead boundary conditions on a potential function, but these are not natural on a geometric level. It is well known that there exists an angle function $\theta: L \to \bb{R} / 2\pi \bb Z$ for Lagrangian submanifolds $L$ of $\mathcal Y$ with the property that the mean curvature vector is given by $H = J\nabla \theta$. If two stationary special Lagrangians intersect, then their Lagrangian angles must differ by a constant - we extend this to create a geometrically natural mixed Dirichlet--Neumann boundary condition for flowing Lagrangian submanifolds.

	Although no work has been done on Lagrangian mean curvature flow with boundary conditions (other than curve-shortening flow), an alternative boundary condition has been studied by Butscher \cite{Butscher_2003}\cite{Butscher_2004} for the related elliptic case of special Lagrangians with boundary on a codimension 2 symplectic submanifold. Boundary conditions for codimension 1 mean curvature flow have been considered in a variety of contexts, for example  by Ecker \cite{EckerMinkowskiDBC}, Priwitzer \cite{PriwitzerDirichlet} and Thorpe \cite{ThorpeDirichlet} in the Dirichlet case, by Buckland \cite{Buckland}, Edelen \cite{EdelenBrakke}\cite{Edelen}, Huisken \cite{Huiskengraph}, Lambert \cite{LambertTorus}\cite{LambertConstruction}, Lira--Wanderley \cite{LiraWanderley}, Stahl \cite{Stahlsecond}\cite{Stahlfirst} and Wheeler \cite{Wheelerhalfspace}\cite{WheelerRotSym} in the Neumann case, and by Wheeler--Wheeler \cite{WheelersDirichletNeumann} in a mixed Dirichlet Neumann case.

	Consider a family of immersed compact-with-boundary Lagrangian submanifolds $F_t:L^n \rightarrow \mathcal{Y}$, and an immersed Lagrangian mean curvature flow $\Sigma_t$ in $\mathcal{Y}$ for $t \in [0,T_\Sigma)$. Denote $L_t := F_t(L^n)$, and suppose that $\partial L_t \subset \Sigma_t$; this may be thought of as $(n-1)$-Dirichlet boundary conditions for the mean curvature flow problem on $L_t$.  
For the final boundary condition, we fix the difference between the Lagrangian angles of $\Sigma_t$ and $L_t$ on $\partial L_t$. We now have a well-posed boundary value problem:
	\begin{equation}\label{IntroBVP}
	\begin{cases}
	\left(\frac{d}{dt} F(x,t)\right)^{NL} = H(x,t)& \text{for all }(x,t) \in L^n\times[0,T)\\
	F(x,0)=F_0(x)&\text{for all }x\in L^n \\
	\partial L_t\subset \Sigma_t & \text{for all }t\in[0,T)\\
	e^{i(\tilde \theta -  \theta)}(x,t) = ie^{i\alpha} & \text{for all } (x,t) \in \partial L^n \times [0,T),
	\end{cases}
	\end{equation} 
	where $NL$ is the normal bundle of $L$, $\theta$ and $\tilde \theta$ are the Lagrangian angles of $L$ and $\Sigma$ respectively, and $\alpha \in (-\pi/2,\pi/2)$ is a constant angle. In the case where $\Sigma_t$ and $L_t$ are zero-Maslov, the final condition may be written as $\tilde \theta - \theta = \alpha + \frac{\pi}{2}$. Our main theorem concerns existence and uniqueness of solutions to \eqref{IntroBVP}, as well as preservation of the Lagrangian condition.
	
	\begin{theorem}\label{Main}
		Let $\Sigma_t$ be a smooth oriented Lagrangian mean curvature flow and suppose that $L_0$ is an oriented smooth compact Lagrangian with boundary which satisfies the boundary conditions in (\ref{IntroBVP}). Then there exists a $T\in(0,T_\Sigma]$ such that a unique solution of (\ref{IntroBVP}) exists for $t\in[0,T)$ which is smooth for $t>0$. Furthermore, if $T<\infty$, at time $T$ at least one of the following hold:
		\begin{enumerate}[label=\alph*)]
			\item \textbf{Boundary flow curvature singularity:} $\sup_{\Sigma_t} |\II^\Sigma|^2 \ra \infty$ as $t\ra T$.
			\item \textbf{Flowing curvature singularity:} $\sup_{L_t} |\II|^2 \ra \infty$ as $t\ra T$.
			\item \textbf{Boundary injectivity singularity:} The boundary injectivity radius of $\partial L_t$ in $L_t$ converges to zero as $t\ra T$.
		\end{enumerate}Furthermore $F_t(L)$ is Lagrangian for all $t \in (0,T)$.
	\end{theorem}

	\begin{remark}
		Whilst a) and b) in Theorem \ref{Main} are standard singularities, the boundary injectivity singularity is new and a result of the flowing boundary condition.
	\end{remark} 
	 
	A priori, the Lagrangian angle is not well-defined for $L_t$ for $t>0$ since the mean curvature flow does not necessarily preserve the Lagrangian condition. We therefore generalise the Neumann boundary condition in equation \eqref{IntroBVP} to a statement that holds for any $n$-dimensional manifold $M$ intersecting along an $(n-1)$-dimensional manifold, see equation (\ref{MCFBC}) in Section \ref{BoundaryCond}. In the case $M_t=L_t$ is Lagrangian, (\ref{MCFBC}) and (\ref{IntroBVP}) are equivalent.
	
	Theorem \ref{Main} is proven in two parts. Firstly, in Section \ref{PresLag}, we show that a solution to (\ref{MCFBC}) with Lagrangian initial condition remains Lagrangian. If we denote by $\omega := \overline \omega|_L$ the restriction of the ambient K\"ahler form to $M_t$, then by a careful analysis of the boundary condition we are able to apply a maximum principle to estimate the rate of increase of $|\omega|^2$ in terms of its initial value. Since the initial condition is Lagrangian, this implies that $|\omega|^2$ is identically zero. For the case of a Lagrangian $L$ without boundary, this was shown by Smoczyk in \cite{Smoczyk1996}.
	
	We postpone the proof of short-time existence  and uniqueness for (\ref{MCFBC}) to Section \ref{STEsec}, see Theorem \ref{STE}. The mixed Dirichlet--Neumann boundary conditions are not well covered in the literature and so we provide a full exposition.
	
\begin{figure}[t]
	\centering
	\begin{subfigure}[b]{0.44\linewidth}
		\includegraphics[width=\linewidth]{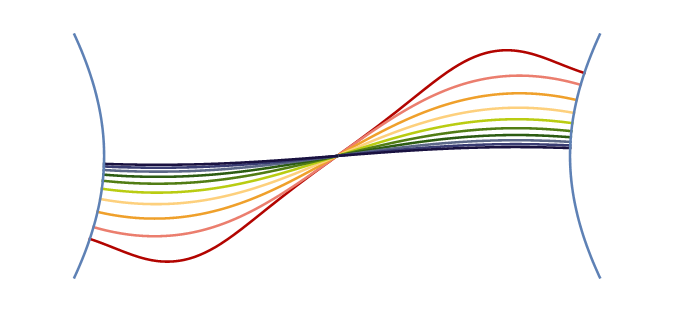}
		\caption{An example of a LMCF with boundary on the Lawlor neck, $\alpha = 0$.}
		\label{fig-lawlorcase1}
	\end{subfigure}
	\begin{subfigure}[b]{0.02\linewidth}
		$\,$
	\end{subfigure}
	\begin{subfigure}[b]{0.44\linewidth}
		\includegraphics[width=\linewidth]{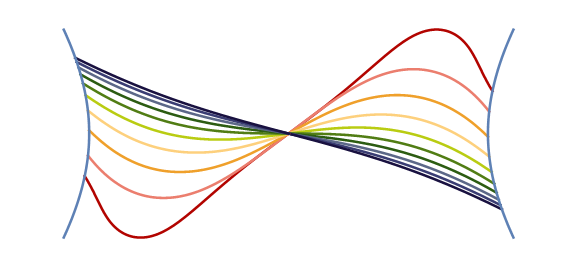}
		\caption{An example of a LMCF with boundary on the Lawlor neck, $\alpha = 0.8$.}
		\label{fig-lawlorcase2}
	\end{subfigure} \caption{}
	\label{fig-lawlor}
\end{figure}

\begin{figure}[t]
	\centering
	\begin{subfigure}[b]{0.42\linewidth}
		\includegraphics[width=\linewidth]{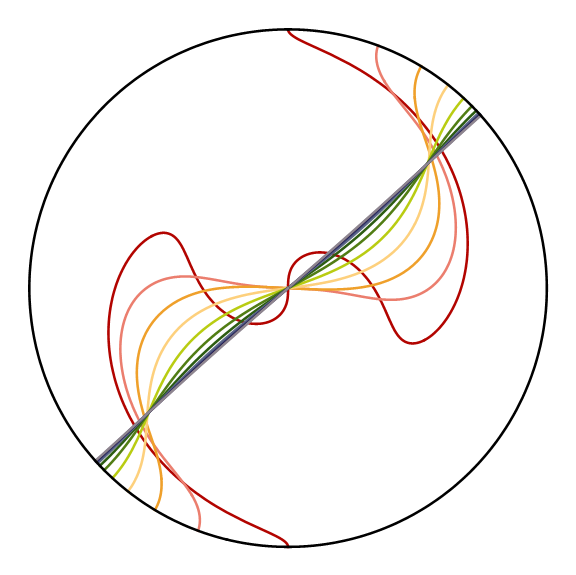}
		\caption{An example of rescaled LMCF with boundary on the Clifford torus, $\alpha = 0$.}
		\label{fig-cliffordcase1}
	\end{subfigure}
	\begin{subfigure}[b]{0.06\linewidth}
		$\,$
	\end{subfigure}
	\begin{subfigure}[b]{0.42\linewidth}
		\includegraphics[width=\linewidth]{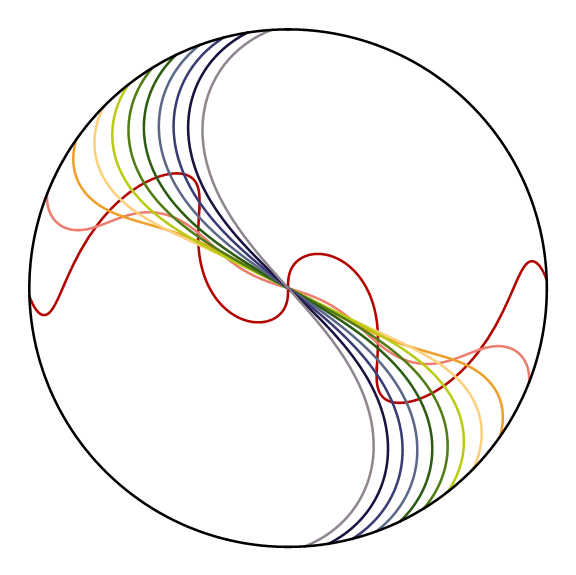}
		\caption{An example of rescaled LMCF with boundary on the Clifford torus, $\alpha = -\frac{2\pi}{5}$.}
		\label{fig-cliffordcase2}
	\end{subfigure} \caption{}\label{fig-Clifford}
\end{figure}
	
	To illustrate the behaviour of the flow, in Section \ref{Examples} we examine the particular case of $S^1$-equivariant Lagrangian submanifolds of $\mathbb{C}^2$; this assumption reduces the PDE problem (\ref{IntroBVP}) to a codimension $1$ flow of the profile curve in $\mathbb{C}$, allowing for easier analysis. Such flows have been studied for ordinary LMCF - see for example \cite{Evans2018}, \cite{Groh2007}, \cite{savas-halilaj_smoczyk_2019} and \cite{Wood}. 
	
	One natural choice of boundary manifold in this setting is the Lawlor neck $\Sigma_{\operatorname{Law}}$ (see Example \ref{example:Lawlor} and Figure \ref{fig-lawlor}).
	It is the only non-flat equivariant special (minimal) Lagrangian in $\mathbb{C}^2$, and is therefore static under the mean curvature flow; this makes it a good choice of boundary manifold for our flow. We prove that any solution to \eqref{IntroBVP} satisfying the almost-calibrated condition (defined in Section \ref{sec-prelim}) with boundary on the static Lawlor neck exists for all time and converges smoothly to a special Lagrangian. A similar result for the boundaryless case was proven in \cite{Wood}, in which it was shown that equivariant Lagrangian planes flowing by mean curvature satisfying the almost-calibrated condition do not form finite-time singularities.
	
	\begin{theorem}\label{LawlorThm}
		Let $F_0$ be an almost-calibrated $S^1$-equivariant Lagrangian embedding of the disc $D^2$ into $\mathbb{C}^2$  %satisfying
with boundary on the static Lawlor neck, $\Sigma_{\text{Law}}$, 
		 such that the Lagragian angle of $L_0$, $\theta_0$, satisfies $\theta_0|_{\p L_0}=-\alpha$. Then there exists a unique, immortal solution to the LMCF problem (\ref{IntroBVP}), and it converges smoothly in infinite time to a special Lagrangian disc.
	\end{theorem}

Another natural choice of boundary manifold is the Clifford torus (see Example \ref{example:Clifford} and Figure \ref{fig-Clifford}). 
The symmetry of the Clifford torus is preserved under mean curvature flow, so it is a self-shrinking solution, and is static under the rescaled flow (defined in Section \ref{sec-cliff}). Here, the condition $\theta - 2\arg(\gamma) \in (-\tfrac{\pi}{2} + \veps, \, \tfrac{\pi}{2} - \veps)$ is a natural preserved condition to consider in place of the almost-calibrated condition, as $\theta - 2\arg(\gamma)$ always vanishes on the boundary. Given this condition, we show a long-time existence and convergence result for the rescaled flow in the $\alpha = 0$ case, as depicted in Figure \ref{fig-cliffordcase1}.
	
\begin{theorem}\label{CliffThm}
		Let $\overline F_0: D \rightarrow \mathbb{C}$ be an $S^1$-equivariant Lagrangian embedding of a disc $D$, with boundary on the \emph{Clifford torus}, $\Sigma_{\text{Cliff}}$.
		Assume that its Lagragian angle $\theta_0$ satisfies
		\[ \theta_0(s) - 2\arg(\overline\gamma_0(s)) \, \in \, (-\tfrac{\pi}{2} + \veps, \, \tfrac{\pi}{2} - \veps)\]
		for some $\veps > 0$, and that $\theta_0 - 2\arg(\gamma_0) = 0$ on $\p L_0$.
		Then there exists a unique, eternal solution to the rescaled LMCF problem \eqref{thm-cliff} (corresponding to (\ref{IntroBVP}) with $\alpha=0$), which converges smoothly in infinite time to a special Lagrangian disc.
	\end{theorem}

In the case of the Clifford torus, numerical evidence suggests that a rescaled solution of $\eqref{IntroBVP}$ with $\alpha \neq 0$ exists for all time and converges to a unique rotating soliton - see Figure \ref{fig-cliffordcase2}.

	\begin{ack}
		The authors would like to thank Jason Lotay and Felix Schulze for many useful discussions, Oliver Schn\"urer for useful comments and Dominic Joyce for suggesting this line of research.  	
		
		The first two authors were supported by Leverhulme Trust Research Project Grant RPG-2016-174. The final author was supported by the Engineering and Physical Sciences Research Council [EP/L015234/1], The EPSRC Centre for Doctoral Training in Geometry and Number Theory (The London School of Geometry and Number Theory), University College London.
	\end{ack}

\section{Preliminaries}\label{sec-prelim}

A K\"ahler manifold $(\mathcal Y^{2n}, \bar g, \bar \omega, J)$ is said to be a \textit{Calabi--Yau} manifold if it is Ricci-flat. On such a manifold, there exists an everywhere non-zero holomorphic $n$-form $\Upsilon$ on $\mathcal Y$ such that $\operatorname{Re}(\Upsilon)$ is a calibration. 

An $n$-dimensional submanifold $F:L^n \to \mathcal Y$ is then called \textit{Lagrangian} if $\omega := F^* \bar \omega =0$. It is well-known that
\[\Upsilon|_L = e^{i \theta} \operatorname{vol}_L,\]
for some multi-valued function $\theta: L \to \mathbb R / 2 \pi \mathbb Z$ called the Lagrangian angle. Lagrangian submanifolds have the additional property that the almost-complex structure $J$ is an isometry between the tangent and normal bundles of $L$, and this isomorphism leads to the remarkable fact that the mean curvature $H$ of $L$ is described by the Lagrangian angle: 
\begin{equation}
\label{H = dtheta} H = J \nabla \theta.
\end{equation} 
If $\theta$ is constant, then $L$ is minimal since it is calibrated by $\operatorname{Re}(e^{i\theta}  \Upsilon)$. Such minimal Lagrangians are known as \emph{special Lagrangians}. Furthermore, (\ref{H = dtheta}) implies that deforming a Lagrangian in the direction of its mean curvature is a Hamiltonian deformation, and raises the possibility that mean curvature flow preserves the Lagrangian condition. In \cite{Smoczyk1996}, Smoczyk applied the parabolic maximum principle to $|\omega|^2$, concluding that if $L_t$ is a mean curvature flow with $L_0$ a closed Lagrangian submanifold, then $L_t$ is Lagrangian for all time.

If $\theta$ is a single-valued function on $L$ then $L$ is called \textit{zero-Maslov}, and if furthermore the condition
\[ \cos(\theta) > \veps > 0 \]
holds, it is called \emph{almost-calibrated}. Since under the mean curvature flow, $\theta$ satisfies the heat equation
\[\frac{d}{dt} \theta =  \Delta \theta,\]
locally, this implies that both almost-calibrated and zero-Maslov are preserved classes under mean curvature flow (without boundary).

A particular class of Lagrangian submanifolds which we shall investigate further in Section \ref{Examples} is that of \emph{equivariant} Lagrangians in $\mathbb{C}^2$. If we consider $\mathcal Y = \mathbb C^2$ with the standard K\"ahler structure, then $\mathcal Y$ is Calabi--Yau with $\Upsilon = dz_1 \wedge \cdots \wedge dz_n$. A Lagrangian $L \subset \mathbb{C}^2$ is said to be \emph{equivariant} if there exists a profile curve on a one-dimensional manifold $U$,
\[ \gamma(s) \, := \, \left( x(s),y(s) \right) \in \mathbb{C}, \]
such that the Lagrangian can be parametrised as
\begin{align*}
L &: U \times S^1 \rightarrow \mathbb{C}^2\\
L(s,\psi) \, &= \, \left( x(s)  +i y(s) \right)(\cos(\psi), \sin(\psi)) \in \mathbb{C}^2.
\end{align*}
In fact, if the submanifold can be parametrised in this way, then it must be a Lagrangian submanifold. Mean curvature flow of equivariant submanifolds is particularly nice as it can be reduced to the study of the equivariant flow of the profile curve $\gamma$, given by
\begin{equation}
\frac{\p \gamma}{\p t} \, = \, k - \frac{\gamma^\perp}{|\gamma|^2}, \label{eq-equi}
\end{equation}
where $k$ is the curvature vector of the profile curve. Note that the profile curve is symmetric across the origin by the equivariance. Two important examples of equivariant Lagrangians are the following:
\begin{example}\label{example:Lawlor} The Lawlor neck, $\Sigma_{\operatorname{Law}}\subset\bb{C}^2$, is an equivariant special Lagrangian, whose profile curve is a hyperbola,
\[ \sigma_{\operatorname{Law}}: \mathbb{R} \ra \mathbb{C}, \quad \quad \sigma_{\operatorname{Law}}(s) \, := \, (\cosh(s), \sinh(s)). \]
We note that in our definition, the Lawlor neck has constant Lagrangian angle equal to $\frac{\pi}2$.
\end{example}

\begin{example}\label{example:Clifford} The Clifford torus, $\Sigma_{\operatorname{Cliff}}\subset\bb{C}^2$, is an equivariant surface whose profile curve is a circle or radius 2,
\[ \sigma_{\operatorname{Cliff}}: S^1 \ra \mathbb{C}, \quad \quad \sigma_{\operatorname{Cliff}}(s) \, := \, (2\cos(s), 2\sin(s)). \]
A short calculation indicates that the Clifford torus satisfies the mean curvature flow self shrinker equation.
\end{example}
The Lagrangian angle is particularly simple for equivariant Lagrangians $L$ away from the origin:
\begin{equation}\label{EqLangle}\theta = (n-1)\arg \gamma + \arg \gamma',
\end{equation}
note it does not depend on the spherical parameter $\alpha$ but only the parameter along the profile curve.

\subsection{Notation and Standard Facts} We employ the following notational conventions throughout this paper. $M_t$ will always be a mean curvature flow with boundary on a Lagrangian mean curvature flow $\Sigma_t$, all in a Calabi--Yau manifold $\CY$. We shall write $L_t = M_t$ only when we have proven the Lagrangian condition is preserved. We shall frequently suppress the subscript $t$ when the meaning is clear. We distinguish between quantities on each by diacritical marks: for instance, the ambient connection on $\CY$ is $\overline \nabla$, the induced connection on $M$ or $L$ is $\nabla$, and the induced connection on $\Sigma$ is $\widetilde \nabla$. We extend this convention in the natural way to other quantities such as the second fundamental form and the mean curvature. For any submanifold $Z \in \CY$, $p\in Z$ and a general vector $V\in T_p \CY$ we will denote orthogonal projection of $V$ onto the tangent space and normal space of $Z$ by $V^{TZ}$ and $V^{NZ}$ respectively. Finally, throughout we will use the Einstein summation convention, where we assume that lower case Roman letters sum $1\leq i,j,k,\ldots\leq n$ and upper case Roman letters sum $1\leq I,K,L, \ldots\leq n-1$.

We also include here for convenience a few basic definitions from differential geometry. Given tangent vector fields $X$ and $Y$ on $M$ we define the second fundamental form of $M$ by
\[\II(X,Y) = \left(\ov \n_X Y\right)^{NM}\ .\]
We note that since $\Sigma_t$ is Lagrangian as above we have that 
\begin{equation}\label{eq-2ffswap}
\ip{\IIS(X,Y)}{JZ} = \ip{\IIS(X,Z)}{JY},
\end{equation}
where $X,Y,Z \in T \Sigma_t$. 

Let $\mu$ be the outward pointing unit vector to $\partial M$. For $p\in\partial M$ let $\gamma_p(s)$ be the unit speed geodesic starting at $p\in\partial M$ with tangent vector $-\mu(p)$. We define the \textit{boundary injectivity radius} to be
\[\text{inj}_{\partial M} = \frac 1 2\min\left\{\l > 0 \,\,\big|\,\, \exists p \in\partial M \text{ such that } \gamma_p((0,\l))\subset M,  \text{ but }  \gamma(\l)\subset \partial M\right\}\ .  \]
If $M$ is compact then $\text{inj}_{\partial M}>0$ and in this case $\text{inj}_{\partial M}$ coincides with the maximal collar region such that the distance to the boundary function is smooth.

\section{The Boundary Condition}\label{BoundaryCond}

Let $\Sigma_t^n$, $t\in [0,T)$ be a Lagrangian mean curvature flow in $\CY^{2n}$. In this section, we generalise (\ref{IntroBVP}) to a boundary problem that holds for any $M_t^n$, not necessarily Lagrangian, with $\partial M_t \subset \Sigma_t$.

Suppose that $M$ satisfies the Dirichlet boundary condition above. This implies that at any point $p\in \partial M$, there exists tangent vectors $e_1, \ldots, e_{n-1}$ of $T_p\partial M$, $\mu\in T_pM$ and $\nu \in T_p\Sigma$ so that $\{e_1, \ldots, e_{n-1}, \mu\}$ is an orthonormal basis of $T_pM$ and $\{e_1, \ldots, e_{n-1}, \nu\}$ is an orthonormal basis of $T_p \Sigma$.

Since $\Sigma$ is Lagrangian,  $\mu$ is of the form
\begin{equation}\label{mu eqn}\mu = \tau+\ip{\nu}{\mu}\nu+\ip{J\nu}{\mu}J\nu,\end{equation}
where $\tau=\tau^IJe_I\in \text{span}\{Je_1, \ldots, Je_{n-1}\}$, and this yields that the Calabi--Yau form $\Upsilon$ relative to $T_p \Sigma$ restricted to $T_p M$ is
\[\Upsilon|_{T_p M} = \Upsilon(e_1, \ldots,e_{n-1}, \mu) = \det \begin{pmatrix}
I&i\tau^I\\
0&\ip{\nu}{\mu}+i\ip{J\nu}{\mu}
\end{pmatrix} = \ip{\nu}{\mu}+i\ip{J\nu}{\mu},
\]
where we note that this complex number has modulus 1 if and only if the tangent space of $M$ is Lagrangian at $p$. We extend the boundary condition in (\ref{IntroBVP}) by simply assuming that the argument of this complex number is constant, that is we impose that there exists a constant $\alpha\in(-\frac \pi 2,\frac \pi 2)$ so that
\[\ip{\nu}{\mu} = \tan \alpha \ip{J\nu}{\mu}\ .\]
If both $\Sigma_t$ and $M_t$ are Lagrangian manifolds this corresponds to a phase difference of $ie^{i\alpha}$ or $ie^{i\alpha}=e^{i(\tilde \theta- \theta)}$.

\begin{remark}
	Although it is beyond the scope of this paper, we believe that an analogous boundary condition could be defined in the non-Ricci-flat setting since we have only used the existence of a \textit{relative} Calabi--Yau form. Hence the results of this paper should be applicable with some modification to Lagrangian mean curvature flows in general K\"ahler--Einstein manifolds. 
\end{remark}

Let $F:M^n\times[0,T)\ra\CY$ be a one parameter family of immersions, and write $M_t = F(M, t)$. We define a reparametrised mean curvature flow as follows:
\begin{equation}
\label{MCFBC}
\begin{cases}
\left(\ddt{} F(x,t)\right)^{NM} = H(x,t)& \text{for all }(x,t) \in M\times[0,T)\\
F(x,0)=F_0(x)&\text{for all }x\in M\\
\partial M_t\subset \Sigma_t & \text{for all }t\in[0,T)\\
\cos{\a}\ip{\nu}{\mu}-\sin\a \ip{J\nu}{\mu}=0 &\text{for all } (x,t)\in \partial M \times[0,T)
\end{cases}
\end{equation}
Note that (\ref{MCFBC}) is exactly (\ref{IntroBVP}) when $F_t(M)$ is Lagrangian.

\subsection{Linear Algebra}\label{linalg}
From now on, we assume that $M$ satisfies the boundary conditions in (\ref{MCFBC}). Following the notation in Section \ref{BoundaryCond}, we recall that at a boundary point we have
\[T_p\Sigma = \text{span}\{e_1, \ldots, e_{n-1}, \nu\}\ ,\]
and, as this tangent space is Lagrangian, 
\[N_p\Sigma = \text{span}\{Je_1, \ldots, Je_{n-1}, J\nu\}\ .\]
We recall that
\[T_p M=\text{span}\{e_1, \ldots, e_{n-1}, \ip{\nu}{\mu} \nu +\ip{J\nu}{\mu}J\nu+\tau\}\]
where $\tau\in JT_p\partial M$. We note that
\[N_pM =  \text{span}\{f_1, \ldots, f_n\} \ ,%\text{span}\{Je_1 - \ip{Je_1}{\mu}\mu,\ldots,Je_{n-1} - \ip{Je_{n-1}}{\mu}\mu, -\mu^{J\mu}\nu+\mu^\nu J\nu\}
\]
where for $1\leq I\leq n-1$
\[f_I = Je_I - \ip{Je_I}{\mu}\mu\ , \quad \quad f_n=-\ip{J\nu}{\mu}\nu+\ip{\nu}{\mu} J\nu\ ;\]
this is no longer an orthonormal basis. This yields an a inner product matrix 
\[
G_{ij}=\ip{f_i}{f_j} = 
 \begin{pmatrix}
  \delta_{IJ} - \tau^I\tau^J & 0\\
  0& 1-|\tau|^2
 \end{pmatrix}
\]
where we write $\tau^I = \ip{\tau}{Je_I} = \ip{\mu}{Je_I}$ . This has inverse
\[
G^{ij}=
 \begin{pmatrix}
  \delta_{IJ} + \frac{\tau^I\tau^J}{1-|\tau|^2} & 0\\
  0&\frac{1}{1-|\tau|^2}
 \end{pmatrix}\ .
\]
We may write 
 \begin{align*}
  \mu&= \mu^{N\Sigma} +\ip{\nu}{\mu}\nu, \qquad
  \nu= \nu^{NM} +\ip{\nu}{\mu}\mu\ .
 \end{align*}
Substituting back into the last terms and rearranging yields
 \begin{align}
  \mu&= \frac{1}{1-\ip{\nu}{\mu}^2}\left[\mu^{N\Sigma} +\ip{\nu}{\mu}\nu^{NM}\right]\label{muinperps}\\
  \nu&= \frac{1}{1-\ip{\nu}{\mu}^2}\left[\nu^{NM} +\ip{\nu}{\mu}\mu^{N\Sigma}\right]\ .\label{nuinperps}
 \end{align}
We have that $\tau=\tau^I Je_I\in N_p\Sigma$. We have that 
\begin{equation}\tau^{NM} = \ip{\tau}{f_i}G^{ij}f_j = \ip{\tau}{f_I}G^{IJ}f_J = (1-|\tau|^2)\tau_IG^{IJ}f_J=\tau^Jf_J = \tau - |\tau|^2\mu
\label{tildemu}
\end{equation}
\begin{equation}
\nu^{NM} = \ip{\nu}{f_i}G^{ij}f_j = -\ip{\nu}{\mu}\tau_IG^{IJ}f_J-\frac{\ip{J\nu}{\mu}}{1-|\tau|^2}f_n=-\frac{\ip{\nu}{\mu}}{1-|\tau|^2}\tau^{NM}-\frac{\ip{J\nu}{\mu}}{1-|\tau|^2}f_n\label{nuasf}
\end{equation}
In the following we will assume that the vectors $e_1, \ldots, e_{n-1}$, $\mu$ and $\nu$ are extended locally to a neighbourhood in $U\subset\partial M_t$ of $p$  so that at every $q\in U$, $\{e_1,\ldots,e_{n-1},\mu\}$ is an orthonormal basis of $T_qM$ and $\{e_1,\ldots,e_{n-1},\nu\}$ is an orthonormal basis of $T_q\Sigma$.

\subsection{Derivatives of the Boundary Conditions}

In this section, we provide identities that arise by differentiating the boundary conditions.

\subsubsection{Dirichlet boundary space derivatives}
We now use the Dirichlet condition to compare first order boundary derivatives.
\begin{lemma}
\label{Dirichletrewrite2}
Suppose that $\Sigma$ is Lagrangian, and $M$ is a $n$-dimensional submanifold with boundary $\partial M\subset \Sigma$. At a point $p\in\partial M$, we have that for any $X,Y \in T_p \partial M$,
\begin{equation*}
\frac{\ip{J\nu}{\mu}^2}{1-\ip{\nu}{\mu}^2}\ip{\IIS_{XY}}{\tau}=\ip{\IIM_{XY}}{\tau} 
 + \frac{|\tau|^2}{1-\ip{\nu}{\mu}^2}\left[\ip{J\nu}{\mu}\ip{\IIS_{XY}}{J\nu } +\ip{\nu}{\mu}\ip{\IIM_{XY}}{\nu}\right]\ .
\end{equation*}
\end{lemma}
\begin{proof}
We may write $\ov\n_X Y$ in two ways, namely
\begin{flalign*}
\ov\n_X Y &=\ip{\IIS_{XY}}{Je^I}Je_I + \ip{\IIS_{XY}}{J\nu}J\nu +\ip{\ov\n_X Y}{e^I}e_I+\ip{\ov\n_X Y}{\nu}\nu\\
&=\ip{\IIM_{XY}}{f_i}G^{ik}f_k +\ip{\ov\n_X Y}{e^I}e_I+\ip{\ov\n_X Y}{\mu}\mu
\end{flalign*}
where the $f_i$ are the basis of $N_pM$ as above. Taking an inner product with $Je_I$, this equality yields
\begin{equation}\ip{\IIS_{XY}}{Je_I} = \ip{\IIM_{XY}}{f_I}+\ip{\ov\n_XY}{\mu}\tau^I\ . 
\label{Dfirst}
\end{equation}	
Due to equation (\ref{muinperps}), 
\begin{flalign}
\ip{\ov\n_XY}{\mu} &= \frac{1}{1-\ip{\nu}{\mu}^2}\left[\ip{\ov\n_XY}{\mu^{N\Sigma} +\ip{\nu}{\mu}\nu^{NM}}\right]\nonumber \\
&=\frac{1}{1-\ip{\nu}{\mu}^2}\left[\ip{\IIS_{XY}}{\mu}+\ip{\nu}{\mu}\ip{\IIM_{XY}}{\nu}\right] \ .\label{partialLcurv}
\end{flalign}
Equation (\ref{Dfirst}) now yields
\begin{flalign*}
\ip{\IIS_{XY}}{Je_I}-\frac{\tau^I}{1-\ip{\nu}{\mu}^2}\ip{\IIS_{XY}}{\mu} = \ip{\IIM_{XY}}{f_I}+\frac{\ip{\nu}{\mu}\tau^I}{1-\ip{\nu}{\mu}^2}\ip{\IIM_{XY}}{\nu}\ .
\end{flalign*}	
Multiplying  by $\tau^I$ and summing, we have that (using (\ref{tildemu}))
\[\ip{\IIS_{XY}}{\tau} - \frac{|\tau|^2}{1-\ip{\nu}{\mu}^2}\ip{\IIS_{XY}}{\mu^{N\Sigma}}  = \ip{\IIM_{XY}}{\tau} + \frac{|\tau|^2\ip{\nu}{\mu}}{1-\ip{\nu}{\mu}^2}\ip{\IIM_{XY}}{\nu}.\]
By (\ref{mu eqn}), we have that 
\begin{equation}\label{mod mu} 1-\ip{\nu}{\mu}^2 - |\tau|^2 -\ip{J\nu}{\mu}^2=0\end{equation}
and hence
\begin{flalign*}
\tau - \frac{|\tau|^2}{1-\ip{\nu}{\mu}^2}(\mu - \ip{\nu}{\mu}\nu)&=\frac{\ip{J\nu}{\mu}^2}{1-\ip{\nu}{\mu}^2}\tau- \frac{\ip{J\nu}{\mu}|\tau|^2}{1-\ip{\nu}{\mu}^2}J\nu.
\end{flalign*}
Thus we conclude
\begin{align*}
\frac{\ip{J\nu}{\mu}^2}{1-\ip{\nu}{\mu}^2}\ip{\IIS_{XY}}{\tau}=\ip{\IIM_{XY}}{\tau} + \frac{|\tau|^2}{1-\ip{\nu}{\mu}^2}\left[\ip{J\nu}{\mu}\ip{\IIS_{XY}}{J\nu } +\ip{\nu}{\mu}\ip{\IIM_{XY}}{\nu}\right]\ .
\end{align*}
\end{proof}

\subsubsection{Dirichlet boundary time derivatives}
We now consider time derivatives: 
\begin{lemma}\label{Dirichlettime}
Let $\Sigma_t$ be a smooth solution of LMCF and $M_t$ satisfies (\ref{MCFBC}). Suppose that $\partial M_t\subset \Sigma_t$ for all $t\geq 0$, then for all $t>0$, 
\[\ip{H- \widetilde H}{\tau}\ip{J\nu}{\mu} = \ip{H-\widetilde H}{J\nu}|\tau|^2 \ .\]
\end{lemma}
\begin{proof}
We consider a point $p(t)=F(p^1(t),\ldots,p^n(t),t)$ such that $p$ stays in $\Sigma_t$ (such a point exists by assumption). Then we must have that
\[\widetilde H=\left(\ddt{p}\right)^{N\Sigma} %= \left(P+\ddt{X}\right)^{N_p\Sigma}
= \left(P+H\right)^{N_p\Sigma}\]
where $P=\pard{p^i}{t}X_i$ is a tangent vector to $M$. Fixing $t$ and writing $P=P^Ie_I +P^\mu \mu$ we see that 
\begin{align*} 
\ip{\widetilde H}{Je_I}&=\tau^IP^\mu +\ip{Je^I}{H}, \qquad
\ip{\widetilde H}{J\nu}=\ip{J\nu}{\mu} P^\mu+\ip{J\nu}{H}.
\end{align*}
This is equivalent to the statement that 
\[H^{N\Sigma}-\widetilde H = -P^\mu\left[\tau + \ip{J\nu}{\mu}J\nu \right]\ .\]
We also see that 
\[\ip{H-\widetilde H}{\tau}=-P^\mu|\tau|^2, \qquad \ip{H-\widetilde H}{J\nu} = -P^\mu \ip{J\nu}{\mu}\]
which yields the claim.
\end{proof}

\subsubsection{Neumann boundary condition space derivatives}
We will see that at a point $p\in \partial M$ such that the Neumann boundary condition holds and $\frac{1}{2}>|\omega|^2(p) = \underset{q\in\partial M}\max |\omega|^2(q)$ we have that
\[\n_I \ip{\nu}{\mu} =0 = \n_I\ip{J\nu}{\mu}\ .\]
We will now investigate the implications of these equalities.

\begin{lemma}\label{Neumannspace}
 Suppose that at some $p\in\partial M$
 \[\n_I \ip{\nu}{\mu} = 0 .\]
 Then
 \[\ip{\nu}{\IIM_{I\mu}}+\ip{\mu}{\IIS_{I\nu}} =0  .\]
\end{lemma}
\begin{proof}
 Using $\n_I\ip{\nu}{\mu}=0$, we have
 \[0=\ip{\ov\n_I\nu }{\mu}+\ip{\ov\n_I\mu }{\nu}\]
 and so using equations (\ref{muinperps}) and (\ref{nuinperps})
 \begin{flalign*}
 0&=\ip{\ov\n_I \nu^{NM} +\ip{\nu}{\mu}\ov\n_I\mu^{N\Sigma} }{\mu}+\ip{\ov\n_I\mu^{N\Sigma} +\ip{\nu}{\mu}\ov\n_I\nu^{NM} }{\nu}\\
 &=-\ip{\nu}{\IIM_{I\mu}}-\ip{\mu}{\IIS_{I\nu}}+\ip{\nu}{\mu}\left[\ip{\ov\n_I\mu^{N\Sigma} }{\mu}+\ip{\ov\n_I\nu^{NM} }{\nu}\right]\\
  &=-\ip{\nu}{\IIM_{I\mu}}-\ip{\mu}{\IIS_{I\nu}}\\&\qquad+\frac{\ip{\nu}{\mu}}{1-\ip{\nu}{\mu}^2}\left[\ip{\ov\n_I\mu^{N\Sigma} }{\mu^{N\Sigma} +\ip{\nu}{\mu}\nu^{NM}}+\ip{\ov\n_I\nu^{NM} }{\nu^{NM} +\ip{\nu}{\mu}\mu^{N\Sigma}}\right]\\
  &=-\ip{\nu}{\IIM_{I\mu}}-\ip{\mu}{\IIS_{I\nu}}\\&\qquad+\frac{\ip{\nu}{\mu}}{1-\ip{\nu}{\mu}^2}\left[\frac{1}{2}(\n_I\left|\mu^{N\Sigma}\right|^2+\n_I\left|\nu^{NM}\right|^2) +\ip{\nu}{\mu}\left(\n_I \ip{\mu^{N\Sigma}}{\nu^{NM}}\right)\right]
 \end{flalign*}
However, we see that
\[\left|\mu^{N\Sigma}\right|^2 = 1-\ip{\nu}{\mu}^2  = \left|\nu^{NM}\right|^2\]
and
\[\ip{\mu^{N\Sigma}}{\nu^{NM}}=\ip{\mu - \ip{\nu}{\mu}\nu}{\nu - \ip{\nu}{\mu}\mu}=\ip{\nu}{\mu}^2-\ip{\nu}{\mu}\]
and so the square bracket vanishes.
\end{proof}

\begin{lemma}\label{boundarymax}
 Suppose that at $p\in\partial M$ we have that
 \[0=\n_I\ip{\nu}{\mu}=\n_I \ip{J\nu}{\mu} .\]
 Then 
 \begin{flalign*}0&=\ip{\IIM_{I\mu}}{J\nu}-\ip{\IIS_{I\nu}}{J\mu}\\
  &\qquad\qquad+\frac{1}{1-\ip{\nu}{\mu}^2}\left[\ip{\IIM_{I\sigma}}{\nu}+\ip{\nu}{\mu}\ip{\IIS_{I\sigma}}{\mu}\right]\ ,
  \end{flalign*}
  where we define $\sigma := J\tau$ to simplify notation.
\end{lemma}
\begin{proof}
We expand the statement $\n_I \ip{J\nu}{\mu}=0$. We first note that
 \begin{flalign*}
  \ip{ \ov \n_I \mu}{J\nu}&=\ip{ \ov \n_I \mu}{(J\nu)^{NM}+(J\nu)^{TM}}\\
  &=\ip{\IIM_{I\mu}}{J\nu}+\ip{J\nu}{\mu}\ip{\ov\n_I \mu}{\mu}\\
  &=\ip{\IIM_{I\mu}}{J\nu}
 \end{flalign*}
 as $|\mu|^2=1$. We also calculate that
 \begin{flalign*}
  \ip{\ov\n_I \nu}{J\mu}&=\ip{\ov\n_I \nu}{(J\mu)^{N\Sigma}+(J\mu)^{T\Sigma}}\\
  &=\ip{\IIS_{I \nu}}{J\mu}+\ip{\ov \n_I \nu}{\sigma - \ip{J\nu}{\mu}\nu}\\
  &=\ip{\IIS_{I\nu}}{J\mu}+\ip{\ov \n_I \nu}{\sigma}\\
  &=\ip{\IIS_{I \nu}}{J\mu}+\frac{1}{1-\ip{\nu}{\mu}^2}\ip{\ov \n_I \left(\nu^{NM} +\ip{\nu}{\mu}\mu^{N\Sigma}\right)}{\sigma}\\
  &=\ip{\IIS_{I\nu}}{J\mu}-\frac{1}{1-\ip{\nu}{\mu}^2}\left[\ip{\IIM_{I\sigma}}{\nu}+\ip{\nu}{\mu}\ip{\IIS_{I\sigma}}{\mu}\right]
 \end{flalign*}
 Putting these together we have that
 \begin{flalign*}
  \n_I \ip{J\nu}{\mu}&=\ip{J\nu}{\ov\n_I\mu} - \ip{\ov\n_I\nu}{J\mu}\\
  &=\ip{\IIM_{I\mu}}{J\nu}-\ip{\IIS_{I \nu}}{J\mu} +\frac{1}{1-\ip{\nu}{\mu}^2}\left[\ip{\IIM_{I\sigma}}{\nu}+\ip{\nu}{\mu}\ip{\IIS_{I\sigma}}{\mu}\right]\ .
 \end{flalign*}
\end{proof}

\section{Preservation of the Lagrangian Condition}\label{PresLag}

In this section, we prove the Lagrangian condition is preserved assuming existence of the flow (see Section \ref{STEsec}).

\begin{theorem}\label{LMCFBdry}
	Let $\Sigma_t$ be a smooth Lagrangian mean curvature flow. Suppose $M_t$ is a solution of (\ref{MCFBC}) with $M_0$ Lagrangian and $\operatorname{inj}(\p M_t) > \delta > 0$, for $ t \in [0,T)$. Then $M_t$ is Lagrangian for all $t \in [0, T)$. 
\end{theorem}

In preparation for this proof, we calculate some important quantities using the coordinate system introduced in Section \ref{BoundaryCond}. Using the Neumann boundary condition of \eqref{MCFBC}, 
\begin{equation}
    \cos \a \ip{\nu}{\mu} -\sin \a \ip{J\nu}{\mu}=0\ \label{eq-bdry},
\end{equation}
it follows from \eqref{mu eqn} that we may write $\mu$ as
\begin{flalign}
\mu &=\frac{\ip{J\nu}{\mu}}{\cos \a}\left(\sin \a  \nu +\cos \a J\nu\right) + \tau, \label{eq-mutrig}
\end{flalign}
and from \eqref{mod mu} that we may write $|\tau|^2$ as
\begin{flalign*}
|\tau|^2 &= 1-\frac{\ip{J\nu}{\mu}^2}{\cos^2\a} =1-\frac{\ip{\nu}{\mu}^2}{\sin^2\a}.
\end{flalign*}

Let $\omega$ be the restriction of $\overline \omega$ to $M$. We wish to consider $|\omega|^2=\omega_{ij} \omega^{ij}$ where $\omega_{ij}=\ip{X_i}{JX_j}$. Calculating on the boundary in the basis $\{e_1,\dots,e_{n-1},\mu\}$ of Section \ref{linalg} we have that 
\[\omega = \begin{pmatrix}
                 0&\begin{matrix}
                    \tau^1\\
                    \vdots\\
                    \tau^{n-1}
                   \end{matrix}\\
                   -\tau^1 \ldots -\tau^{n-1} & 0
                \end{pmatrix}
\]
and so at the boundary
\begin{equation} \label{eq-omegatau}
|\omega|^2 = 2|\tau|^2\ = 2 - \frac{2\ip{J\nu}{\mu}^2}{\cos^2(\alpha)}.
\end{equation}
As a result if $|\omega|^2<\frac{1}{2}$ at a boundary point then
\begin{equation}\label{RandomIneq}\frac{\ip{J\nu}{\mu}^2}{\cos^2\a}>\frac 3 4 \ ,\end{equation}
and so at such a point, since $\nu^{NM} = \nu - \ip{\nu}{\mu}\mu$,
\[\left|\nu^{NM}\right|^2 =\left|\mu^{N\Sigma}\right|^2 =  1-\ip{\nu}{\mu}^2 = |\tau|^2 +\ip{J\nu}{\mu}^2>\frac 3 4\cos^2\a>0\ .\]
Finally,
\[\n_k \omega_{ij} = \ip{\IIM_{ik}}{JX_j} - \ip{\IIM_{jk}}{JX_i},\]
and so, remembering $\sigma = J\tau$,
\begin{flalign*}
\n_\mu |\omega|^2 &= 2\left[\ip{\IIM_{i\mu}}{JX_j} - \ip{\IIM_{j\mu}}{JX_i}\right]\omega^{ij} \\
&=2\ip{\IIM_{I\mu}}{J\mu}\ip{e_I}{J\mu}+2\ip{\IIM_{\mu\mu}}{Je_I}\ip{\mu}{Je_I} -  2\ip{\IIM_{I\mu}}{J\mu}\ip{\mu}{Je_I} - 2\ip{\IIM_{\mu\mu}}{Je_I}\ip{e_I}{J\mu}\\
&=4\ip{\IIM_{I\mu}}{J\mu} \ip{e_I}{J\mu}+4\ip{\IIM_{\mu\mu}}{Je_I}\ip{\mu}{Je_I}\\
&=4\ip{\IIM_{\sigma\mu}}{J\mu} +4\ip{\IIM_{\mu\mu}}{\tau}\ .
\end{flalign*}

We now prove the key estimate to prove Theorem \ref{LMCFBdry}.
\begin{lemma}\label{boundaryomega}
Let $p$ be a boundary maximum of $|\omega|^2$ where $|\omega|<\frac 1 2$ and suppose that $\Sigma$ satisfies LMCF. Then we have that
\begin{flalign*}
\n_\mu |\omega|^2&=  2|\omega|^2 \bigg [ -\tan^2\a +\frac{1-\ip{\nu}{\mu}^2}{\cos^2\a\ip{J\nu}{\mu}}\ip{H-\widetilde{H}}{J\nu}+\frac{1}{\cos^2\a}\ip{\widetilde{H}}{\tau}\\
&\qquad\qquad+ \frac{1}{\cos^2\a}\left[\ip{J\nu}{\mu}\ip{\IIS^I_{I}}{J\nu } 
 +\ip{\nu}{\mu}\ip{\IIM^I_{I}}{\nu}\right]\\
 &\qquad\qquad- \, \frac{\tan \a \ip{J\nu}{\mu}}{|\sigma|^2\left(1-\ip{\nu}{\mu}^2\right)}\left[\ip{\IIM_{\sigma \sigma}}{\nu}+\ip{\nu}{\mu}\ip{\IIS_{\sigma \sigma}}{\mu}\right]- \, \frac{\ip{J\nu}{\mu}^2}{|\sigma|^2} \ip{\IIS_{\sigma \sigma}}{J\nu} \bigg ].
\end{flalign*}
and in particular, if $|\IIM|_p|<C_M$, $|\IIS|_p|<C_\Sigma$ then there exists a constant $C=C(n, \alpha)$ so that 
\[\n_\mu|\omega|^2 = C(C_M+C_\Sigma)|\omega|^2\ .\]
\end{lemma}

\begin{proof}

We first prove that
\begin{equation}
    0 \, = \, \n_I \ip{\nu}{\mu} \, = \, \n_I \ip{J\nu}{\mu} \label{eq-bdryderiv};
\end{equation} 
this will allow us to apply Lemmas \ref{Neumannspace} and \ref{boundarymax}.
By \eqref{eq-omegatau}, $p$ is a boundary maximum of $|\tau|^2$, and so
\begin{flalign*}
0\,=\,\frac 1 2\n_I |\tau|^2 \, &= \, -\ip{\nu}{\mu}\n_I\ip{\nu}{\mu}-\ip{J\nu}{\mu}\n_I\ip{J\nu}{\mu} \\
&= -\frac{\ip{J\nu}{\mu}}{\cos \a}\left[\sin \a \, \n_I\ip{\nu}{\mu}+\cos \a \, \n_I\ip{J\nu}{\mu}\right].
\end{flalign*}
By (\ref{RandomIneq}) we have
\[ \sin \a \, \n_I\ip{\nu}{\mu}+\cos \a \, \n_I\ip{J\nu}{\mu} = 0, \]
and differentiating (\ref{eq-bdry}) yields
\[ \cos\a \,\n_I\ip{\nu}{\mu} - \sin\a \, \n_I \ip{J\nu}{\mu}= 0.\]
These together imply equation (\ref{eq-bdryderiv}).

 We now wish to estimate $\frac{1}{4}\n_\mu |\omega|^2 = \ip{\IIM_{\sigma \mu}}{J\mu}+\ip{\IIM_{\mu \mu}}{\tau}$ at the boundary in terms of $|\omega|^2$ or equivalently $|\tau|^2 = |\sigma|^2$. 
 
Using (\ref{eq-bdry}) and Lemmas \ref{Neumannspace} and \ref{boundarymax}:
\begin{flalign*}
\ip{\IIM_{\sigma \mu}}{J\mu} &= \ip{\IIM_{\sigma \mu}}{-\ip{J\nu}{\mu}\nu +\ip{\nu}{\mu}J\nu}\\
& = \frac{\ip{J\nu}{\mu}}{\cos \a}\ip{\IIM_{\sigma \mu}}{-\cos \a \, \nu +\sin \a \, J\nu}\\
&= \frac{\ip{J\nu}{\mu}}{\cos \a}\ip{\IIS_{\sigma \nu}}{\cos \a \, \mu +\sin \a \,J\mu} \,-\,\frac{\tan \a \ip{J\nu}{\mu}}{1-\ip{\nu}{\mu}^2}\left[\Big\langle\IIM_{\sigma \sigma},\nu\Big\rangle +\ip{\nu}{\mu}\ip{\IIS_{\sigma \sigma}}{\mu}\right].
\end{flalign*}
We may extract a $|\tau|^2$ from the second of these terms, so working with the first term:
\begin{flalign*}
\frac{\ip{J\nu}{\mu}}{\cos \a}&\ip{\IIS_{\sigma \nu}}{\cos \a \, \mu +\sin \a \, J\mu}\\
&=\frac{\ip{J\nu}{\mu}}{\cos \a}\ip{\IIS_{\sigma \nu}}{\cos \a \ip{J\nu}{\mu}J\nu+\cos \a \ip{\nu}{\mu}\nu  -\sin \a \ip{J\nu}{\mu}\nu +\sin \a\ip{\nu}{\mu}J\nu + \cos \a \tau}\\
&=\frac{\ip{J\nu}{\mu}^2}{\cos^2 \a}\ip{\IIS_{\sigma \nu}}{\cos^2 \a \, J\nu +\sin^2 \a \, J\nu + \cos^2 \a \tau}\\
&=\frac{\ip{J\nu}{\mu}^2}{\cos^2 \a}\ip{\IIS_{\sigma \nu}}{J\nu + \cos^2 \a \tau}.
 \end{flalign*}
Then, using \eqref{eq-2ffswap}, and Lemma \ref{Dirichletrewrite2} for the third line:
\begin{flalign*}
\frac{\ip{J\nu}{\mu}^2}{\cos^2 \a}\ip{\IIS_{\sigma \nu}}{J\nu + \cos^2 \a \tau}&=-\frac{\ip{J\nu}{\mu}^2}{\cos^2 \a}\ip{\IIS_{\nu \nu}}{\tau} \, - \, \ip{J\nu}{\mu}^2\ip{\IIS_{\sigma \sigma}}{J\nu}\\
&=\frac{\ip{J\nu}{\mu}^2}{\cos^2 \a}\ip{\IIS^I_{I}}{\tau}-\frac{\ip{J\nu}{\mu}^2}{\cos^2 \a}\ip{\widetilde{H}}{\tau} - \, \ip{J\nu}{\mu}^2\ip{\IIS_{\sigma \sigma}}{J\nu}\\
&=\frac{1-\ip{\nu}{\mu}^2}{\cos^2\a}\ip{\IIM^I_{I}}{\tau}  -\frac{\ip{J\nu}{\mu}^2}{\cos^2 \a}\ip{\widetilde{H}}{\tau}  - \ip{J\nu}{\mu}^2\ip{\IIS_{\sigma \sigma}}{J\nu}.\\
&\qquad
+ \frac{|\tau|^2}{\cos^2\a}\left[\ip{J\nu}{\mu}\ip{\IIS^I_{I}}{J\nu }
+\ip{\nu}{\mu}\ip{\IIM^I_{I}}{\nu}\right]
\end{flalign*}
The final two terms contain a $|\tau|^2$, so we work with only the first two terms. Using Lemma \ref{Dirichlettime}:
\begin{flalign*} \frac{1-\ip{\nu}{\mu}^2}{\cos^2\a}&\ip{\IIM^I_{I}}{\tau}-\frac{\ip{J\nu}{\mu}^2}{\cos^2 \a}\ip{\widetilde{H}}{\tau}\\& =\frac{1-\ip{\nu}{\mu}^2}{\cos^2\a}\left[\ip{H}{\tau}-\ip{\IIM_{\mu\mu}}{\tau} \right]-\frac{\ip{J\nu}{\mu}^2}{\cos^2 \a}\ip{\widetilde{H}}{\tau}\\
&=\frac{(1-\ip{\nu}{\mu}^2)|\tau|^2}{\cos^2\a\ip{J\nu}{\mu}}\ip{H-\widetilde{H}}{J\nu}-\frac{1-\ip{\nu}{\mu}^2}{\cos^2\a}\ip{\IIM_{\mu\mu}}{\tau} +\frac{|\tau|^2}{\cos^2\a}\ip{\widetilde{H}}{\tau}
\end{flalign*} 
Finally we note after rewriting $\ip{\IIM_{\sigma \mu}}{J\mu}$ following all the steps as above, the coefficient of $\ip{\IIM_{\mu\mu}}{\tau}$ in the overall equation for $\frac 1 2 \n_\mu |\omega|^2$ is now 
\begin{flalign*}
1-\left(\frac{1-\ip{\nu}{\mu}^2}{\cos^2\a}\right) \,&=\,\frac{-\sin^2\a+\ip{\nu}{\mu}^2}{\cos^2\a}
\,=\,-\tan^2 \a\left(1-\frac{\ip{\nu}{\mu}^2}{\sin^2\a}\right)
\,=\,-\tan^2\a \, |\tau|^2\ .
\end{flalign*}
Putting all of this together, we obtain the result.
\end{proof}

We now need a function $\rho$ with a bounded evolution such that $\n_\mu \rho = 1$ for all boundary points. A natural choice would be the ambient distance to $\Sigma$, but unfortunately this is not smooth at $\Sigma$ and we cannot in general avoid intersections of the interior of $M$ with $\Sigma$ due to the lack of comparison principles in higher codimension. We instead consider a function based on the intrinsic distance to $\Sigma$.

\begin{lemma}\label{rho}
Suppose $\Sigma_t$ satisfies LMCF and $M_t$ satisfies (\ref{MCFBC}) such that there exist constants $C_\Sigma$ and $C_M$ so that
\[\sup_{M \times[0,T)} |\IIM|<C_M, \qquad\qquad \sup_{\Sigma \times [0,T)} |\IIS|<C_\Sigma \ .\] 
Let $\operatorname{Inj}(\p M_t) > \delta > 0$ on $[0,T)$. Then there exists a function $\rho:M_t \ra \bb{R}$ which is smooth and has the properties that 
\[\begin{cases}
\ho \rho \leq C_\rho & \text{on }M_t\\
\n_\mu \rho =-1 &\text{on } \partial M_t
\end{cases}
\]
where $C_\rho$ depends only on $\IIS$, $\IIM$, and $\delta$.
\end{lemma}

\begin{proof}
		Let $r(p,t)= \operatorname{dist}^{M_t}(p, \partial M_t) $, $\, r:M\times[0,T)\rightarrow \mathbb{R}$ be the intrinsic distance to the boundary. Note that $r$ satisfies $\nabla_{\mu} r = -1$ at the boundary. Define the collar region $U_{R} \subset M$ by 
		\[U_{R} = \{p \in M \,: \, r(p,t) \leq R, \,\,\, \forall t \in [0,T)\},\]
		and denote by $g_t$ the pullback metric on $U_R$ at time $t$. Since $\IIM$ and $\IIS$ are uniformly bounded, we can guarantee that $r$ is smooth on $U_R$ by choosing $R< \delta $ sufficiently small (dependent on $\IIS, \IIM$) so that $F_t(U_R)$ contains no focal or conjugate points for all times $t\in [0,T)$. We write the metric on $U_R$ as a product metric $g_t = dr^2 + g_r$, and note that since $r$ is a non-singular distance function, we have the fundamental equation
		\begin{equation}\label{HessEqn} \partial_r g_r  =2\operatorname{Hess}(r), 
		\end{equation}
		(see for instance \cite[section 3.2.4]{Petersen_2018}). Since (\ref{HessEqn}) is linear, the Hessian cannot blow-up on $U_R$ unless the metric degenerates. However, since $U_R$ contains no  focal points, $g_r$ cannot degenerate and hence
		\[|\operatorname{Hess}(r)| \leq  C(\IIS, \IIM).\]
		
		We now consider the time derivative of $r$ for $r<\frac 1 2 R$. For any $p,t$ we have that there exists a unique geodesic $\gamma_{(p,t)}:[0,1]\ra M$ such that $\ell(\gamma_{(p,t)})=r$, $\gamma_{(p,t)}(0)=p$ and $\gamma_{(p,t)}(1)\in\partial M$. $\gamma_{(p,t)}$ must vary smoothly with time as otherwise it would contain conjugate points which are disallowed by the restriction of $r$. Since $\gamma_{(p,t)}(s)$ is a minimiser for the metric $g_t$ we have 
		\[0=\frac{1}{\ell(\gamma)}\int_0^1 \ip{\frac{d\gamma'}{dt}}{\gamma'}_{g_t} ds,\]
		where from now on we will abuse notation and write $\gamma_{(p,t)}=\gamma$.
		We therefore calculate (using \cite[Lemma 4]{Smoczyk2002}) that
		\[\left. \ddt{r}\right|_{(p,t)}=\frac{1}{\ell(\gamma)}\int_0^1\left(\ip{\frac{d\gamma'}{dt}}{\gamma'}_{g_t}+\frac 1 2(\gamma')^i(\gamma')^j\ddt{g_{ij}}\right)ds=-\frac{1}{\ell(\gamma)}\int_0^1\ip{H(\gamma)}{\IIM_\gamma(\gamma',\gamma')}ds.\]
		We therefore have that for $r<\frac 1 2 R$,
		\[\ho r \leq C(\IIS, \IIM)\]
		and at the boundary
		\[\n_\mu r=-1\ .\]
		The lemma is achieved by setting $\rho=\eta(r)$ where $\eta$ is a smooth cutoff function so that 
		\[
		\begin{cases}
		\eta(x)=x &\text{for } x\in[0,\frac{R}{8}]\\
		\eta(x)=\frac R 4 &\text{for } x\in[\frac R 2 ,\infty)\\
		\pard{\eta}{x}(x)<8 & \text{for } x\in \bb{R}.
		\end{cases}
		\]
\end{proof}

\begin{lemma}\label{omegamaxprinc}
Suppose that $\Sigma_t$ satisfies LMCF and $M_t$ is a solution of (\ref{MCFBC}) on the time interval $[0,T)$. Suppose that there exist constants $C_M$, $C_\Sigma$  and $\delta_\Sigma$ as in Lemma \ref{rho}. Suppose that $\sup_{M_0}|\omega|^2<\frac 1 2$ and $\widetilde{T}$ is chosen so that for all $t\in[0,\widetilde{T})$, $\sup_{M_t} |\omega|^2< \frac 1 2$. Then, there exists constants $C_1=C_1(C_M, C_\Sigma,n)$, $C_2=C_2(C_M, C_\Sigma,n)$ such that for all $t\in[0,\widetilde{T})$,
\[|\omega|^2 \leq C_1 \ e^{C_2t}\ \underset{M_0}{\sup}|\omega|^2\]
\end{lemma}
\begin{proof}
  For $\rho$ as in Lemma \ref{rho}, we now consider 
 \[f = |\omega|^2e^{A\rho-Bt}\]
 where $0<A,B\in \bb{R}$. At the boundary we note that using Lemmas \ref{boundaryomega} and \ref{rho} 
 \[\n_\mu f \leq |\omega|^2e^{A\rho-Bt}(C(C_\Sigma+C_M)-A)\]
 which is negative if we set $A=C(C_\Sigma+C_M)+1$. Therefore $f$ has no boundary maxima.
 
 Using the estimates of Smoczyk \cite[Lemma 3.2.8]{Smoczyk2000} we have that there exists a $C_2=C_2(C_M)$ so that
 \[\ho |\omega|^2 \leq C_2 |\omega|^2\ .\]
 As a result, at an increasing maximum of $f$ we may estimate
 \begin{flalign*}
  0&\leq\ho f \\&= |\omega|^2e^{A\rho-Bt}\left[ \frac{1}{|\omega|^2}\ho |\omega|^2 +A\ho \rho -A^2 |\n \rho|^2-2\ip{\frac{\n|\omega|^2}{|\omega|^2}}{A\n\rho}-B\right]\\
  &= |\omega|^2e^{A\rho-Bt}\left[ \frac{1}{|\omega|^2}\ho |\omega|^2 +A\ho \rho +A^2 |\n \rho|^2-B\right]\\
  &\leq |\omega|^2e^{A\rho-Bt}\left[C_2+AC_\rho+A^2 - B\right]
 \end{flalign*}
where we used that as at a maximum $\n f = 0$, we have that $\frac{\n|\omega|^2}{|\omega|^2}= -A\n\rho$. Clearly, making $B$ sufficiently large now yields a contradiction, implying that 
\[f\leq \underset{M_0}\sup f\ ,\]
completing the proof.
\end{proof}

\begin{proof}[Proof of Theorem \ref{LMCFBdry}]
Suppose $M_t$ is a solution of (\ref{MCFBC}) with $M_0$ Lagrangian and $\operatorname{inj}(\p M_t) > \delta > 0$, for $ t \in [0,T)$. Then for any $\widehat{T}\in(0,T)$, there exists a constant $C_M$ so that 
\[\sup_{L^n\times[0,\widehat{T})} |\IIM|<C_M, \qquad \sup_{L^n\times[0,\widehat{T})} |\IIS|<C_\Sigma\ .\] 
There also exists a maximal time $\widetilde{T}\leq \widehat{T}$ such that for all $t\in[0,\widetilde{T})$, $\sup_{M_t} |\omega|^2< \frac 1 2$. We may therefore apply Lemma \ref{omegamaxprinc} to see that for all $t\in(0,\widetilde{T})$, $|\omega|^2=0$ and so we see that $\widetilde{T}=\widehat{T}$. As $\widehat{T}$ was arbitrary we see that for all $t\in[0,T)$, $|\omega|^2\equiv0$ .
\end{proof}

\section{Equivariant Examples}\label{Examples}

In this section, we examine the behaviour of LMCF with boundary in the equivariant case, with two very natural choices of boundary manifold - the Lawlor neck and the Clifford torus. In both cases, we prove a long-time existence and smooth convergence result - of the original flow in the case of the Lawlor neck, and of a rescaled flow in the case of the Clifford torus.

\subsection{Long-Time Convergence to a Special Lagrangian}

Before we specialise to our two specific boundary manifolds, we will first prove the following more general proposition about long time convergence of LMCF with boundary to a special Lagrangian. We remark that this holds not just in the equivariant case, but for any uniformly smooth almost-calibrated flow that exists for all time. 

\begin{proposition}\label{prop-conv}
	Suppose that:
	\begin{itemize}
		\item $\Sigma_t = \Sigma$ is a special Lagrangian with Lagrangian angle $\frac \pi 2$,
		\item $L_0$ is almost-calibrated, that is $\theta_0\in(-\frac{\pi}{2}+\e , \frac \pi 2 - \e)$, 
		\item and the solution to \eqref{IntroBVP}, $L_t$, exists for $t\in[0,\infty)$ with uniform estimates $|\n^k \II|^2<C_k$. 
	\end{itemize}
	Then $L_t$ converges smoothly to a special Lagrangian with Lagrangian angle $\a$.
\end{proposition}

To begin, we calculate the following evolution equation:
\begin{lemma}\label{integralevol}
	Suppose $L_0$ is zero-Maslov and $L_t$ is a solution to \eqref{IntroBVP}. Then for any be a smooth function $f$ on $L_t$,
	\[\ddt{}\int_{L_t}f d\mathcal{H}^n = \int_{L_t} \ddt{f} -|H|^2 f d\mathcal{H}^n + \int_{\partial L_t} f\left[\ip{\widetilde{H}}{J\nu}\ip{J\nu}{\mu}^{-1}-\tan \a \n_\mu \theta\right]d\mathcal{H}^{n-1} .\]
\end{lemma}
\begin{proof}
	Here we have to distinguish between the standard mean curvature flow $F$
	\[\ddt{F} = H\]
	which may ``flow through the boundary'' and a reparametrised mean curvature flow $X:L^n \ra \CY$ such that $X(\partial L, t) \subset\Sigma_t$ and
	$\left(\ddt{X}\right)^\perp=H$, say
	\[\ddt{X} = H + V,\]
	where $V$ is a time dependent tangential vector field on $L_t$. 
	In particular with respect to $X$, we have
	\begin{flalign*}
	\ddt{}\ip{X_i}{X_j} &= -2H^{\alpha}\II_{{\alpha}ij} + \ip{\ov\n_{X_j} V}{X_i}+ \ip{\ov\n_{X_i} V}{X_j}\ .
	\end{flalign*}
	We therefore see that for a general smooth function $f$,
	\begin{align*}
	\ddt{}\int_{L_t} f d\mathcal{H}^n &= \int_{L_t} \pard{f}t +f\operatorname{div}(V)-|H|^2 fd\mathcal{H}^n \\
	&= \int_{L_t} \pard{f}t -\ip{V}{\n f}-|H|^2 fd\mathcal{H}^n+\int_{\partial L_t}f\ip{V}{\mu} d\mathcal{H}^{n-1}\ ,
	\end{align*}
	where we write $\pard{f}{t}$ for time differentiation with respect to $X$ (as opposed to $F$, for which we write $\ddt{f}$) and we note that 
	\[\ddt{f} = \pard{f}{t}-\ip{\n f}{V}\ .\]
	At the boundary $H-\widetilde{H}+V \in T\Sigma_t$ and so, as in the proof of Lemma \ref{Dirichlettime}, $H^{N\Sigma}-\widetilde{H}=C J \nu$. Writing $V$ in the basis from Section \ref{BoundaryCond},
	\[ \ip{V}{\mu}\ip{\mu}{J\nu}=\ip{V}{J\nu} =\ip{\widetilde{H}-H}{J\nu}\ .\]
	We observe that due to our boundary condition, $\ip{H}{J\nu}=\ip{H}{J\mu}\ip{\mu}{\nu}=\ip{\mu}{\nu}\n_\mu \theta$, and recall that $\frac{\ip{\nu}{\mu}}{\ip{J\nu}{\mu}}=\tan\a$, completing the Lemma.
\end{proof}
\begin{cor}\label{integralid}
	If $\Sigma$ is special Lagrangian with Lagrangian angle $\frac \pi 2$, then
	\[\ddt{}\int_{L_t}f d\mathcal{H}^n = \int_{L_t} \ho{f} -|H|^2 f d\mathcal{H}^n + \int_{\partial L_t} \n_\mu f-f \tan \a  \n_\mu \theta d\mathcal{H}^{n-1}\ ,\]
	and if $f=f(\theta)$ then
	\[\ddt{}\int_{L_t}f d\mathcal{H}^n = \int_{L_t} -|H|^2(f''+ f) d\mathcal{H}^n + \int_{\partial L_t} (f'-f\tan \a ) \n_\mu \theta d\mathcal{H}^{n-1}\ .\]
\end{cor}
We now make the following observation
\begin{lemma}
	If $\Sigma$ is special Lagrangian with Lagrangian angle $\frac \pi 2$, and $\theta_0\in(-\frac{\pi}{2}+\veps, \frac \pi 2 - \veps)$ then while the flow exists
	\[\ddt{}\int_{L_t} \cos (\theta)d\mathcal{H}^n=0\ .\]
	In particular, $|L_t|$ is bounded from above and below.
\end{lemma}
\begin{proof}
	Due to the boundary condition on $\partial L$, $\theta = -\alpha$, and so the maximum principle implies that the bounds on $\theta$ are preserved. Set $f(x) = \cos(x)$, then $f''=-f$ and $f'(-\a)-\tan(\a)f(-\a) = 0$. $|L_t|$ is bounded as $\cos(\theta)$ is bounded from above and below away from 0 (depending on $\veps$).
\end{proof}

\begin{lemma}\label{IntergralEstimates}
	If $\Sigma$ is special Lagrangian with Lagrangian angle $\frac \pi 2$, $L_0$ is zero Maslov and there exists a constant $V$ such that $|L_t|<V$. Then there exists a constant $c=c(n,V)$ such that
	\[\int_{L_t} (\theta +\a)^2 d\mathcal{H}^n \leq Ce^{-ct}, \qquad \int_0^\infty\int_{L_t} |H|^2e^{\frac c 2 t}d\mathcal{H}^n dt \leq C\ \]
\end{lemma}
\begin{proof}
	We apply Corollary \ref{integralid} with $f(\theta)=(\theta+\a)^{2p}$ for some $p\geq 1$. In particular, at the boundary $f=f'=0$ and so
	\begin{flalign*}
	\ddt{}\int_{L_t} (\theta+\a)^{2p} d\mathcal{H}^n% = -\int_{L_t}|H|^2\left[(\theta+\a)^{2p} + 2p(2p-1)(\theta+\a)^{2p-2}\right] d\operatorname{vol} 
	=-\int_{L_t}|H|^2(\theta+\a)^{2p} + \frac{2p(2p-1)}{p^2}|\n(\theta+\a)^p|^2 d\mathcal{H}^n.
	\end{flalign*}
	We recall that the Micheal--Simon Sobolev inequality \cite{MSIneq} implies that
	\[\left(\int_{L_t} \phi^\frac{2n}{n-1}\right)^\frac{n-1}{2n} \leq C(n, |L_t|)\sqrt{\int_{L_t} |\n \phi|^2+|H|^2|\phi|^2 d\mathcal{H}^n}, \]
	and we note that as $\theta+\a$ is zero on $\partial L_t$, it is a function of compact support on the interior of $L_t$ and this theorem applies to $\phi=(\theta+\a)^p$ for all $p\geq 1$ (alternatively see \cite[Lemma 1.1]{Gerhardt}).
	
	\hide{
		This follows from the standard Micheal Simon inequality : Put $\phi^p$ in place of $\phi$ to get
		\begin{align*}
		\left(\int_{L_t} \phi^\frac{pn}{n-1}\right)^\frac{p(n-1)}{pn}&=\left(\int_{L_t} \phi^\frac{pn}{n-1}\right)^\frac{n-1}{n}\\
		& \leq C(n)\int_{L_t} |\n \phi|\phi^{p-1}+|H\phi||\phi|^{p-1}d\mathcal{H}^n \\
		&\leq C(n,p, |L_t|)\left(\int_{L_t}\phi^p d\operatorname{vol}\right)^\frac{p-1}{p}\left(\int_{L_t}|\n\phi|^p+|\phi|^pd\mathcal{H}^n\right)^\frac 1 p\\
		&\leq C(n,p, |L_t|)\left(\int_{L_t}\phi^\frac{np}{n-1} d\mathcal{H}^n\right)^\frac{(p-1)(n-1)}{pn}\left(\int_{L_t}|\n\phi|^p+|\phi|^pd\mathcal{H}^n\right)^\frac 1 p
		\end{align*}
		so
		\[\left(\int_{L_t} \phi^\frac{pn}{n-1}\right)^\frac{(n-1)}{pn}\leq C(n,p, |L_t|)\left(\int_{L_t}|\n\phi|^p+|\phi|^pd\mathcal{H}^n\right)^\frac 1 p\]
	}
	We see that by choosing $\phi = (\theta+\a)^{p}$ then
	\begin{flalign*}
	\ddt{}\int_{L_t} (\theta+\a)^{2p} d\mathcal{H}^n &\leq -\tilde{c}(n, |L_t|)\left(\int_{L_t}\left[(\theta+\a)^{2p}\right]^\frac{n}{n-1}d\mathcal{H}^n\right)^\frac{n-1}n \\
	&\leq -c(n, |L_t|)\int_{L_t}(\theta+\a)^{2p}d\mathcal{H}^n,
	\end{flalign*}
	and so
	\[\ddt{}\int_{L_t} (\theta+\a)^{2p}e^{ct} d\mathcal{H}^n\leq 0.\]
	
	Repeating the above for $p=1$, but only using half the possible exponent in $t$ we have
	\begin{flalign*}
	\ddt{}\int_{L_t} (\theta+\a)^{2}e^{\frac{c}{2}t} d\mathcal{H}^n &\leq -\frac 1 2 e^{\frac{c}{2}t}\int_{L_t} |H|^2(\theta + \alpha)^2 + 2|\n\theta|^2 d\mathcal{H}^n \leq -\frac 1 2 e^{\frac{c}{2}t}\int_{L_t}2|H|^2 d\mathcal{H}^n\ .
	\end{flalign*}
	Integrating implies the final claim.
\end{proof}

\textit{Proof of Proposition \ref{prop-conv}.} Due to Lemma \ref{IntergralEstimates} and the above regularity assumptions, there exists a $T>0$ such that for all $t>T$, $|H|<e^{-\frac{c}4t}$. This bounds the normal velocity of the parametrisation $F$, and as a result we see that for $s,t>T$, $\operatorname{dist}(L_s,L_t)< \frac{4}{c}e^{-\frac{c}{4}\operatorname{min}\{s,t\}}$. Clearly, as $t\ra\infty$, $H\ra 0$, and so we see that $L_t$ converges to a special Lagrangian, first subsequentially by Arzela--Ascoli, then uniformly by the above, then smoothly by interpolation.
\qed

\subsection{The Lawlor Neck}

Our first example is an LMCF with boundary on the Lawlor neck, which has constant Lagrangian angle $\tilde\theta = \frac{\pi}{2}$. It follows that the boundary condition of (\ref{IntroBVP}) is equivalent to 
\[ \theta\big|_{\p L} \, = \, -\alpha. \]
We prove the following long-time existence result.

\begin{theorem}
	Let $L_0$ be an $S^1$-equivariant Lagrangian embedding of the disc $D^2$ into $\mathbb{C}^2$ with Lagrangian angle $\theta_0$ satisfying
	\[ \theta_0(s) \, \in \, (-\tfrac{\pi}{2}+ \veps, \, \tfrac{\pi}{2}- \veps)\]
	for some $\veps > 0$, with boundary on the Lawlor neck with profile curve $\sigma_{\operatorname{Law}} = \{(\pm\cosh(\phi), \sinh(\phi)) : \phi \in \mathbb{R}\},$ and with $\theta_0|_{\p L_0}=-\alpha$ (as in Figures \ref{fig-lawlorcase1} and \ref{fig-lawlorcase2}).
	Then there exists a unique, immortal solution to the LMCF problem:
	\begin{align}
	\begin{cases}
	\left(\ddt{} F(x,t)\right)^{NM} = H(x,t)& \text{for all }(x,t) \in D\times[t_0,\infty)\\
	F(x,t_0)=L_0(x)&\text{for all }x\in D\\
	\partial L_t\subset \Sigma_{\operatorname{Law}} & \text{for all }t\in[t_0,\infty)\\
	\theta_t|_{\p L_t}=-\alpha &\text{for all } (x,t)\in \p D \times[t_0,\infty),
	\end{cases} \label{eq-equiflow}
	\end{align}
	and it converges smoothly in infinite time to the disc with profile curve $\gamma^{\infty}(s) = (s,s\tan(\frac{-\alpha}{2}))$.
\end{theorem}
\begin{remark}
	The `almost-calibrated' condition $\theta_0 \in (-\tfrac{\pi}{2}+ \veps, \, \tfrac{\pi}{2}- \veps)$  is necessary, as there exist Lagrangian discs which are not almost-calibrated but which form a finite-time singularity under the flow - see \cite{Neves2013} for an example.
	
	If the profile curve $\gamma_t$ does not pass through the origin, i.e. if the topology of the flow is not a disc, then a finite-time singularity will form. For example one can prove using the barriers of this section that any curve that does not initially pass through the origin must approach the origin as $t \rightarrow \infty$, and therefore by the equivariance the curvature $|A|^2$ must blow up.
\end{remark} 

\subsubsection{Parametrisation}

For simplicity, we work throughout with the profile curves of our flow and the boundary manifold, and we will work with the following parametrisation for the profile curve. Consider the foliation
\[ Y(s, \phi) \, := \, (s \cosh(\phi), s \sinh(\phi) ) \]
and graphs of the form
\begin{align}
\gamma_t(s) \, &= \, Y(s, v_t(s)) \, = \, (s \cosh(v_t(s)), \, s \sinh(v_t(s))) \label{eq-parametrisation2}\\
\gamma'_t(s) \, &= \, (\cosh(v_t(s)) + sv'_t(s)\sinh(v_t(s)), \, \sinh(v_t(s)) + sv'_t(s)\cosh(v_t(s))).\notag
\end{align}
In this parametrisation, the problem (\ref{eq-equiflow}) is reduced to the following boundary value problem:
\begin{equation}
\begin{cases}\label{eq-graphmcf}
\frac{\p v}{\p t} \, = \, \frac{v'' + 2s^{-1}v' - s(v')^3}{|\gamma'|^2} \, + \, \frac{v'}{s \cosh(2v)}  &\text{for } \, s \in [-1,1], \,\,t \geq t_0,\\
v(s,t_0) = v_0 &\text{for } \, s \in [-1,1], \\
sv'(s,t) \, = \, \frac{\tan(-\alpha)}{\cosh(2v(s,t))}-\tanh(2v(s,t)) &\text{for } \, s \in \{-1,1\}, \,\, t \geq t_0.
\end{cases}
\end{equation}

Note that this PDE problem is uniformly parabolic away from the origin, if we can bound $|\gamma'|$ and $|\gamma| = s\cosh(2v)$. We must also show that this parametrisation is valid for our problem.

\subsubsection{The Lagrangian Angle and $C^1$ Bounds}

The Lagrangian angle for an equivariant LMCF is given by
\[ \theta(s) \, = \, \arg(\gamma) + \arg(\gamma').  \]
It is an important quantity, because on the interior of the abstract manifold it has very simple evolution equations:
\begin{equation} \fr{\p \theta}{\p t} \, = \, \Delta\theta, \quad \quad \fr{\p (\theta)^2}{\p t} \, = \, -2|H|^2 + \Delta (\theta)^2. \label{eq-langlevol}
\end{equation}
\begin{lemma}
	A solution of (\ref{eq-equiflow}) on $[t_0,T)$ which satisfies $\theta \in (-\tfrac{\pi}{2}+ \veps, \, \tfrac{\pi}{2}- \veps)$ at the initial time, satisfies this condition for all $t \in [t_0,T)$.
\end{lemma}
\begin{proof}
	The boundary conditions on our flow are $\theta\big|_{\p L} = -\alpha$. Therefore by (\ref{eq-langlevol}), $\theta$ solves the Dirichlet problem for the heat equation on the abstract manifold, and by the parabolic maximum principle must be bounded by its initial values.
\end{proof}
We will now show that our flow may be parametrised using the parametrisation (\ref{eq-parametrisation2}) for as long as the flow exists, and derive $C^1$ bounds on the graph function $v$ away from the origin. Certainly it may be parametrised in this way on a small ball $B$ around the origin, since at the origin we have the identity
\[ \theta = 2\arg(\gamma'),\] 
and so it follows from the almost-calibrated condition for $\theta$ that, on $B$, the curve intersects the Lawlor neck foliation $Y(s,\phi)$ transversely. On this ball $B$,
\begin{align}
\theta(s) \, &= \arg(\gamma) + \arg(\gamma') \, = \, \arg(\gamma \gamma') \notag \\
&= \, \arg \left( s \, + \, i \left( s \sinh(2v) + s^2 v' \cosh(2v) \right) \right) \notag\\
\implies \tan(\theta) \, &= \, \sinh(2v) + sv' \cosh(2v) \notag \\
\implies sv' \, &= \, \frac{\tan(\theta)}{\cosh(2v)} \, - \, \tanh(2v) \label{eq-v'bound} \\
\implies |\gamma'(s)| \, &\leq \, \left(1 + s|v'|\right)(\cosh(v) + |\sinh(v)|) \notag\\
& \leq \left(1 + \frac{|\tan(\theta)|}{\cosh(2v)} \, + \, |\tanh(2v)|\right)(\cosh(v) + |\sinh(v)|). \label{eq-gamma'bound}
\end{align}
This will give us a uniform $C^1$ bound for $v$ on any annulus centred at the origin, if we can parametrise globally in this way, and bound the function $v$.
\begin{lemma}\label{lem-bound}
	Let $\gamma$ be the profile curve of an equivariant Lagrangian submanifold $L \subset \mathbb{C}^2$ with boundary on the Lawlor neck, satisfying $\theta \in (-\frac{\pi}{2}+\veps,\frac{\pi}{2}-\veps)$. Then one connected component of the curve $\gamma\setminus\{O\}$ is parametrisable using the parametrisation \eqref{eq-parametrisation2}, and satisfies
	\begin{align*}
	\arg(\gamma) \, &\in \, (-\tfrac{\pi}{4} + \tfrac{\veps}{2}, \, \tfrac{\pi}{4} - \tfrac{\veps}{2}), \\
	v \, &\in \, (-V,V),
	\end{align*}
	for $V = \tanh^{-1}(\tan(\tfrac{\pi}{4} - \tfrac{\veps}{2}) ) \, < \, \infty$.
	The other connected component satisfies analogous bounds.
\end{lemma}
\begin{proof}
	At the origin, we must have $\arg(\gamma'(0)) \in (-\tfrac{\pi}{4} + \tfrac{\veps}{2}, \, \tfrac{\pi}{4} - \tfrac{\veps}{2})$ (for one choice of orientation) by the bound on $\theta$, therefore for small s the curve is parametrisable by \eqref{eq-parametrisation2}, and the first bound holds. 
	If there was some smallest $s_0$ such that
	\[\arg(\gamma(s_0)) = \tfrac{\pi}{4} - \tfrac{\veps}{2},\]
	that at this point,
	\begin{align*}
	\arg(\gamma'(s_0)) &\geq \tfrac{\pi}{4} - \tfrac{\veps}{2}
	\implies \theta(s_0) \, \geq \, \tfrac{\pi}{2} - \veps
	\end{align*}
	which is a contradiction. An identical argument works for the lower bound, and so the first statement is proven.
	
	For the second, note that in the foliation  $Y(s, \phi) = (s\cosh(\phi), s\sinh(\phi))$, the line of constant argument $\alpha$ satisfies
	\begin{align*}
	\tan(\alpha) \, &= \, \frac{\sinh(\phi)}{\cosh(\phi)} \, = \, \tanh(\phi) 	\implies \phi \, = \, \tanh^{-1}(\tan(\alpha)),
	\end{align*}
	therefore lines of constant angle are equivalent to lines of constant $\phi$, with the above correspondence. The first bound then implies the second, for as long as the parametrisation is valid.
	Finally, this bound on $v$, along with (\ref{eq-v'bound}), proves that $v'$ is bounded on any annulus - therefore the parametrisation is valid for all $s>0$. The other half of the curve $\gamma$ is a reflection of the first in the origin, by the equivariance, and so analogous results hold.
\end{proof}
Using this lemma, (\ref{eq-gamma'bound}) implies that $|\gamma'(s)|<C_1$, for some uniform constant $C_1$. We can use this to derive the following density bound on small balls, which will be useful later:
\begin{align*}
\int_{B_\delta \cap {\gamma_t}} d\mathcal{H}^1
\leq \, \int_{-\delta}^{\delta} |{\gamma}'(s)| ds \leq 2\delta C_1.
\end{align*}

\subsubsection{Long-Time Existence}

Using the mean curvature flow equation \eqref{eq-graphmcf}, and the $C^1$ bounds we just derived, we can now prove long-time existence.
\begin{lemma}\label{lem-nofintimesing}
	A finite-time singularity for a solution of (\ref{eq-equiflow}) cannot occur.
\end{lemma}
\begin{proof}
	By (\ref{eq-gamma'bound}), the mean curvature flow equation (\ref{eq-graphmcf}) is uniformly parabolic on any annulus centred at the origin. Therefore, Schauder estimates give a bound on all curvatures for as long as the flow exists, and so a singularity cannot occur away from the origin.
	
	Unfortunately, the equation (\ref{eq-graphmcf}) degenerates at the origin, so this case must be dealt with separately. Assume that a singularity occurs at the origin at time $0$, and let $L^i_t$, $\gamma^i_t$ be the type I rescalings of the rotated flow and their profile curves around this singularity with factor $\lambda^i$, defined by
	\[ L^i_t \, := \, \lambda^i L_{(\lambda^i)^{-2}t}. \]
	We will show that the density of $\gamma^i_t$ converges to 1, and then White's local regularity theorem will imply that the curvatures are bounded, contradicting the assumption of a singularity at $(O,0)$.
	\begin{lemma}\label{lem-nev} Let $L^i_t$ be a sequence of rescalings of an equivariant LMCF $L_t \subset \mathbb{C}^2$ around the spacetime point $(O,0)$. Assume that $\p L^i_t \rightarrow \infty$ as $i \rightarrow \infty$, uniformly on the time interval $[t_0,0)$, and assume also that the flow is uniformly bounded in $C^3$ on $\p L_t$.
		
		Then for any $a<b<0$ and $R>0$,
		\[ \lim_{i\rightarrow \infty} \int_a^b \int_{L^i_t \cap B_R} \left( |H|^2 + |x^\perp|^2 \right) d\mathcal{H}^2 \, = \, 0.\]
	\end{lemma}
	\begin{proof}
		We need the following version of Huisken's monotonicity formula, which holds for flows $M_t^n$ with boundary. For a spacetime point $X:=(x_0,t_0)$,
		\begin{multline} \frac{\p}{\p t} \int_{M_t} f \Phi_X \, d\mathcal{H}^n \, = \, \int_{M_t} \Phi_X \left( \frac{\p f}{\p t} - \Delta^{M}f - f \left| H + \frac{(x- x_0)^\perp}{2(t_0 - t)} \right|^2 \right) d\mathcal{H}^n \, \\ + \, \int_{\p M_t} \Phi_X \ip{f \frac{x-x_0}{2(t_0-t)} + \nabla^Mf}{\nu} d\mathcal{H}^{n-1}. \label{monotonicity}
		\end{multline}
		This formula is derived the same way as the standard monotonicity formula, but there are extra boundary terms from use of the divergence theorem. 
		Using (\ref{eq-langlevol}),(\ref{monotonicity}) and denoting by $\Phi$ the monotonicity kernel centred at $(O,0)$, 
		\begin{align}
		\frac{\p}{\p t} \int_{L^i_t} \Phi d\mathcal{H}^2 \, &= \, -\int_{L^i_t} \Phi \left| H - \frac{x^\perp}{2t} \right|^2 \, d\mathcal{H}^2 \, + \, \int_{\p L^i_t} \Phi \ip{-\frac{x}{2t}}{\nu} d\mathcal{H}^{1}, \label{eq-evol1}\\
		\frac{\p}{\p t} \int_{L^i_t} (\theta^i_t)^2 \Phi d\mathcal{H}^2\, &= \, \int_{L^i_t} \Phi \left( - 2|H|^2 - (\theta^i_t)^2 \left| H - \frac{x^\perp}{2t} \right|^2 \right) \, d\mathcal{H}^2 \, \notag \\ &\qquad \qquad +  \, \int_{\p L^i_t} \Phi \ip{-(\theta^i_t)^2\frac{x}{2t} + \nabla^M(\theta^i_t)^2}{\nu} d\mathcal{H}^{1}. \notag
		\end{align}
		Therefore, 
		\begin{align*}
		&\lim_{i\rightarrow \infty} 2 \int_a^b \int_{L^i_t} |H|^2 \Phi d\mathcal{H}^2 dt \\ 
		&\leq \lim_{i\rightarrow \infty} \left( \int_{L^i_a} (\theta^i_a)^2 \Phi d\mathcal{H}^2 - \int_{L^i_b} (\theta^i_b)^2 \Phi d\mathcal{H}^2 + \int_a^b \int_{\p L^i_t} \Phi \ip{-(\theta^i_t)^2\frac{x^\perp}{2t} + \nabla^M(\theta^i_t)^2}{\nu} d\mathcal{H}^{1} dt \right).
		\end{align*}
		The boundary $\p L^i_t$ is a circle, radius $d^i(t)>\mu^i$ for $\mu^i \rightarrow \infty$ independent of $t$, and circumference $2\pi d^i(t)$. Additionally, the Lagrangian angle and its derivative are bounded on $\p L^i_t$ by the assumed $C^3$ bound, so we can estimate the last integral using a constant $C$ depending only on this bound. Using this, and relating the first two integrals to the original flow by scaling invariance of the heat kernel,
		\begin{align*}
		&\lim_{i\rightarrow \infty} 2 \int_a^b \int_{L^i_t} |H|^2 \Phi d\mathcal{H}^2 dt \\
		&\leq \lim_{i\rightarrow \infty} \left( \int_{L_{(\lambda^i)^{-2}a}} (\theta_{(\lambda^i)^{-2}a})^2 \Phi d\mathcal{H}^2 - \int_{L_{(\lambda^i)^{-2}b}} (\theta_{(\lambda^i)^{-2}b})^2 \Phi d\mathcal{H}^2 + C\int_a^b 2\pi d^i(t) e^\frac{-d^i(t)^2}{2t} (d^i(t) + 1) dt \right).
		\end{align*}
		This limit is equal to 0, since by Huisken monotonicity with boundary (\ref{monotonicity}) the first two terms cancel in the limit and by assumption $d^i(t) \rightarrow \infty$. It can similarly be shown using (\ref{eq-evol1}) that
		\begin{align*}\lim_{i\rightarrow \infty} \int_a^b \int_{L^i_t} \left| H - \frac{x^\perp}{2t} \right|^2 \Phi d\mathcal{H}^2 dt = 0,
		\end{align*}
		and since on $B_R{\times}[a,b]$ we can estimate $\Phi$ from below, these together imply the result.
	\end{proof}
	We now continue with the proof. Note that Schauder estimates applied to the graph equation (\ref{eq-graphmcf}) imply that our flow has uniformly bounded curvatures at the boundary, and since the Lawlor neck is static, it diverges to infinity under any sequence of rescalings - therefore Lemma \ref{lem-nev} may be applied.	Consider the set 
	\[ K := \{ (s \cosh(v),s \sinh(v) ) | s \in [-R,R], v \in [-V,V]\};\]
	 $K$ must contain $\gamma_t \cap (B_R{\setminus}B_\delta
	 )$ for any $t$. The set $K$ is itself contained in a larger ball, $B_{\tilde R}$, and on this ball we can apply Lemma \ref{lem-nev} to show that, for almost all $t$,
	\begin{align*}
	\int_{\gamma^i_t \cap B_{\tilde{R}}} |\gamma^\perp|^2 d\mathcal{H}^1 \rightarrow 0
	\end{align*}
	as $i \rightarrow \infty$ (where we suppress the superscript $i$ for readability). Therefore,
	\begin{align*}
	\int_{\gamma^i_t \cap B_{\tilde{R}}} |\gamma^\perp|^2 d\mathcal{H}^1 \, &\geq \, 2\int_\delta^R \frac{s^4(v_t')^2}{|\gamma_t'|}ds \,\geq\, \frac{2\delta^4}{C_1} \int_{\delta}^R(v_t')^2 ds \rightarrow 0. 
	\end{align*}
	It follows by H\"older's inequality that $v \rightarrow \overline v \in \mathbb{R}$ uniformly as $i \rightarrow \infty$, and that $v^{-1}(\gamma^i_t \cap B_R) \rightarrow \left[ -\frac{R}{\sqrt{\cosh(2\overline v)}},\frac{R}{\sqrt{\cosh(2\overline v)}}\right]$. Now fixing $r>0$ and using a localised heat kernel $\Phi^{\rho}$ supported in $B_R$, we use this $L^2$ estimate and the co-area formula to calculate the localised Gaussian density:
	\begin{align*}
	\lim_{i \rightarrow \infty} &\Theta^{\rho}(L^i,0,r) \, = \, \lim_{i \rightarrow \infty}\int_{L^i_{-r^2}} \Phi^{\rho} d\mathcal{H}^2\\
	&= \, \lim_{i \rightarrow \infty} \int_\delta^{\frac{R}{\sqrt{\cosh(2\overline v)}}} \Phi^{\rho}(\gamma^i, -r^2) 2\pi |\gamma||\gamma'| ds \, + \, \lim_{i \rightarrow \infty} \int_0^\delta \Phi^{\rho}(\gamma^i(s), -r^2) 2\pi |\gamma||\gamma'| ds \\
	&\leq \, \lim_{i \rightarrow \infty} \int_{\delta}^{\frac{R}{\sqrt{\cosh(2\overline v)}}} \Phi^{\rho}(\gamma^i, -r^2) 2\pi s \sqrt{\cosh(2v)} \sqrt{(1 + s^2(v')^2)\cosh(2v) + s(v')\sinh(2v)} ds \, + \, C\delta \\
	&\leq \, \int_0^{\frac{R}{\sqrt{\cosh(2\overline v)}}} 2\pi s\cosh(2\overline v) \,\Phi^{\rho}(s\cosh(\overline v) + is \sinh(\overline v), -r^2)\, ds \, + \, C\delta \,\\
	&= \int_0^R 2\pi \sigma \Phi^{\rho}(\sigma, -r^2) \, d\sigma \, + \, C\delta\\
	&= \int_{D_R} \Phi^{\rho}(\cdot, -r^2) d\mathcal{H}^2 \, + \, C\delta \quad = \quad 1+ C\delta,
	\end{align*}
	 for $D_R := \{(s\cos(\psi),s\sin(\psi)) \in \mathbb{C}^2 \,\,|\,\, s<R, \,\, \psi \in [0,2\pi] \}$, where the last line follows from the fact that $\Phi^{\rho}$ is normalised to integrate to $1$ over a plane. $\Theta^{\rho}(L^i,0,r)$ can therefore be made as close to 1 as desired, by choosing $\delta$ sufficiently small and $i$ sufficiently large.
	 
	 More generally, we are able to bound the density $\Theta^\rho \left( L^i,X,\frac{1}{\sqrt{2}}r \right) $ for all $(x_0,r_0) \in P \left( O,\frac{1}{\sqrt{2}}r \right) = B_r(O)\times \left(-\frac{1}{2}r^2,0 \right]$. Using the monotonicity formula (\ref{monotonicity}),
	\begin{align*}
	\Theta^{\rho} \left(L^i, (x_0,r_0), \frac{1}{\sqrt{2}}r \right) \, = \, \int_{L^i_{r_0-\frac{1}{2}r^2}}\Phi^{\rho}_{(x_0,r_0)}(\cdot,r_0-\tfrac{1}{2}r^2) d\mathcal{H}^2 \, \leq \, \int_{L^i_{-r^2}}\Phi^{\rho}_{(x_0,r_0)}(\cdot,-r^2) d\mathcal{H}^2,
	\end{align*}
	and by a very similar calculation to the above we can choose $i$ large so that this is less than $1 + \veps$. It follows by White's local regularity theorem that $|A|$ and its derivatives are bounded uniformly in the parabolic ball $P(O,\tfrac{r}{8})$. This is a contradiction, and so no singularity can occur.
\end{proof}

\subsubsection{Smooth Convergence to the Disc}

We now prove that the profile curve $\gamma$ converges smoothly in infinite time to the real axis.

\begin{theorem}
	Any solution to (\ref{eq-equiflow}) is immortal, and converges smoothly in infinite time to the real axis.
\end{theorem}

\begin{proof}
	The $C^1$ bound (\ref{eq-gamma'bound}) implies that our graphical mean curvature flow equation (\ref{eq-graphmcf}) is uniformly parabolic, and so Schauder estimates give bounds on all curvatures on any annulus. In order to apply Proposition \ref{prop-conv}, it is left to show that we have uniform curvature bounds near the origin - for this we use White regularity. Fix $r>0$, then for all $\delta$,
	\begin{align*}
	\int_{L_{t-r^2}} \Phi^\rho_{(x,t)} d\mathcal{H}^2 \, &= \, \int_{L_{t-r^2}\cap B_{\delta}}\Phi^\rho_{(x,t)}d\mathcal{H}^2 \, + \, \int_{L_{t-r^2}\cap B_{\delta}^c}\Phi^\rho_{(x,t)}d\mathcal{H}^2 \,
	 \leq \, \delta^2 C\, + \, \int_{L_{t-r^2}\cap B_{\delta}^c}\Phi^\rho_{(x,t)}d\mathcal{H}^2.
	\end{align*}
	Therefore for any $\veps$, we may take $\delta$ sufficiently small such that
	\[ \int_{L_{t-r^2}} \Phi^\rho_{(x,t)} d\mathcal{H}^2 \leq \int_{L_{t-r^2}\cap B_{\delta}^c}\Phi^\rho_{(x,t)}d\mathcal{H}^2 + \veps.\]
	By smooth convergence to the disc outside $B_{\veps}$, we may take $t$ sufficiently large such that the integral in the last line is less than $1$ (the localised kernel $\Phi^\rho$ has the property that it integrates to $1$ on a hyperplane). In general then, for any $\veps$ we may take $t$ sufficiently large such that
	\[ 	\int_{L_{t-r^2}} \Phi^\rho_{(x,t)} d\mathcal{H}^2 \, \, \leq \, 1 + \veps, \]
	locally uniformly in $x$ and $t$. But now White's regularity theorem gives us a uniform bound on $|A|^2$ and its higher derivatives. This implies that our flow converges smoothly to a special Lagrangian by Proposition \ref{prop-conv}, that must be equivariant and must pass through the origin. There is only one submanifold with these properties that also intersects the Lawlor neck - an equivariant disc - and so we are done.
\end{proof}

\subsection{The Clifford Torus}\label{sec-cliff}

Our second example concerns equivariant discs $L$ (profile curve $\gamma$) with boundary on the Clifford torus. The Lagrangian angle of the Clifford torus $\Sigma$ with profile curve $\sigma$ is given by
\[ \tilde\theta \, = \, \frac{\pi}{2} + 2\arg(\sigma), \]
and therefore the boundary condition of (\ref{IntroBVP}) becomes
\[ \theta \big|_{\p L} - 2\arg(\gamma) \, = \, -\alpha. \]
As before, we restrict to the $\alpha=0$ case, which corresponds to the profile curves meeting orthogonally at the boundary.

The Clifford torus is slightly more complicated to work with than the Lawlor neck, as it is not a static solution to MCF. However it is a self-similarly shrinking solution, with profile curve
\[ \sigma_t = \left(\sqrt{-4t}\cos(s), \sqrt{-4t}\sin(s) \right), \]
on the time interval $[t_0,0)$. It is then natural to perform the rescaling
\begin{align*}
\overline\Sigma_\tau &:= \frac{1}{\sqrt{-t}} \Sigma_t \big|_{t=-e^{-\tau}} \\
\implies \overline\sigma_\tau \, &= \, (2\cos(s), 2\sin(s))
\end{align*}
which is a static solution to the rescaled MCF equation
\[ \left(\frac{\p \overline F}{\p \tau}\right)^\perp \, = \, H \, + \, \frac{\overline F^\perp}{2},\]
on the time interval $[\tau_0, \infty) = [-\log(-t_0), \infty)$. Applying this rescaling also to our LMCF with boundary means we are working with a static boundary manifold, albeit with a different PDE problem.

In this section, we will prove that the rescaled flow is immortal and converges in infinite time to a flat equivariant disc. In terms of the original flow, this means that no singularity occurs before the final time $0$, and any sequence of parabolic rescalings centred at the singular spacetime point $(O,0)$ converges to a flat equivariant disc. This is a self-similarly shrinking solution to LMCF with boundary, so this result is analogous to the general result of ordinary MCF that Type I blowups are self-similarly shrinking solutions. 

Throughout this section we will work with both the rescaled flow, denoted $\overline L_\tau$ with profile curve $\overline \gamma_\tau$, and the original flow, denoted $L_t$ with profile curve $\gamma_t$. For reference, the rescaled flow for the profile curve is given by
\begin{equation} \label{eq-rescaledequi}
\left(\frac{\p \overline\gamma}{\p \tau}\right)^\perp = k - \frac{\overline\gamma^\perp}{|\overline \gamma|^2} + \frac{\overline\gamma^\perp}{2}.
\end{equation}

\begin{theorem}\label{thm-cliff}
	Let $\overline L_0: D \rightarrow \mathbb{C}$ be an $S^1$-equivariant Lagrangian embedding of a disc $D$, with boundary on the Clifford torus 
	\[\Sigma_{\operatorname{Cliff}} := \{2e^{i\phi}(\cos(\psi),\sin(\psi))\in \mathbb{C}^2 : \phi,\psi \in [0,2\pi)\},\]
	and let $\overline \gamma_0:[-2,2]: \mathbb{C}$ be its profile curve in $\mathbb{C}$. 
	Assume that its Lagragian angle $\theta_0$ satisfies
	\[ \theta_0(s) - 2\arg(\overline\gamma_0(s)) \, \in \, (-\tfrac{\pi}{2} + \veps, \, \tfrac{\pi}{2} - \veps)\]
	for some $\veps > 0$.
	Then there exists a unique, eternal solution to the rescaled LMCF problem:
	\begin{align}
	\begin{cases}
	\left(\frac{ \p}{\p \tau}\overline F(x,\tau)\right)^{NM} = H(x,\tau) \, + \, \frac{\overline F(x,\tau)^\perp}{2} & \text{for all }(x,\tau) \in D\times[\tau_0,\infty)\\
	\overline F(x,\tau_0)= \overline L_0(x)&\text{for all }x\in D\\
	\partial \overline L_\tau\subset \Sigma_{\operatorname{Cliff}} & \text{for all }\tau\in[\tau_0,\infty)\\
	\theta_\tau|_{\p \overline L_\tau} - 2\arg(\overline \gamma_0)=0 &\text{for all } (x,\tau)\in \p D \times[\tau_0,\infty),
	\end{cases} \label{eq-equiflowcliff}
	\end{align}
	which converges in smoothly in infinite time to a flat disc.
\end{theorem}
\begin{remark}
	Note that here, we demand the condition $\theta_0(s) - 2\arg(\gamma_0(s)) \, \in \, (-\tfrac{\pi}{2} + \veps, \, \tfrac{\pi}{2} - \veps)$ in place of the almost-calibrated condition of the Lawlor neck case. This is more natural, as not only is this always satisfied at the boundary, but it is also equivalent to graphicality in a radial parametrisation, as will be shown in the next section.
	
	If we work with a different boundary condition, $\alpha \neq 0$ (corresponding to a different fixed angle between the profile curves), numerical evidence suggests that we still have long-time existence, and the flow converges to a rotating soliton of the rescaled LMCF with boundary problem; see Figures \ref{fig-cliffordcase1} and \ref{fig-cliffordcase2}.
\end{remark}

\subsubsection{Radial Parametrisation}

We will work throughout with the radial parametrisation of the rescaled profile curve:
\begin{align} 
\overline\gamma:[-2,2]\rightarrow \mathbb{C}, \quad \overline\gamma(r) &:= re^{i\phi(r)} \notag \\
\implies \overline\gamma'(r) \, &= \, (1+ir\phi')e^{i\phi}  \notag \\
\implies \overline\gamma''(r) \, &= \, (-r(\phi')^2 +i(2\phi' + r\phi''))e^{i\phi}. \label{eq-paramsrad}
\end{align}
Writing $\nu := \frac{i\overline\gamma'}{|\overline\gamma'|}$, the mean curvature is given by:
\begin{align*}
H \, &= \, k - \frac{\overline\gamma^\perp}{|\overline\gamma|^2}
= \, \frac{(\overline\gamma'')^\perp}{|\overline\gamma'|^2} - \frac{\overline\gamma^\perp}{|\overline\gamma|^2} \\
&= \left( \frac{r\phi'' + r^2 (\phi')^3 + 2\phi'}{|\overline\gamma'|^3} \, + \, \frac{\phi'}{|\overline\gamma'|} \right)\nu,
\end{align*}
and therefore in this parametrisation, the problem (\ref{eq-equiflowcliff}) becomes
\begin{align} \label{eq-rescaledmcfpara}
\begin{cases}
r \frac{\p \phi}{\p t} \, = \, \frac{r\phi'' + r^2 (\phi')^3 + 2\phi'}{1 + r^2(\phi')^2} + \phi' - \frac{r^2 \phi'}{2} &\text{for } r \in [-2,2], \,\, \tau \geq \tau_0, \\
\phi(r,\tau) = \phi_0 &\text{for } r \in [-2,2],\\
\phi'(r,\tau) = 0 &\text{for } r \in \{-2,2\}, \,\, \tau \geq \tau_0.
\end{cases}
\end{align}
\begin{lemma} \label{lem-cliffbarriers}
	In the above parametrisation, the only static solutions to the rescaled LMCF with boundary (\ref{eq-equiflowcliff}) are straight lines through the origin, with $\phi = \phi_0$.
\end{lemma}
\begin{proof}
	Using (\ref{eq-rescaledmcfpara}),
	\begin{align*}
	H + \frac{\overline F^\perp}{2} = 0 \, &\iff \, r\phi'' + 3\phi' + 2r^2(\phi')^3 - \frac{r^2\phi'}{2} - \frac{r^4(\phi')^3}{2} = 0 \\
	&\iff \frac{d\lambda}{dr} + (\lambda + \lambda^3)\left(\frac{2}{r} - \frac{r}{2}\right) = 0
	\end{align*}
	away from $r=0$, for $\lambda = r\phi'$. This ODE, along with the boundary condition $\lambda = 0$, has the unique solution $\lambda=0$, which implies that our static solution is a straight line.
\end{proof}

\subsubsection{$C^1$-bounds on the Graph Function}

The important thing about this parametrisation is that our assumed condition on the Lagrangian angle corresponds to graphicality and gradient bounds for $\phi$.
\begin{lemma}
	Assume that $\overline F_\tau$ is a solution to (\ref{eq-equiflowcliff}) on $[\tau_0,T)$, such that at time $\tau_0$, 
	\begin{equation} \label{eq-anglecondition}
	\theta_t - 2\arg{(\overline \gamma_\tau)} \in \left( -\frac{\pi}{2} + \veps, \frac{\pi}{2} - \veps \right).
	\end{equation}
	Then for all $\tau \in [\tau_0,T)$:
	\begin{itemize}
		\item The condition (\ref{eq-anglecondition}) holds,
		\item The flow can be radially parametrised as $\overline\gamma_\tau(r) = re^{i\phi\tau(r)}$,
		\item In this parametrisation, there exists a constant $C_2$ such that $|r\phi'_\tau| \leq C_2$. Therefore $|\overline\gamma'|$ is uniformly bounded, and $\phi\tau'$ is uniformly bounded on any annulus centred at the origin.
	\end{itemize}
\end{lemma}
\begin{proof}
	If we parametrise the initial profile curve $\overline\gamma_{0}$ by arclength, then it may be written in polar coordinates as
	\begin{align}
	\overline\gamma_{0}(s) \, &= \,  r(s)e^{i\phi(s)}, \quad \quad
	\overline\gamma_{0}'(s) \, = \, (r' + ir\phi')e^{i\phi}. \label{eq-paramsarc}
	\end{align}
	Therefore the Lagrangian angle of $\gamma_0$ may be expressed as
	\[ \theta(s) = 2\phi + \tan^{-1}\left(\frac{r\phi'}{r'}\right). \]
	Note that at the origin, we must have $r'>0$. Since $|\overline\gamma'| = \sqrt{(r')^2 + r^2 (\phi')^2}=1$, $|r\phi'|$ and $|r'|$ are bounded from above, and so (\ref{eq-anglecondition}) corresponds to a positive lower bound on $r'$. 
	This allows us to reparametrise as $\overline\gamma(r) = re^{i\phi}$, and in this parametrisation, 
	\[ \theta(r) = 2\phi + \tan^{-1}(r\phi'),\]
	therefore the condition (\ref{eq-anglecondition}) corresponds to a uniform upper bound on $|r\phi'|$.\\
	
	It is left to prove that (\ref{eq-anglecondition}) is preserved; we start by calculating the evolution equation of $\theta - 2\phi$. Working with the arclength parametrisation of the original unrescaled flow, $\gamma(s) = r(s)e^{i\phi(s)}$, the metric and Laplacian on the manifold are given by
	\begin{align*}
	g \, = \, ds\otimes ds + r^2 d\beta \otimes d\beta, \quad \quad \Delta f \, = \, \frac{1}{|g|}\p_i \left( |g| g^{ij} \p_j f \right) = \frac{\p^2 f}{\p s^2} \, + \, \frac{1}{r^2} \frac{\p^2 f}{\p \beta^2} \, + \,  \frac{\ip{\gamma'}{\gamma}}{r^2} \frac{\p f}{\p s},
	\end{align*}
	where $\beta$ is the coordinate of the $S^1$-equivariance. If $f$ is an equivariant function, as $\theta$ and $\phi$ both are, then the middle term vanishes. Now, writing $\nu := i\gamma'$, it follows from (\ref{eq-paramsarc}) that 
	\begin{align*}
	\frac{\p \phi}{\p s} \, &= \, -\frac{\ip{\gamma}{\nu}}{r^2}, \quad \quad
	\frac{\p^2 \phi}{\p s^2} \, = \, -\frac{\ip{\gamma}{i\gamma''}}{r^2} + 2\frac{\ip{\gamma'}{\gamma} \ip{\gamma}{\nu}}{r^4}
	\, = \, \frac{\ip{i\gamma}{k}}{r^2} + 2\frac{\ip{\gamma'}{ \gamma}}{r^2}\frac{\ip{\gamma}{\nu}}{r^2} ,
	\end{align*}
	and using the standard equivariant MCF equation,
	\begin{align*} \frac{\p \gamma}{\p t} \, = \, k - \frac{\gamma^\perp}{r^2},  \quad \implies \quad
	\frac{\p \phi}{\p t} \, &= \, \ip{\frac{i\gamma}{r^2}}{\frac{\p \gamma}{\p t}} \, = \, \frac{\ip{i\gamma}{k}}{r^2} - \frac{\ip{\gamma'}{ \gamma}}{r^2}\frac{\ip{\gamma}{\nu}}{r^2}.
	\end{align*}
	Additionally, under this flow the Lagrangian angle satisfies the heat equation
	\[\left(\frac{\p}{\p t} - \Delta \right) \theta \, = \, 0. \]
	Putting this all together, we arrive at the evolution equation:
	\begin{align*} 
	\left( \frac{\p}{\p t} - \Delta \right)(\theta - 2\phi) \, &= \, 2\left( -\frac{\p}{\p t} + \frac{\p^2}{\p s^2} + \frac{\ip{\gamma'}{\gamma}}{r^2} \frac{\p}{\p s} \right)\phi
	\, = \, 4\frac{\ip{\gamma'}{\gamma}}{r^2}\frac{\ip{\gamma}{\nu}}{r^2}.
	\end{align*}
	Now, remembering that $\theta - 2\phi$ = $\arg(\gamma') - \arg(\gamma)$, it follows that
	\begin{align}
	\cos(\theta- 2\phi) \, &= \cos(\arg(\gamma') - \arg(\gamma))\, = \, \frac{\ip{\gamma'}{\gamma}}{r} , \notag\\
	\sin(\theta - 2\phi) \, &= \, \cos\left(\arg(\gamma') - \arg(\gamma) - \frac{\pi}{2}\right) \, = \, -\frac{\ip{\gamma}{\nu}}{r}, \notag \\
	\implies \left( \frac{\p}{\p t} - \Delta \right)(\theta - 2\phi) \, &= \, -2\frac{\sin \left(2(\theta-2\phi)\right)}{r^2}. \label{eq-thetaalphaevol}
	\end{align}
	Therefore, 
	\begin{align*}
	\left( \frac{\p}{\p t} - \Delta \right)\sin(\theta - 2\phi) \,&= \, \cos \left(\theta-2\phi\right) \left( \frac{\p}{\p t} - \Delta \right)(\theta-2\phi) + \sin \left(\theta-2\phi\right) \ip{\frac{\p (\theta - 2\phi)}{\p s}}{\frac{\p (\theta - 2\phi)}{\p s}} \\
	&= -\frac{4}{r^2}\sin(\theta-2\phi)\cos^2(\theta-2\phi) + \frac{ \sin(\theta-2\phi)}{\cos^2(\theta-2\phi)}\left| \nabla \sin(\theta-2\phi) \right|^2.
	\end{align*}
	Now for a contradiction, assume that at some point $p \in \gamma_t$, we have an increasing maximum of $\theta-2\phi$ (and of $\sin(\theta-2\phi)$) that is larger than $\frac{\pi}{2}-\veps$. Since this function is zero on the boundary and at the origin, it must occur at some interior point away from the origin. Then at this point, it is valid to parametrise by arclength and use standard (normal) mean curvature flow, so that the above calculation is valid. The weak maximum principle, applied in the cases of a positive maximum or negative minimum, then provides a contradiction.
\end{proof}
Finally, using simple barriers we also obtain uniform $C^0$ estimates on the function $\phi$.
\begin{lemma}\label{lem-cliffbarrierargument}
	Let $\overline\gamma$ be a radially parametrised solution to (\ref{eq-equiflowcliff}) on the time interval $[t_0,T)$, which satisfies $\phi_{t_1} \in [\phi_-, \phi_+]$, $\theta_{t_1} \in [\theta_-, \theta_+]$ for some $t_1 \in [t_0,T)$. Define 
	\[ A_- := \min \left \{\frac{\theta_-}{2}, \phi_- \right \}, \quad \quad  A_+ := \max \left \{\frac{\theta_+}{2}, \phi_+\right \}. \] Then for all $t \in [t_1,T)$, $\phi_{t} \in [A_-, A_+]$.
	\end{lemma}
\begin{proof}
	We only prove that $\phi_t \leq A_+$, since the $A_-$ case is identical. For a contradiction, assume that there exist $\delta$ and a first time $t_{\delta} \in (t_1, T)$ such that
	\[ \max_{L_{t_\delta}} \, = \, A_\delta := A_+ + \delta. \]
	Then using the radial parametrisation, if this maximum is achieved on $[-2,2]\setminus \{0\}$, we may use the strong parabolic maximum principle applied to the boundary value problem (\ref{eq-rescaledmcfpara}), comparing with the static solution $\tilde\phi \equiv A_\delta$. This implies that locally in space and time $\phi \equiv A_\delta$, which is a contradiction.
	
	On the other hand, if this maximum is achieved at the origin $r=0$, then since $\theta - 2\phi = 0$ at this point, $\theta_{t_\delta}(0)=2\phi_{t_\delta}(0) = 2A_\delta$, which is larger than the maximum of $\theta_{t_1}$. Since $\theta$ satisfies a heat equation on the abstract disc, it follows by the parabolic maximum principle and the fact that $\theta-2\phi = 0$ on the boundary that we must have 
	\[ \theta_{t_\delta}(-2) = \theta_{t_\delta}(2) = 2A_\delta \, \implies \phi_{t_\delta}(-2) = \phi_{t_\delta}(2) = A_\delta. \]
	But now as before we may apply the maximum principle at the boundary to $\phi$ to derive a contradiction.
\end{proof}
\subsubsection{Long-Time Existence}
We now prove long-time existence for our rescaled flow, in a very similar way to the Lawlor neck case.
\begin{lemma}
	A finite time singularity for a solution of (\ref{eq-equiflowcliff}) cannot occur.
\end{lemma}
\begin{proof}
Note that a finite-time singularity of (\ref{eq-equiflowcliff}) corresponds to a singularity of the unrescaled flow before time $0$.

Working with the rescaled flow, we have shown that it is graphical and that the graph function $\phi$ satisfies the equation (\ref{eq-rescaledmcfpara}), which is uniformly parabolic away from the origin by the $C^1$ bounds of the last section. Therefore we have uniform bounds on all derivatives by parabolic Schauder estimates, and no singularity can occur away from the origin.

Just as before, we must deal with the origin separately. Assuming that a singularity of the original flow $L_t$ occurs before the final time $0$, the image of $\p L_t$ under any sequence of rescalings around this singularity will diverge to infinity, just as with the Lawlor neck (since at the time of the singularity, the Clifford torus is outside a neighbourhood of the origin). Therefore Lemma \ref{lem-nev} applies, and it follows that
\begin{align*}
\int_{\gamma_t^i \cap ( B_R\setminus B_\delta)} |\gamma^\perp|^2 d\mathcal{H}^1 \, = \, 2\int_\delta^R \frac{r^4 (\phi')^2}{|\gamma'|}dr \, \geq \, \frac{2\delta^4}{C_2} \int_\delta^R (\phi')^2 dr \rightarrow 0.
\end{align*}
In exactly the same way as in the proof of Lemma \ref{lem-nofintimesing}, this estimate gives us bounds on the densities, and White regularity implies smooth convergence of the rescalings. This is a contradiction to the assumption of singularity formation at $(O,0)$.
\end{proof}
\subsubsection{Subsequential Convergence to the Disc}
We now prove subsequential convergence to the disc, working with the original flow throughout. Take a sequence of rescalings $ L^i_t$ around the spacetime point $(O,0)$ with factors $\lambda_i \rightarrow \infty$. We may use the graphicality and smooth estimates from Schauder theory away from the origin to conclude that, subsequentially, the profile curves $\gamma^i_t$ converge to a limiting smooth graph on $A\times[a,b]$, where $A$ is any annulus centred at the origin. A diagonal argument gives a subsequence converging locally smoothly away from the origin to a limiting flow $ \gamma^\infty_t$, with limiting angle function $\phi^\infty_t$ well defined everywhere but the origin.

Using the boundary version of Huisken's monotonicity formula (\ref{monotonicity}) with $f = (\theta-2\phi)^2$, using the evolution equation (\ref{eq-thetaalphaevol}) and noting that $f=0$ and $\nabla f=0$ on the boundary gives the monotonicity formula:
\begin{align}
&\frac{\p}{\p t} \int_{L^i_t} f \Phi d\mathcal{H}^2 \notag \\\, &= \, \int_{ L^i_t} \Phi\left( \left(\frac{\p}{\p t} - \Delta\right)f - f\left| H - \frac{x^\perp}{2t}\right|^2 \right) d\mathcal{H}^2 \,
+ \, \int_{\p  L^i_t}\Phi \left\langle f\frac{x}{2t} + \nabla f, \nu \right\rangle d\mathcal{H}^{1}\notag\\ 
&= \int_{L^i_t} \Phi \left( 2(\theta-2\phi) \left( \frac{\p}{\p t}-\Delta\right)(\theta-2\phi) - 2\left(\frac{\p}{\p s}(\theta-2\phi)\right)^2 - (\theta - 2\phi)^2\left| H - \frac{x^\perp}{2t}\right|^2 \right) d\mathcal{H}^2\notag \\
&= \int_{L^i_t} \Phi\left( -\frac{4}{r^2}(\theta-2\phi)\sin(2(\theta-2\phi)) - 2\left(\frac{\p}{\p s}(\theta-2\phi)\right)^2 - (\theta - 2\phi)^2\left| H - \frac{x^\perp}{2t}\right|^2 \right) d\mathcal{H}^2. \label{eq-cliffmonotonicity}
\end{align}
Therefore, choosing $0<a<b$,
\begin{align*}
\lim_{i\rightarrow \infty} \int_a^b \int_{L^i_t}\Phi (\theta - 2\phi)^2\left| H - \frac{x^\perp}{2t}\right|^2 d\mathcal{H}^2 dt \, &\leq \, \lim_{i\rightarrow \infty} \left( \int_{L^i_a} f\Phi d\mathcal{H}^n - \int_{L^i_b} f\Phi d\mathcal{H}^2 \right) \\
&= \, \lim_{i\rightarrow \infty}\left( \int_{L_{(\lambda^i)^{-2}a}} f\Phi d\mathcal{H}^n - \int_{L_{(\lambda^i)^{-2}b}} f\Phi d\mathcal{H}^2 \right)\\
&= \, 0.
\end{align*}
This implies (by the locally smooth convergence) that $(\theta - 2\phi)^2\left| H - \frac{x^\perp}{2t}\right|^2 \equiv 0$ for the limiting manifold $L^\infty_t$, for any $t \in \mathbb{R}$. But if on an open subset we have $\theta-2\phi\equiv 0$, then the subset must be a part of a straight line through the origin. Therefore on this subset we also have $ \left| H - \frac{x^\perp}{2t} \right| \equiv 0$, and so $\gamma^\infty$ is a self-shrinker. By Lemma \ref{lem-cliffbarriers} the only option is a straight line through the origin; therefore $\phi^\infty = A$ for some constant $A \in \mathbb{R}$. Additionally, since we have smooth convergence on any annulus, we have the integral estimate
\begin{equation} \int_{L_t^i}\Phi (\theta - 2\phi)^2 d\mathcal{H}^2 \, \rightarrow \, 0. \quad \quad \mbox{as} \quad i \rightarrow \infty. \label{eq-theta2alphal2}
\end{equation}
This convergence of the rescalings corresponds to subsequential convergence in the rescaled flow. Taking any sequence $\tau_i$, and choosing $\lambda_i :=  e^{\frac{\tau_i}{2}}$:
\[ \overline L_{\tau_i} \, = \,e^{\frac{\tau_i}{2}} L_{-e^{-\tau_i}} \, = \, \lambda_i L_{-\lambda_i^{-2}} \, = \, L^i_{-1}.  \]
By the work above we know that, up to a subsequence, this converges smoothly away from the origin to a disc.

\subsubsection{Smooth Convergence to the Disc}

We have proven subsequential convergence to the disc, but we could still have different subsequences converging to different discs, and we also haven't shown that the curvature remains boudned at the origin. To solve these problems, we will demonstrate uniform curvature estimates via a Type II blowup argument. 

Assume that the curvature of the rescaled flow $|A|$ diverges to infinity as $\tau \rightarrow \infty$. Then we may find a sequence $\tau_i$ such that $\max_{\overline L_{\tau_i}}|\overline A_{\tau_i}| \rightarrow \infty$ as $i \rightarrow \infty$. In the unrescaled flow, this sequence corresponds to a sequence of times $t_i = -e^{-\tau_i}$, such that
\[ \sqrt{-2t_i}\,\max_{ L_{t_i}}| A_{t_i}| \rightarrow \infty;  \]
i.e. the singularity is a Type II singularity.

Passing to a subsequence we may ensure that the manifolds $\overline L_{\tau_i}$ converge smoothly to a disc on an annulus by the work of the previous section - therefore the curvature blowup must be uniformly away from the boundary. By standard theory of Type II blowups, we also know that we may choose a sequence of points $x_i$ such that the sequence 
\[ \hat L_t^{(x_i,t_i)} := A_i \left(L_{t_i + A_i^{-2}t} - x_i \right)  \]
converges locally smoothly to a limiting flow $\hat L^\infty_t$, where $A_i := \max_{L_{t_i}}| A_{t_i}|$. We may pick these points in $\mathbb{C} \times \{0\}$, and define the rescaled profile curve $\hat \gamma^i_t$ in the same say as above by considering $x_i$ to be an element of $\mathbb{C}$.

We now prove locally uniform convergence of $\theta-2\phi$ to 0 for the Type II rescalings $\hat L^{(x_i,t_i)}_t$. The argument is identical to that given in \cite{Wood}, in which more details are given.

\begin{lemma}
	For any bounded parabolic region $\Omega \times I \subset \mathbb{C}^2 \times \mathbb{R}$,
	\begin{equation} \theta-2\phi \rightarrow 0 \quad \text{as }\quad i \rightarrow \infty, \quad \mbox{uniformly in } \Omega\times I.  \label{eq-typeIIangle}
	\end{equation}
	Explicitly, for any $\veps > 0$, there exists $N \in \mathbb{N}$ such that for any $t \in I$, $\chi \in \Omega \cap \hat L^\infty_t$, and any sequence $\chi_i \in \Omega \cap \hat L^{(x_i,t_i)}_t$ converging to $\chi$,
	\[\theta^i_t - 2\phi^i_t \leq \veps, \]
	where $\phi^i_t(p)$ is the angle of the point $\hat \gamma^i_t(p)$ in the rescaled profile curve, relative to the image of the origin under the rescaling, $-A_ix_i$.
\end{lemma}

\begin{proof} Choosing
	\[ \lambda_i \, := \, \frac{1}{2}\min \left \{\frac{1}{\sqrt[4]{-t_i}}, \frac{1}{\sqrt{x_i}}, \sqrt{A_i} \right \}, \]
it is then possible to pick an $N$ such that for any $i>N$ and $\tau \in I$,
\[ (-t_iA_i^2)(1 + t_i^{-1}\lambda_i^{-2}) \geq \tau. \]
 It follows that
 \begin{align*}
 |\theta^i_\tau(\chi_i)-2\phi^i_\tau|^2 \, &= \, 
 \int_{\hat L_{\tau}^{(x_i,t_i)}} (\theta - 2\phi)^2 \, \Phi_{(\chi_i,\tau)} d\mathcal{H}^2 \notag \\
 &\leq \, \int_{\hat L_{(-t_iA_i^2)(1 + t_i^{-1}\lambda_i^{-2})}^{(x_i,t_i)}} (\theta - 2\phi)^2 \, \Phi_{(\chi_i,\tau)} d\mathcal{H}^2\\
 & = \, \int_{L_{-1}^{i}} (\theta - 2\phi)^2 \, \Phi_{\left( \lambda_i\left( A_i^{-1}\chi_i + x_i \right),\lambda_i^2\left( A_i^{-2}\tau + t_i\right)\right) } d\mathcal{H}^2,
 \end{align*}
where for the first inequality we use Huisken monotonicity (\ref{eq-cliffmonotonicity}), and in the second we use invariance of the kernel $\Phi$ to equate the integral over the Type II rescaling with an integral over the type I rescaling $L^i_{-1}$, centred at $(0,0)$ and with rescaling factor $\lambda_i$.
Then, since $\left( \lambda_i\left( A_i^{-1}\chi_i + x_i \right),\lambda_i^2\left( A_i^{-2}\tau + t_i\right)\right) \rightarrow (0,0)$ uniformly in $\Omega \times I$, and by the $L^2$ convergence (\ref{eq-theta2alphal2}), we may find $\tilde N \geq N$ such that for $i \geq \tilde N$,
\begin{align*}
|\theta^i_\tau(\chi_i)-2\phi^i_\tau|^2 \, &\leq \, \int_{L_{-1}^{i}} (\theta - 2\phi)^2 \, \Phi_{(0,0)} d\mathcal{H}^2  \, + \, \tfrac{\veps}{2}\, \leq \, \veps.
\end{align*}
\end{proof}
This lemma implies that the limiting profile curve $\hat \gamma^\infty_t$ is a straight line. However, this is a contradiction, as the Type II blowup satisfies $\max|\hat A|=1$ by construction.\\

Therefore, the rescaled flow $\overline L_\tau$ satisfies uniform curvature bounds, and so the subsequential convergence of $\overline L_{\tau_i}$ to a disc is in fact everywhere smooth. In particular, on passing to a subsequence their Lagrangian angles converge smoothly to a constant, as do their angle functions $\phi$. We may now apply Lemma \ref{lem-cliffbarrierargument} to conclude that the flow converges smoothly in $\tau$ to a Lagrangian disc, which proves Theorem \ref{thm-cliff}.

\section{Short-Time Existence}\label{STEsec}
\subsection{Statement}
In this section we prove the following theorem:

\begin{theorem}\label{STE}
	Let $\Sigma_t$ be a smooth oriented Lagrangian mean curvature flow, and let $M_0$ be an oriented smooth compact Lagrangian with boundary satisfying the boundary conditions in (\ref{MCFBC}). Then there exists a $T\in(0,\infty)$ such that a unique solution of (\ref{MCFBC}) exists for $t\in[0,T)$, and this solution is smooth for $t>0$. Furthermore, if we assume this $T$ is maximal, then at $T$ at least one of the following hold:%, ordered by stupidity:
	\begin{enumerate}[label=\alph*)]
		\item \textbf{Boundary flow curvature singularity:} $\sup_{\Sigma_t} |\IIS|^2 \ra \infty$ as $t\ra T$.
		\item \textbf{Flowing curvature singularity:} $\sup_{M_t} |\II|^2 \ra \infty$ as $t\ra T$.
		\item \textbf{Boundary injectivity singularity:} The boundary injectivity radius of $\partial M_t$ in $M_t$ converges to zero as $t\ra T$.
	\end{enumerate}
\end{theorem}

\subsection{Diffeomorphism onto $NM_0$}\label{SectionDiffeo}
We require a diffeomorphism to pull back the mean curvature flow equations to a quasilinear parabolic equation on a time dependent section of the normal bundle of $M_0$.
\begin{proposition}\label{diffeoprop}
	Suppose that $\Sigma_t$ is a smooth flow of $n$-manifolds in $\CY$ and $M_0$ is a smooth $n$-manifold with boundary $\partial M_0\subset \Sigma$ satisfying the boundary conditions (\ref{MCFBC}). Then there exist constants $0<C_{Y}^u, T_Y$ and a mapping $Y:NM_0\times[0,T_Y)\ra \CY$ such that 
	\begin{enumerate}[label=\alph*)]
		\item If $u_0$ denotes the zero section of $NM_0$, then $Y(u_0,0) = M_0$.
		\item Let $K=\{u \in N M_0: |u|<C_Y^u\}$. Then $Y(\cdot,t)$ restricted to $K$ is a local diffeomorphism onto its image, for all $t \in [0,T_Y)$.
		\item $Y$ is smooth on $K\times[0,T_Y)$.
		\item Near any $p\in \partial M_0$, locally there exists a time independent vector field $\nu_0\in N\partial M_0\cap NM_0$ such that $Y(\l\nu_0, t)\in \Sigma_t$ for all $\l\in(-C_Y^u, C_Y^u)$, $t\in[0,T_Y)$. We may assume that $|DY(\nu_0, 0)|=1$ everywhere.\label{nu0property}
	\end{enumerate}
\end{proposition}
\begin{proof}
	The boundary conditions immediately ensure that such a map exists for time $t=0$, that is to say there exists $Y_0:NM_0 \to \mathcal Y$ with the given properties. Since $\Sigma_t$ is Lagrangian, we can find Weinstein neighbourhoods of $\Sigma_t$, i.e. symplectomorphisms $\beta_t: V_t \to W_t$, where $V_t$ is a tubular neighbourhood of $\Sigma_t$ in $\mathcal Y$ and $W_t$ is a tubular neighbourhood of the zero section of $T^* \Sigma$ with the standard symplectic structure. Since $\Sigma_t$ has bounded geometry for sufficiently small $t$, the size of these neighbourhoods does not degenerate, and we can restrict each $\beta_t$ to some uniform neighbourhood $V \subset \cap_t V_t$. For a sufficiently small collar region $C \subset NM_0$ of the boundary $N\partial M_0 \cap NM_0$, $Y_0(C) \subset V$. Define $Y(v,t) = \beta_t^{-1}(\beta_0(Y_0(v)))$, for $v \in C$. Note that since $Y_0(\nu_0)$ is tangent to $\Sigma_0$ by assumption and $\beta_t^{-1}\circ \beta_0$ maps $\Sigma_0$ to $\Sigma_t$, $Y(\nu_0,t)$ is tangent to $\Sigma_t$. Extend $Y$ to a map from $NM_0$ by interpolating with the standard geodesic embedding of the normal bundle away from the boundary by some suitable cut-off function. 
\end{proof}

\subsection{Mean Curvature Flow as a Flow of Sections}\label{MCFassection}
We now write mean curvature flow in terms of a time dependent section of the normal bundle $u \in \Gamma(NM_0)$. Specifically, we use the parametrisation $X(x,t):=Y(u(x,t))$ and consider the PDE given by
\[\left(\ddt{{X}}\right)^{NM} = H\ .\]
In what follows, we denote by $\np$ the induced normal connection on $NM_0$, and write $\mu_0$ for the outward pointing unit vector tangent to $M_0$ and normal to $\partial M_0$.

\begin{proposition}\label{QLsystem}
	Suppose that we have time dependent diffeomorphisms $Y:NM_0\times[0,T) \ra \CY$ as in Proposition \ref{diffeoprop}. Then there exist constants $c_D, c_T>0$ such that our boundary value problem (\ref{MCFBC}) starting from $M_0$ for $t\in [0,c_T)$ is equivalent to finding a time dependent section of $u\in \Gamma(NM_0\times[0,c_T))$ satisfying
	\begin{equation}
	\begin{cases}
	\pard{u}t = \hatG^{ij}(x,t, u,\np u) \np^2_{ij} u + B(x,t, u, \np u) & \text{ on }M_0\times[0,c_T)\\
	u - \ip{u}{\Voo}_0\Voo=0 & \text{ on }\partial M_0\times[0,c_T)\\
	\ip{\np_{-\mu_0} u}{W(x,t,u,\np_\partial u)}_Y = R(x,t,u,\np_\partial u) & \text{ on }\partial M_0\times[0,c_T)\\
	u(\cdot, 0)=0&\label{parabolicsystembundle}
	\end{cases}
	\end{equation}
	where $W$ and $R$ depend on $\alpha, x, t, u$ and $\np_\partial u$ (where $\np_\partial u$ represents dependence on any derivative in directions tangential to $\partial M_0$, but not on $\np_{\mu_0} u$), $\ip{\cdot}{\cdot}_Y$ is an inner product depending on $x$ and $t$ and $\ip{\cdot}{\cdot}_0=\ip{\cdot}{\cdot}_Y|_{t=0}$.
	
	Furthermore, if
	\[|u|<c_D \quad\text{and}\quad |\np u|<c_D,\]
	then all coefficients depend smoothly on their entries, $\hatG^{ij}$ is uniformly positive definite and we have the uniform obliqueness condition
	\[\ip{\Voo}{W}_Y\geq \frac{1}{2}\cos\a\ .\]
\end{proposition}
\begin{proof}
	We consider \eqref{MCFBC} for $t<T_Y$, where $T_Y$ is as defined in Proposition \ref{diffeoprop}. For this proof, we write $\n$, $g$ and $\langle \cdot,\cdot \rangle$ for the induced connection, metric and inner product on $NM_0$, and write $\np$ for the induced normal connection. We define the time dependent metric $g^Y=Y^* \ov{g}$ on $NM_0$ to be the pullback of the metric on $\CY$ by the mapping $Y$ at time $t$. We also denote the associated inner products $\ip{\cdot}{\cdot}_Y$, and $\ip{\cdot}{\cdot}_0 := \ip{\cdot}{\cdot}_Y\big|_{t=0}$ (which we note is not the same as  $\ip{\cdot}{\cdot}$). We will work entirely in $NM_0$ with the pulled back metric $g^Y$ and the content of this proof is the calculation of the equation for $u$ induced by the reparametrised mean curvature flow equations.
	
	We write $\hat{X}:M_0\times[0,\delta)\ra NM_0$ given by $\hat{X}(p,t)=(p,u(p,t))$. We have
	\[\pard{\hat{X}}{x^i} = \pard{}{x^i}-\ip{u}{\up{0}\II_i^k}\pard{}{x^k}+ \np_iu\ ,\]
	(where $\up{0}\II$ is the second fundamental form of $M_0$) and so the induced metric on $M_t$ is
	\begin{flalign*}
	\hatG_{ij} &= \left(\delta_i^k-\ip{u}{\up{0}\II_i^k}\right)g^Y_{x^kx^l}\left(\delta_j^l-\ip{u}{\up{0}\II_j^l}\right)+ \left(\delta_i^k-\ip{u}{\up{0}\II_i^k}\right)g^Y_{x^ky^l}\np_j u^l\\
	&\qquad + \np_iu^k g^Y_{y^kx^l}\left(\delta_j^l-\ip{u}{\up{0}\II_j^l}\right)+\np_iu^k\np_ju^l g^Y_{y^ky^l},
	\end{flalign*}
	where we are taking standard coordinates on $NM_0$ so that the $\pard{}{x^i}$ are tangent vectors and $\pard{}{y^i}$ are normal vectors. At $t=0$, $u\equiv 0$ and so at this time $\hatG_{ij} = g^{M_0}_{ij}$. Therefore, by continuity there exist $0<c_D$, $0<c_T<T_Y$ such that if 
	\[|u|<c_D,\qquad  |\np u|<c_D, \qquad 0\leq t<c_T\ ,\]
	then $\hatG_{ij}>\frac 1 2 g^{M_0}_{ij}$, i.e. $\hatG_{ij}$ is uniformly positive definite. Similarly, for sufficiently small $c_D$, $c_T$ we may assume that
	\begin{equation}\sigma_{ij}:=\ip{\left(\pard{}{y^i}\right)^{NM_t}}{\left(\pard{}{y^j}\right)^{NM_t}}_{Y}>\frac{1}{2}g_{y^iy^j}^{NM_0}\ ,\label{goodnormal}
	\end{equation}
	where $v^{NM_t}$ indicates the normal part of the vector $v$ \emph{with respect to $g^Y$} on $NM_0$.
	Calculating with respect to $\n$, and denoting the Christoffel symbols of this connection by $\Gamma$, we have that 
	\begin{flalign*}
	\n_{\pard{\hat{X}}{x^i}}\pard{\hat{X}}{x^j}-\Gamma_{ij}^k\pard{\hat{X}}{x^k}%&= \widetilde{\ov\n}_\pard{}{x_j}\pard{}{x^i}-\ip{\np_ju}{\up{0}\II_i^k}\pard{}{x^k}-\ip{u}{\up{0}\n_j^\perp\up{0}\II_i^k+\II(\up{0}\n_ji,p)\up{0}g^{pk}}\pard{}{x^k}-\ip{u}{\II_i^k}\II_{jk}\\
	%&+ \np_j\np_iu +\np_{\up{0}\n_ij} u-\ip{\np_iu}{\II_j^k}\pard{}{x^k}-\up{0}\Gamma_{ij}^k\pard{\hat{X}}{x^k}\\
	%&=\up{0}\II_{ij}-\ip{\up{0}\np_ju}{\up{0}\II_i^k}\pard{}{x^k}-\ip{u}{\up{0}\n_j^\perp\up{0}\II_i^k}\pard{}{x^k}-\ip{u}{\II_i^k}\II_{jk}\\
	%&+ \np_j\np_iu -\ip{\np_iu}{\II_j^k}\pard{}{x^k}\\
	&=\np^2_{ij}u-\ip{\np_ju}{\up{0}\II_i^k}\pard{}{x^k}-\ip{\np_iu}{\up{0}\II_j^k}\pard{}{x^k}\\
	&\qquad-\ip{u}{\np_j\up{0}\II_i^k}\pard{}{x^k}-\ip{u}{\up{0}\II_i^k}\up{0}\II_{jk}+\up{0}\II_{ij}.
	\end{flalign*}
	The difference between two connections is tensorial, and so we have that there exists a smooth time dependent tensor $T$ such that
	\begin{align*}
	\left(\n_{\pard{\hat{X}}{x^i}}\pard{\hat{X}}{x^j}-\Gamma_{ij}^k\pard{\hat{X}}{x^k}\right) &- \left(\up{Y}\ov \n_{\pard{\hat{X}}{x^i}}\pard{\hat{X}}{x^j}-\up{Y}\Gamma_{ij}^k\pard{\hat{X}}{x^k}\right)=T\left(\pard{\hat{X}}{x^i},\pard{\hat{X}}{x^j}\right)\ .
	\end{align*}
	We immediately see that
	\begin{align}\label{Curvinu}
	\up{M_t}\II_{ij} = \left(\np^2_{ij} u\right)^{NM_t} + \widetilde{B}_{ij}
	\end{align}
	where $\widetilde{B}_{ij}$ is a tensor depending on $x$, $t$, $u$, $\np u$. 
	
	Finally we have that reparametrised mean curvature flow is given by
	\[\left(\pard{u}{t}\right)^{NM_t} = \hatG^{ij}\left(\np^2_{ij} u\right)^{NM_t} + \hatG^{ij}\widetilde{B}_{ij}\]
	and so, using (\ref{goodnormal}), we have the claimed result:
	\[\pard{u}t = \hatG^{ij} \np^2_{ij} u + B(x, t, u, \np u).\]
	
	We now do the same for the boundary conditions. We recall that we have the normal vector field $\Voo\in N\partial M_0$, and write $\mu_0$ for the  outward pointing unit vector tangent to $M_0$ and normal to $\partial M_0$. By the construction of $Y$, $\partial M_t\subset \Sigma_t$ is equivalent to $u(p,t) = \l(p,t) \Voo$ for all $(p,t)\in\partial M_0\times[0,\delta_T)$. With respect to the metric $g_0:=g^Y|_{t=0}$ (and using property \ref{nu0property} in Proposition \ref{diffeoprop}) we have 
	\[0=u - \ip{\Voo}{u}_0\Voo\ .\]
	
	We define $\check{\mu} = \mu_0^i\pard{\hat{X}}{x^i}$, which in standard boundary coordinates gives $\check{\mu} = -\pard{\hat{X}}{x^n}$. We then define vectors $\nu$ and $\mu$ in the following way:
	\begin{align*}\widetilde{\nu} = \Voo - \ip{\Voo}{\hat{X}_J}_Yg^{IJ}_\partial\hat{X}_I \ , \qquad \widetilde{\mu} = \check\mu - \ip{\check\mu}{\hat{X}_J}_Yg^{IJ}_\partial \hat{X}_I\ ,\qquad\nu = |\widetilde{\nu}|_Y^{-1}\widetilde{\nu}, \qquad \mu = |\widetilde{\mu}|_Y^{-1}\widetilde{\mu}.
	\end{align*}
	 Here, as usual, $1\leq I \leq n-1$ are assumed to be local coordinates of the boundary of $\partial M_0$ and $g_{IJ}^\partial$ is the induced metric on $\partial M_t$ in these coordinates. Note that $\mu$ and $\nu$ correspond to the notation of Section \ref{BoundaryCond}. We now note that $g_{IJ}^\partial = \ip{{X}_I}{{X}_J} = \ip{\hat{X}_I}{\hat{X}_J}_Y$, which may be written explicitly (as with $\hatG$ above) as a function of $x$, $t$, $u$ and $\np_Iu$ but not $\np_{\mu_0} u$. We now rewrite the Neumann boundary condition
	 \[\cos \a \ip{\nu}{\mu} - \sin \a \ip{J\nu}{\mu}=0\]
	 in terms of $\hat X$. Denoting by $J$ the pulled back complex structure from $\mathcal{Y}$, and remembering that $\Sigma$ is Lagrangian, we calculate:
	\begin{align*}
	\ip{\nu}{\mu}_Y|\widetilde{\nu}|_Y|\widetilde{\mu}|_Y &= \ip{\check\mu}{\Voo}_Y - \ip{\Voo}{\hat{X}_I}_Yg^{IJ}_\partial \ip{\hat{X}_J}{\check\mu}_Y\\
	\ip{J\nu}{\mu}_Y|\widetilde{\nu}|_Y|\widetilde{\mu}|_Y &= \ip{\check\mu}{J\Voo}_Y - \ip{\Voo}{\hat{X}_I}_Yg^{IJ}_\partial \ip{J\hat{X}_J}{\check\mu}_Y\ ,\\
	\implies 0&=\ip{\check\mu}{\cos \a \Voo - \sin \a J\Voo - \cos \a \ip{\Voo }{\hat{X}_I}_Yg^{IJ}_\partial\hat{X}_J + \sin \a \ip{\Voo }{\hat{X}_I}_Yg^{IJ}_\partial J\hat{X}_J}_Y.
	\end{align*}
	If we then define $\overline{\mu} = \overline{\mu}(x,u)$, $W=W(x,t,u,\np_\partial u)$ by
	\begin{align*} 
	\overline{\mu} &:=\check{\mu} - \np_{\mu_0}u = \mu_0 - \ip{u}{\up{0}\II\left(\mu_0,\pard{}{x^y}\right)}_0g^{kl}_0 \pard{}{x^l }, \\
	W &:= \cos \a \Voo - \sin \a J\Voo - \cos \a \ip{\Voo }{\hat{X}_I}_Yg^{IJ}_\partial\hat{X}_J + \sin \a \ip{\Voo }{\hat{X}_I}_Yg^{IJ}_\partial J\hat{X}_J.
	\end{align*}
	The above boundary condition may now be written
	\[\ip{\np_{-\mu_0} u}{W}_Y = \ip{\overline{\mu}}{W}_Y=:R(x,t,u,\np_\partial u)\] 
	% We have that the above is equivalent to
	%\[0=\ip{\n_{\mu_0}^\perp u}{\cos \a \Voo - \sin \a J\Voo  - \ip{\Voo}{\cos \a \hat{X}_I}_Yg^{IJ}_\partial \hat{X}_J}_Y+R\]
	Finally, since $\Sigma$ is Lagrangian,
	\[\ip{\Voo}{W}_Y=\ip{\Voo}{\cos \a \Voo - \ip{\Voo}{\cos \a \hat{X}_I}_Yg^{IJ}_\partial \hat{X}_J}_Y = \cos\a \big|\Voo^{T\Sigma\cap N\partial M_0}\big|^2_Y\ ,\]
	where $\Voo^{T\Sigma\cap N\partial M_0}$ is the $g_Y$-orthogonal projection of $\Voo$ into ${T\Sigma\cap N\partial M_0}$. Using the same arguments used to show that $\sigma$ was positive if $c_D$, $c_T$ were small enough then we see that that $|\Voo^{T\Sigma\cap N\partial M_0}|^2$ is a function of $x, t, u, \np u $ which may be assumed to be strictly positive for sufficiently small $c_D$ and $c_T$, and so the obliqueness condition is satisfied.
\end{proof}

A necessary issue to ensure sufficient regularity at time $t=0$ is that compatibility conditions are satisfied. The $0^\text{th}$ compatibility condition is that the initial data satisfies the boundary conditions and are necessary to avoid ``jumping'' at $t=0$. In the case of \eqref{parabolicsystembundle}, if we wish to have a solution which is twice differentiable in space and once differentiable in time (in fact in $\CPN{2+\a}$, see Appendix \ref{Hoeldersection} for a definition) then we require the \emph{first Dirichlet compatibility condition}, namely that at $t=0$
\begin{equation}
0=\ddt{}\left(u - \ip{u}{\Voo}_0\Voo\right) \label{eq-firstdirichlet}
\end{equation}
where $\ddt{u}$ is determined by the first line of \eqref{parabolicsystembundle}. This becomes an algebraic condition on the parabolic system and the initial data.  For a full definition of compatibility conditions, see \cite[pages 319--320]{LSU}. 

Fortunately, the fact that $M_0$ is Lagrangian and satisfies the Neumann and Dirichlet boundary conditions gives us the first Dirichlet compatibility condition for free:

\begin{lemma}[Dirichlet compatibility conditions]\label{Dirichletcompatibility}
	If $M_0$ and $\Sigma_0$ are Lagrangian and satisfy the boundary conditions of \eqref{parabolicsystembundle}, then the first Dirichlet compatibility condition (\ref{eq-firstdirichlet}) is always satisfied.
\end{lemma}
\begin{proof}
	For an arbitrary $p\in\partial M_0$, we need to demonstrate that 
	\[W(p)=\left[\left. \ddt{u}\right|_{t=0}- \ip{\left. \ddt{u}\right|_{t=0}}{\Voo}\Voo\right](p)\]
	is zero. Since $\nu|_{t=0}=DY|_{t=0}(\nu_0)$ (by property \ref{nu0property} in Proposition \ref{diffeoprop}) we have that
	\[DY|_{t=0}(W)=DY|_{t=0}\left(\left. \ddt{u}\right|_{t=0}\right)-\ip{DY|_{t=0}\left(\left. \ddt{u}\right|_{t=0}\right)}{\nu}\nu\]
	By construction of \eqref{parabolicsystembundle} we have that
	\[H|_{t=0} = \left(\left.\pard{Y}{t}\right|_{t=0}+DY|_{t=0}\left(\pard{u}{t}\right)\right)^{NM_0}=\left(\widetilde H|_{t=0} +W+DY|_{t=0}\left(\pard{u}{t}\right)\right)^{NM_0}\]
	for some $W\in T_p\Sigma$. For $e_1, \ldots, e_{n-1}$ an orthonormal basis of $T_p\partial M_0$ we therefore have that
	\[\ip{DY|_{t=0}\left(\pard{u}{t}\right)}{Je_I}=\ip{H|_{t=0}-\widetilde H|_{t=0}}{Je_I}=\ip{\n(\theta-\widetilde \theta)}{e_I}=0\]
	due to the Neumann boundary condition. If $P$ is orthogonal projection on to $\text{span}\{Je_1, \ldots, Je_{n-1}, \nu\}$ then
	\[P(DY|_{t=0}(W))=0\ .\]
	$P$ restricted to $DY(N_pM_0)$ is a linear isomorphism -- otherwise $Y$ cannot be a diffeomorphism as $DY|_{t=0}$ restricted to $T_pM_0$ is the  identity. Therefore, $W=0$.
\end{proof}

In what follows we will require a local version of \eqref{parabolicsystembundle}. Note that in suitable coordinates, the boundary condition splits into $n-1$ Dirichlet conditions and one Neumann condition.

\begin{lemma}[Local coordinates]
	If $u$ is a solution of \eqref{parabolicsystembundle} which is in $\CPNt{2+\a}{c_T}$ (see Appendix \ref{Hoeldersection} for a definition) then at any boundary point $p\in\partial M_0$, there exist local coordinates of $M_0$ on $U\subset \{x\in \bb{R}^n|x^n\geq 0\}$ and a local trivialisation of $NM_0$ such that on $S:=U\cap\{x\in \bb{R}^n|x^n=0\}$, $\Voo = \pard{}{y^n}$ and the above system (\ref{parabolicsystembundle}) may be written as 
	\begin{equation}
	\begin{cases}
	u_t^k - g^{ij}(x,t,u,\np u)\np_{ij}u^k - b(x,t,u,\np u)=0 &\text{on }U\times[0,c_T) \,\, \forall k,\\
	u^I =0, \text{ for }1\leq I \leq n-1 &\text{on } S\times[0,c_T), \\
	\np_nu^ns(x,t,u,\np_\partial u^n)\cos\a - r(x,t,u,\np_n u^I, \np_\partial u^n)=0 &\text{on } S\times[0,c_T),\\
	u^k(\cdot, 0)=0,\text{ for }1\leq k \leq n &
	\end{cases}\label{nearboundary}
	\end{equation}
	where all coefficients are smooth, $g^{ij}$ is positive definite and $s(x,t,u,\np_\partial u^n)>\frac 1 2$ as long as $|u|_{\CPNtn{1}{c_T}}<c_D$ and $\np_n u^I$ indicates dependence on $\np_n u^I$ for all $1\leq I \leq n-1$. 
\end{lemma}
\begin{proof}
	For $q\in \partial M_0$ in a neighbourhood of $p$ take vectors $\widetilde{e_I}$ on $\partial M_0$ so that $\widetilde{e}_1, \ldots, \widetilde{e}_{n-1}, \Voo$ is a $g_0$ orthonormal basis of $N_qM_0$. Clearly we may take local coordinates and a local trivialisation so that on $S$, $\mu_0 = -\pard{}{x^n}$, $\pard{}{y^I}=\widetilde{e}_I$ and $\Voo=\pard{}{y^n}$. The first second and last lines above follow immediately. 
	
	For the Neumann boundary condition we may write $\np_{-\mu_0}u = \np_{n}u^n \nu_0 +\np_{n}u^I \widetilde{e}_I$, and so we see that $s=\ip{\nu_0}{W}_Y$ and $r = -\np_{n}u^I \ip{\widetilde{e}_I}{W}_Y + R$. Finally since $u^I$ is differentiable at the boundary, we have that $\np_Ju^I=0$, and so $s$ and $r$ have the dependences as claimed.
\end{proof}

\subsection{Linearisation}\label{Linearisationsection}
In codimension one, or if equation (\ref{nearboundary}) held on the entirety of $M_0$ (which is equivalent to the normal bundle being trivial) then we would simply be able to apply standard PDE methods, similar to those in \cite[Section 8.3]{Lieberman} to obtain short time existence. However, as we are working with an arbitrary normal bundle and with non-standard boundary conditions, to the best of the authors' knowledge our case is not covered by the literature and so a little more work is required.

We may write
\[\mathcal{P}:\CPN{2+\a} \ra \CPN{\a}\times \CPNb{2+\a}\times C_{\partial, T}^{2+\alpha; \frac{2+\alpha}{2}} ,\]
given by 
\begin{align*}\mathcal{P}u \, &=\, (\mathcal{P}_1 v, \mathcal{P}_2 v,\mathcal{P}_3 v)\\
&= \Big(u_t - \hatG^{ij}(x,t, u,\np u) \np^2_{ij} u - B(x,t, u, \np u)\ ,\\
&\qquad\  u - \ip{\Voo}{u}_Y\Voo\ ,\,\, \ip{\np_{-\mu_0} u }{W(x,t, u, \widetilde \n_I u)} - R(x,t, u, \np_I u )\Big),
\end{align*}
so that a solution of (\ref{parabolicsystembundle}) is given by 
\begin{equation}
\label{QLsysteminitialdata}
\begin{cases}
\mathcal{P}u=0,\\
u(\cdot, 0)=0.
\end{cases}
\end{equation}
We write the Fr\'echet derivative of $P$ at a general $u\in \CPN{2+\a}$ by
\begin{align*}
\mathcal{L}_u&:\CPN{2+\a} \ra \CPN{\a}\times \CPNb{2+\a}\times C_{\partial, T}^{2+\alpha; \frac{2+\alpha}{2}},\\
\mathcal{L}_uv&=(\mathcal{L}^1_u v, \mathcal{D}_uv,\mathcal{N}_uv)
=\Big(v_t - \up{u}a^{ij} \np_{ij} v - \up{u}b(\np v) - \up{u}c(v)\ ,\,\, v - \ip{\Voo}{v}_0\Voo\ ,\,\,\up{u}\beta(\np v )+\up{u}e(v)\Big),
\end{align*}
where, as usual $\mathcal{L}_uv =\frac{d}{dt}\big|_{t=0}\mathcal{P}(u+tv)$ (so in particular, $\mathcal{D}_u = \mathcal{P}_2$). Explicitly, writing $G = G(x,t, z^l, p^i_k)$ and so on as usual (where $p_i^k$ corresponds to $\np_i u^k$), then in coordinates as in \eqref{nearboundary} we have locally 
\[\mathcal{P}_3 u= \np_n u^n s(x,t,u,\np_\partial u^n)\cos \a-r(x,t,u, \n u^I, \n_\partial u^n)\ ,\]
and so 
\begin{equation*}
\begin{matrix}
\up{u}a^{ij} = \hatG^{ij},\qquad\qquad &\up{u}b^{ik}_l=\pard{\hatG^{ab}}{p_i^l}\np^2_{ab} u^k +\pard{B^k}{p_i^l},\\
\up{u}c^k_l=\pard{\hatG^{ab}}{z^l}\np^2_{ab} u^k +\pard{B^k}{z^l},& \up{u}\beta_n^n=s(x,t,u,\np_I u)\cos\a,\\
\up{u}\beta_n^I = \np_n u^n\pard{s}{p_I^n}\cos \a-\pard{r}{p_I^n}, &\up{u}\beta_J^i = -\pard{r}{p_i^J},\\
\up{u}e_l=\np_n u^n\pard{s}{z^l}\cos\a -\pard{r}{z^l},& %W= w(x,u, D_I u).
\end{matrix}
\end{equation*}
where all coefficients of $\mathcal{P}$ are evaluated at $(x,t, u,\np u)$ and we may write $\up{u}\beta(\np u) =\up{u}\beta^n_n\np_nu^n+\up{u}\beta_n^I\np_Iu^n+\up{u}\beta_J^i\np_iu^J$ and so on.
%We do not write them all here, but note that 
%\[G^{ij} = \hatG^{ij}(x,t, u,\np u), \text{ and } W= w(x,u, D_I u)\]
%and so 
In particular, we note that if $|u|, |\np u |<c_D$ and $t<c_T$, the linearisation is uniformly parabolic and oblique. 

We define $|\mathcal{P}|_{C^{k,\a}}$ to be the $C^{k,\a}$-norm on $\mathcal{P}$ where $\mathcal{P}$ is considered as a map acting on $(x,t,u, \np u, \np^2u)$.

\subsection{Newton Iteration and Compatibility Conditions}\label{NewtonSection}

As in \cite{Panagiotis}, \cite{Weidemaier}, we prove short time existence for \eqref{QLsysteminitialdata} by application of the contraction mapping theorem to a mapping determined by the Newton method on Banach spaces. Specifically, we will consider a mapping $S$ which takes suitable functions $u$ to the solution $v$ of
\begin{equation}\mathcal{L}_u(v) = \mathcal{L}_u(u) - \mathcal{P}(u)\ ,\label{Newton}
\end{equation}
where $\mathcal{L}_u$ is the Fr\'echet derivative of $\mathcal{P}$, as above. Clearly, at a fixed point of $S$ then $Su=v=u$ and we have a solution of $\mathcal{P} u=0$. 

We now define the domain of $S$. For $\tau<C_T$, let $\vn$ be a solution of the equation 
\begin{equation}
\begin{cases}
\mathcal{L}^1_0\vn = B(x,t,0,0)& \text{ on } M_0\times[0,\tau)\\
\mathcal{D}_0\vn =0& \text{ on } \partial M_0\times[0,\tau)\\
\mathcal{N}_0\vn = R(x,t,0,0)&\text{ on } \partial M_0\times[0,\tau)\\
\vn(\cdot, 0)=0;&\label{vneq}
\end{cases}
\end{equation}
note this is the linearisation of \eqref{parabolicsystembundle} at $u \equiv 0$. By Proposition \ref{linearsystemexists} in Appendix \ref{app-para}, a solution $\vn\in\CPN{2+\a}$ of \eqref{vneq} always exists (if the $0^{\text{th}}$ and up to the $1^\text{st}$ compatibility conditions are satisfied for $\mathcal{N}$ and $\mathcal{D}$ respectively). We fix $\a\in(0,1)$ and for any $\delta< \frac 1 2 c_D$ and $\tau<c_T$ we define 
\[\Aaa = \left\{v\in \CPNt{2+\a}{\tau}\Big| |v|_\CPNtn{1}\tau\leq\delta, |v-\vn|_{\CPNtn{2+\a}\tau}\leq\Theta, v(\cdot,0)=0 \right\}\ ,\]
which is complete as a subset of $\CPNt{2+\a}{\tau}$. We will show that that, given $\delta$ and $\Theta$, there exists a $\tau$ such that $S$ maps $S:\Aaa\ra \Aaa$, and furthermore $S$ is a contraction mapping.

We rewrite the linear parabolic system given by (\ref{Newton}) as
\begin{equation}
\begin{cases}
\mathcal{L}^1_u v = \mathcal{L}^1_u u - \mathcal{P}^1 u =:\psi_u&\text{ on }M_0\times[0,\tau)\\
v-\ip{v}{\Voo}_0\Voo=0 &\text{ on } \partial M_0\times[0,\tau)\\
\mathcal{N}_u v = \mathcal{N}_u u -\mathcal{P}^3(u)=:\Upsilon_u&\text{ on } \partial M_0\times[0,\tau)\\
v(\cdot, 0)=0. &
\end{cases}\label{NewtonRewrite}
\end{equation} 

Clearly \eqref{NewtonRewrite} satisfies compatibility conditions on Dirichlet condition (the second line above) if $\mathcal{P}$ does for $\mathcal{P}_2$.

\subsection{Proof of Contraction}\label{Contractionsection}

The purpose of this section is to prove the following Proposition. 
\begin{proposition}\label{Scontractionmapping}
	Suppose that (\ref{QLsysteminitialdata}) satisfies compatibility conditions up to the first order on $\mathcal{P}_2$ and to the $0^\text{th}$ order on $\mathcal{P}^3$. Then there exists a $\tau=\tau(|\vn|_{2+\a}, \a,\delta, \Theta, |\mathcal{P}|_{C^{k,\a}})$ such that the mapping $S$ defined above maps $S:\Aaa\ra\Aaa$ is a contraction mapping.
\end{proposition}
\begin{proof}
	We need to show that firstly that $S$ maps into $\Aaa$ and secondly that $S$ is a contraction.
	
	Let $v$ be a solution of \eqref{NewtonRewrite}, and observe that $v-\vn$ satisfies
	\begin{equation*}
	\begin{cases}
	\mathcal{L}^1_u (v-\vn) = \psi_u +\mathcal{L}_0^1 \tilde v - \mathcal{L}_u^1 \tilde v-B(x,t,0,0)\\
	\mathcal{D}_u (v-\vn) =0\\
	\mathcal{N}_u (v-\vn) = \Upsilon_u + \mathcal{N}_0 \vn - \mathcal{N}_u \vn - R(x,t,0,0)\\
	(v-\vn)(\cdot,0) = 0,
	\end{cases}
	\end{equation*} 
	where all coefficients of $\mathcal{L}^1_u$ are in $C^{\a;\frac \a 2}$, the coefficients of $\mathcal{N}_u$ are in $C^{1,\a}$ and $\mathcal{D}_u$ is smooth. By applying Schauder estimates (Proposition \ref{linearsystemholder} in Appendix \ref{app-para}) we see that
	\begin{align*} 
	|v-\vn|_\CPNtn{2+\a}\tau&\leq C \left(|\Upsilon_u+\mathcal{N}_0\vn -\mathcal{N}_u\vn - r(x,t,0,0)|_\CPNbtn{1+\a}\tau\right.\\
	&\qquad\qquad\qquad\left.+|\psi_u+\mathcal{L}^1_0\vn -\mathcal{L}^1_u\vn - b(x,t,0,0)|_\CPNtn{\alpha}\tau\right)
	\end{align*}
	where $C=C(|\mathcal{P}|_{C^{2,\a}}, \Theta)$ is uniformly bounded. In Lemma \ref{mappinglemma} below we see that the bracket on the right hand side may be made arbitrarily small by restricting to a sufficiently small time interval, and so by making $\tau$ sufficiently small we may ensure that $|v-\vn|_\CPNtn{2+\a}\tau<\Theta$.
	
	Crude estimates imply that 
	\[|v-\vn|_\CPNtn{0}\tau \leq \Theta\tau\ ,\]
	and so by interpolation we have that
	\[|v-\vn|_\CPNtn{1}\tau \leq C(\alpha)|v-\vn|_\CPNtn{0}\tau^\frac{1+\a}{2+\a}|v-\vn|_\CPNtn{2+\a}\tau^\frac{1}{2+\a}\leq C(\a,\Theta)\tau^\frac{1+\a}{2+\a} \ .\]
	A similar interpolation implies we may restrict $\tau$ so that $|\vn|_\CPNtn{1}\tau<\frac \delta 2$, and so by making $\tau$ sufficiently small,
	\[|v|_\CPNtn{1}\tau\leq |v-\vn|_\CPNtn{1}\tau+|\vn|_\CPNtn{1}\tau<\delta\]
	and so $v\in \Aaa$.
	
	Proving that $S$ is a contraction follows an identical argument. If $Su_1=v_1$ and $Su_2=v_2$, then $v_1-v_2=Su_1-Su_2$ satisfies a linear parabolic equation
	\begin{equation*}
	\begin{cases}
	\mathcal{L}^1_{u_1} (v_1-v_2) = \psi_{u_1} - \psi_{u_2} +(\mathcal{L}^1_{u_2}-\mathcal{L}^1_{u_1})v_2\\
	\mathcal{D}_{u_1} (v_1-v_2) =0\\
	\mathcal{N}_{u_1} (v_1-v_2) = \beta_{u_1} - \beta_{u_2} + (\mathcal{N}_{u_2} - \mathcal{N}_{u_1}) v_2\\
	(v_1-v_2)(\cdot,0) = 0.
	\end{cases}
	\end{equation*}
	Again, the coefficients of this equation are suitably regular and so applying Proposition \ref{linearsystemholder} we see that
	\begin{align*} 
	|Su_1-Su_2|_\CPNtn{2+\a}\tau&=|v_1-v_2|_\CPNtn{2+\a}\tau\\
	&\leq C \left(|\beta_{u_1} - \beta_{u_2} + (\mathcal{N}_{u_2} - \mathcal{N}_{u_1}) v_2|_\CPNbtn{1+\a}\tau \,+\, |\psi_{u_1} - \psi_{u_2} +(\mathcal{L}^1_{u_2}-\mathcal{L}^1_{u_1})v_2|_\CPNtn{\alpha}\tau\right)\ ,
	\end{align*}
	for some $C=C(|\mathcal{P}|_{C^{2,\a}}, \Theta)$. In Lemma \ref{contractionlemma} we see that by making $\tau$ sufficiently small the bracket may be estimated by an arbitrarily small multiple of $|u_1-u_2|_\CPNtn{2+\a}\tau$, and so the contraction property is proven.
\end{proof}

\begin{lemma}[Mapping Lemma]\label{mappinglemma}
	For $u\in \Aaa$, there exists constants $C=C(\a,|\vn|_{2+\a},\Theta, |\mathcal{P}|_{C^{2,\a}})$ and $0<p=p(\a)$ such that
	\[|\psi_u+\mathcal{L}^1_0\vn -\mathcal{L}^1_u\vn +B(x,t,0,0)|_\CPNtn{\alpha}\tau+|\Upsilon_u+\mathcal{N}^3_0\vn -\mathcal{N}^3_u\vn +r(x,t,0,0)|_\CPNtn{1+\a}\tau\leq C\tau^p\ .\]
\end{lemma}
\begin{proof}
	We calculate that
	\begin{flalign}
	\label{firstterm}\psi^k_u&+(\mathcal L^1_0\vn)^k -(\mathcal L^1_u\vn)^k -B^k(x,t,0,0)\\
	&=\left[B^k(x,t,u,\np u)-B^k(x,t,0,0)\right] -\up{u}b^{ik}_l\np_iu^l-\up{u}c_l^k u^l\nonumber\\
	&\qquad +\left[\up{u}a^{ij}-\up{0}a^{ij}\right]\np^2_{ij}\vn^k+\left[\up{u}b^{ik}_l-\up{0}b^{ik}_l\right]\np_i\vn^l +\left[\up{u}c_l^k-\up{0}c_l^k\right]\vn^l\nonumber
	\end{flalign}
	%Each term in this expression in the above expression is bounded in $\CPNt{\a}\tau$. 
	We prove the claim by showing each of the square brackets may be made sufficiently small and making liberal use of the estimate
	\begin{equation}|f g |_\CPNtn{\a}{\tau} \leq C(\a)(|f|_\CPNtn{0}\tau|g|_\CPNtn{\a}{\tau} +|g|_\CPNtn{0}\tau|f|_\CPNtn{\a}{\tau}) \ .\label{thatstheone}
	\end{equation}
	For example, we calculate that 
	\[\up{u}b^{ki}_l-\up{0}b^{ki}_l = \int_0^1 \frac{d}{d\tau} b^{ki}_l(x,t,\tau u,\tau\np u, \tau\np^2 u)d\tau= ({b_1})^{ki}_l u^k+(b_2)^{kim}_{lr}\np_mu^r+(b_3)^{kims}_{lr}\np^2_{ms}u^r\]
	where $ b_1,b_2,b_3$ are smooth functions. This yields 
	\[|\up{u}b^{ki}_l-\up{0}b^{ki}_l|_\CPNtn{\alpha}{\tau}<C|u|_\CPNtn{2+\a}\tau=C(\Theta, |\vn|_{2+\a})\ .\]
	In fact each of the above terms is a multiple of two terms that are bounded in $\CPN{\a}$ and are zero at $t=0$. This imples, for example, that $|\up{u}b^{ki}_l-\up{0}b^{ki}_l|_\CPNtn{0}\tau<C\tau^\a$, and so \eqref{thatstheone} implies the required bounds.

	We also have that
	\begin{flalign*}
	\Upsilon_u&-r(x,t,0,0)+\mathcal{N}_0\vn -\mathcal{N}_u\vn \\
	&=\up{u}\beta_I^n\n_nu^I +\up{u}\beta_i^J\n_J u^i+\up{u}e_lu^l+\left[r(x,t,u,\np u)-r(x,t,0,0)\right]\\
	&\quad+\left[\up{0}\beta_n^n-\up{u}\beta_n^n\right]\np_n\vn^n+\left[\up{0}\beta^I_n-\up{u}\beta^I_n\right]\np_n\vn^I+\left[\up{0}\beta^I_J-\up{u}\beta^I_J\right]\np_I\vn^J+\left[\up{0}e_k-\up{u}e_k\right]\vn^k\end{flalign*}
	which, using methods as above is clearly bounded in $\CPN{\a}$. When taking a derivative, some cancellation occurs (due to the form of the linearisation), and we have that
	\begin{flalign*}
	\np_K(\Upsilon_u&-r(x,t,0,0))\\
	&=\np_nu^n\cos(\alpha)\left(\pard{s}{p^n_I}\np^2_{KI}u^n+\pard{s}{z^k}\n_Ku^k\right)+\left[\pard{r}{x^K}(x,t,u,\np u)-\pard{r}{x^K}(x,t,0,0)\right]\\
	&\qquad+\frac{d\up{u}\beta_I^n}{dx^K}\n_Iu^n +\frac{d\up{u}\beta_I^J}{dx^K}\n_I^Ju+\frac{d\up{u}e_l}{dx^K}u^l\ .
	\end{flalign*}
	The first term is as above, while the remaining four each contain a factor which (by interpolation) is small in $\CPN{\a}$, for example
	\[|\np u|_\CPNtn{\a}\tau \leq |u|_\CPNtn{0}\tau^\frac{1}{1+\a}|u|^\frac{1+\a}{2+\a}_{\CPNtn{2+\a}\tau}\leq C\tau^\frac 1 {1+\a} .\]

	Finally, we see that
	\begin{flalign*}
	\np_K(\mathcal{N}_0\vn -&\mathcal{N}_u\vn)
	=\left[\up{0}\beta_n^n-\up{u}\beta_n^n\right]\np_{Kn}^2\vn^n\\
	&\qquad+ \left(\left[\pard{\up{0}\beta_n^n}{x^K}-\pard{\up{u}\beta_n^n}{x^K}\right]+\left[\pard{\up{0}\beta_n^n}{z^k}-\pard{\up{u}\beta_n^n}{z^k}\right]\n_Ku^k+ \left[\pard{\up{0}\beta_n^n}{p_i^k}-\pard{\up{u}\beta_n^n}{p_i^k}\right]\np_{Ki}u^k\right)\np_n\vn^n\\
	&\qquad+\text{ similar terms.}
	\end{flalign*}
	and again each of these terms may be dealt with similarly to earlier cases.
\end{proof}

\begin{lemma}[Contraction Lemma]\label{contractionlemma}
	For $u,w, v_2\in\Aaa$, there exists constants $p=p(\a)>0$ and $C=C(\a, |\vn|_{2+\a}, \Theta, |\mathcal{P}|_{C^{2,\a}})$ such that
	\begin{align*}|\Upsilon_{u} - \Upsilon_{w}|_\CPNbtn{1+\a}\tau& + |(\mathcal{N}_{u} - \mathcal{N}_{w}) v_2|_\CPNbtn{1+\a}\tau\\
	&+|\psi_{u} - \psi_{w}|_\CPNtn{\alpha}\tau +|(\mathcal{L}^1_{u}-\mathcal{L}^1_{w})v_2|_\CPNtn{\alpha}\tau<C\tau^p|u-w|_\CPNtn{2+\a}\tau
	\end{align*}
\end{lemma}
\begin{proof}
	This is almost exactly as in the previous proof. We have that
	\[\psi_u^k = B^k(x,t,u, \n u ) - \up{u}b_l^{ki}\np_iu^l - \up{u}c^k_lu^l\]
	so
	\begin{align*}\psi_u^k - \psi_w^k& = \left[B^k(x,t,u, \n u )-B^k(x,t,w, \n w )\right] \\
	&\qquad+\up{w}b_l^{ki}\left[\np_iw^l-\np_iu^l\right]+\left[\up{w}b_l^{ki}- \up{u}b_l^{ki}\right]\np_iu^l \\
	&\qquad+\up{w}c^k_l\left[w^l-u^l\right]+\left[\up{w}c^k_l- \up{u}c^k_l\right]u^l
	\end{align*}
	All terms may be estimated using similar interpolation methods to in the  previous lemma. For example
	\[ \left|\left[\up{w}b_l^{ki}- \up{u}b_l^{ki}\right]\np_iu^l \right|_\CPNtn{\a}\tau \leq C|w-u|_\CPNtn{2}\tau|u|_\CPNtn{1+\a}\tau+C|w-u|_\CPNtn{2+\a}\tau|u|_\CPNtn{1}\tau\leq C \tau^\frac{1}{2+\a}|w-u|_\CPNtn{2+\a}\tau\ .\]
	We clearly have that 
	\[|\psi_u - \psi_w|_\CPNtn{\a}\tau<C|w-u|_\CPNtn{1+\a}\tau<C\tau^\frac{1}{2+\a}|u-w|_\CPNtn{2+\a}\tau\]
	where we used that $|u(x,t)-w(x,t)|<2t$. Identical methods may be applied to $|(\mathcal{L}^1_{u}-\mathcal{L}^1_{w})v_2|_\CPNtn{\alpha}\tau$. Also,
	\[\Upsilon_u = \up{u}\beta^n_I\n_nu^I + \up{u}\beta^J_i\n_J u^i +\up{u}e_lu^l+r(x,t,u,\n u).\]
	Here we note that, as by assumption the initial data satisfies the Neumann boundary conditions $0=\mathcal{P}_3|_{t=0}=r(x,0,0,0)$, by looking at the equations for $\beta_i^j$ it follows that $\up{0}\beta^i_j=0$ unless $i=j=n$. As a result, for such $i, j$ we have that $|\up{u}\beta_i^j|_0\leq |u|_1\leq \Theta\tau$ and in particular, as $| \beta_i^j|_{1+\a}<C$, $|\beta|_1\leq C\tau^\frac \a{1+\a}$. 
	As a result of this observation we may obtain the relevant $\CPN{1+\a}$ bound for
	\begin{align*}
	\Upsilon_u - \Upsilon_w &= \up{u}\beta^I_n\left[\np_Iu^n-\np_Iw^n\right] +\up{u}\beta^n_J\left[\np_nu^J-\np_nw^J\right]+\left[\up{u}\beta^I_n-\up{w}\beta^I_n\right]\np_Iw^n\\
	&\qquad+\left[\up{u}\beta^n_J-\up{w}\beta^n_J\right]\np_nw^J + \up{u}e_l\left[u^l-w^l\right] +\left[\up{u}e_l-\up{w}e_l\right]w^l\\
	&\qquad+r(x,t,u,\n u)-r(x,t,w,\n w)\ ,
	\end{align*}
	where the final term follows from writing
	\begin{flalign*}r(x,t,&u,\n u)-r(x,t,w,\n w)\\
	&=\int_0^t\int_0^1\pard{}{\beta}\pard{}{\tau} r(x,\beta,\tau u+(1-\tau )w,\tau \np u+(1-\tau )\np w)d\tau d\beta\\
	& = t\left[(r_1)_j^i\np_i(w-u)^j+(r_2)_j(w-u)^j\right]\ .\end{flalign*}
	
	The term
	\[(\mathcal{N}_{u} - \mathcal{N}_{w}) v_2\]
	may be estimated similarly, but this is easier due to estimates that we already have on $v_2$, and so this is left as an exercise to the reader.
\end{proof}

\subsection{Proof of Theorem \ref{STE}}
Before proving Theorem \ref{STE} we collect the conclusions of the previous sections:

\begin{proposition}\label{ExistsOne} Suppose that \eqref{QLsysteminitialdata} satisfies up to $1^\text{st}$ order compatibility conditions on $\mathcal{P}_2$ and $0^\text{th}$ order compatibility conditions on $\mathcal{P}_3$. Then there exists a maximal time $0<T=T(|\mathcal{P}|_{C^{2,\a}}, c_D,c_T)\leq c_T$ such that there exists a unique solution $u\in\CPNt{2+\a}{T}$ of \eqref{QLsysteminitialdata}. This solution is smooth for $t>0$ and if $T<c_T$ then either $|u|_\CPNtn{1}{T}=c_D$ or $|u|_\CPNtn{1+\alpha}{\tau}\ra \infty$ as $\tau \ra T$ for all $\a\in(0,1)$.
\end{proposition}
\begin{proof}
	Short time existence follows from Proposition \ref{Scontractionmapping}, and uniqueness also follows from application of the contraction mapping theorem. Standard Schauder estimates now imply that the solution is smooth for $t>0$. Suppose a solution $u$ exists until time $\tau<c_T$ and there exists an $\a\in(0,1)$, $C<\infty$ so that $|u|_\CPNtn{1+\alpha}{\tau}<C$ and $|u|_\CPNtn{1}{\tau}<c_D$. Schauder estimates imply that the solution is smooth up to time $\tau$, and writing $\varphi(\cdot)=u(\cdot, \tau)$, we see that $\tilde{u}(x,t)=u(x,t-\tau)-\varphi(x)$ satisfies an equation of the same form as (\ref{QLsysteminitialdata}) with compatibility conditions to all orders. Therefore Proposition \ref{Scontractionmapping} implies that the solution may be extended (smoothly) to a later time, implying $\tau$ was not maximal. 
\end{proof}
\begin{proof}[Proof of Theorem \ref{STE}]
	
		Propositions \ref{QLsystem}, Lemma \ref{Dirichletcompatibility} and Proposition \ref{ExistsOne} imply that a solution exists for some positive maximal time $T>0$. 
	
	Suppose that for all the time $t<T<\infty$ that the solution exists there are constants $C_{\II}$, $C_{\II^\Sigma}$ such that
	\begin{enumerate}[label=\alph*)]
		\item  $|\II(x,t)|<C_{\II}$,\label{curvassump}
		\item $|\II^\Sigma(x,t)|<C_{\II^\Sigma}$,\label{bdrycurvassump}
		\item the boundary injectivity radius is uniformly bounded from below, \label{lastone}
	\end{enumerate}
	We see that due to the above assumptions there exists a bounded, compact set $M_T$ such that $M_t$ converges to $M_T$ uniformly as $t\ra T$. Since we have uniform curvature bounds, on the interior of $M_t$ we have standard local curvature estimates via standard methods such as the proof of \cite[Theorem 3.4]{EckerHuiskenInteriorEstimates}, and so we can guarantee that away from $\partial M_T:=\lim_{t\ra T}M_t$, $M_T$ is smooth. Similarly we have that for all $0<\frac{T}{2}<t<T$, $\Sigma_t$ is uniformly smooth.
	
	We must demonstrate the same at the boundary where no suitable local estimates are currently known to the authors. Our concern is that a region of the boundary somehow conspires to have exploding derivatives of curvature as $t\ra T$, which in turn implies that \eqref{parabolicsystembundle} has arbitrarily large coefficients in $C^{2,\a}$ and/or arbitrarily small $c_T, c_D$. To get around this problem, we locally rewrite \eqref{parabolicsystembundle} over a ``neutral'' manifold so that the corresponding system has uniformly bounded coefficients and we may apply local Schauder estimates up to the boundary.
	
	For some $\e$ to be determined (depending only on $C_{\II}$, $C_{\widetilde\II}$), we pick a point $p\in \partial M_{T-\e}$. We now define a small portion of a submanifold $Q$, which is constructed by first choosing $\partial Q\subset\Sigma_{T-\e}$ to be the image of the exponential map of $\Sigma_{T-\e}$ at $p$  applied to $T_p\partial M_{T-\e}$. For every $q\in \partial Q$, we pick a vector field $\mu_Q(q)$ so that $\text{span}\{T_q\partial Q, \mu_Q(q)\}$ is Lagrangian and with Lagrangian angle determined by the boundary condition, and $\mu(p)$ points into $M_{T-\e}$. Finally we define $Q$ by extending $\mu_Q(q)$ by geodesics. Clearly there exists a $\delta=\d(C_{\II}, C_{\widetilde\II})>0$ such that $Q\cap B_\delta(p)$ is uniformly $C^k$ depending only on our uniform bounds on $\Sigma$.
	
	As in Proposition \ref{diffeoprop} we may construct time dependent local diffeomorphisms $Y$ from $NQ\times[0,T_Y)$ to $\CY$ a so that, by again reducing $\delta$ we may locally write \eqref{MCFBC} as in Proposition \ref{QLsystem}, except that now $u$ a time-dependent section of $N(Q\cap B_\delta(p))$ for $t\in[T-\e, T-\e+c_T)$ and in place of the final line of \eqref{parabolicsystembundle} we need to specify initial data $u(\cdot, T-\e)=\varphi(\cdot)$. We note that $c_T$, $c_D$ and the coefficients of the system depend only on $\Sigma$. We of course choose initial data $\varphi$ to parametrise $M_{T-\e}$, where we note that $\varphi(p)=0$ and $\widetilde{\n} \varphi(p)=0$ and so by choosing $\delta$ sufficiently small (depending on $C_{\II}$) we may assume $|\varphi|_1<\frac{1}{2}c_D$. Futhermore using \ref{curvassump} and \eqref{Curvinu}, we have that there exists a constant $C$ depending only on $C_{\II^\Sigma}$ and $C_{\II}$ such that while $|u|_{\Gamma^{1;\frac{1}{2}}_{[T-\e,t]}(Q\cap B_\delta(p))}<c_D$,
	\[|\np^2 u|\leq C\ .\]
	Using \eqref{parabolicsystembundle} we therefore see that while $|u|_{\Gamma^{1;\frac{1}{2}}_{[T-\e,t]}(Q\cap B_\delta(p))}<c_D$ there exists $C_1=C_1(C)$ such that for $t>T-\e$,
	\[|u(\cdot,t)-\varphi(\cdot)|<C_1(t-(T-\e))\ .\]
	and so 
	\[|u-\varphi|_{\Gamma^{1;\frac{1}{2}}_{[T-\e,T-\e+\tau]}(Q\cap B_\delta(p))}<C_2\sqrt{\tau}\ .\]
	where $C_2=C_2(C, C_\varphi)$. We therefore see that there exists a uniform time $\tau=\tau(C_{\II}, C_{\overline\II})$ such that the localised version of \eqref{parabolicsystembundle} on $Q$ is parabolic. Schauder estimates imply that we have uniform estimates on $Q\cap B_{\frac{\delta}{2}}(q)$ to all orders on the solution $u$. As $q$ was arbitrary, we may take $\e = \frac \tau 2$ to obtain smooth estimates on a neighbourhood of $\partial M_T$.
	
	As a result, $M_T$ is a smooth manifold and Proposition \ref{ExistsOne} may now be applied to see that $T$ was not the final time.
	\end{proof}

\begin{appendix}
	\section{{H}\"older Spaces}\label{Hoeldersection}
	Before dealing with the above PDE, we define the function spaces in which we will work. Let $\Omega\subset\bb{R}^n$ be a domain, and define the parabolic domain $\Omega_T = \Omega\times[0,T)$ for some $T>0$. For a chosen $\rho_0>0$, we define H\"older norm for functions on $\Omega$ to be 
	\[|u|_\CE k \a \Omega = \sum_{|\b|\leq k} |D^x_\beta u|_0 +\sum_{|\b|= k} [D^x_\b u]_\CE 0 \a \Omega\ ,\]
	where the sum is over all multi-indices $\beta$, $| - |_0$ is the standard $C^0$ norm, and $[-]_{C^{0,\alpha}}$ is the H\"older seminorm defined by
	\[[u]_\CE 0 \a \Omega :=\underset{\substack{ {\scriptstyle x,y\in\Omega}\\{\scriptstyle |x-y|<\rho_0}}}\sup\frac{|u(x)-u(y)|}{|x-y|^\a}.\]
	Similarly we define parabolic H\"older norm by
	\[|u|_\CP {k+\a} {\Omega_T} = \sum_{2r+|\beta|\leq k} |D_t^rD_\beta^x u|_0 +[u]_\CP{k+\a}{\Omega_T} \ , \]
	where
	\begin{align*}[u]_\CP{k+\a}{\Omega_T} &= \sum_{2r+|\b|=k} \sup_{t\in[0,T]} [D_t^rD_\beta u(\cdot, t)]_\CE 0 \a \Omega %\\&\qquad\qquad\qquad
	+\sum_{0<k+\a-r-|\beta|<2} \sup_{x\in\Omega} [D_t^r D_\beta u(x,\cdot)]_\CE{0}{\frac{\a}{2}}{(0,T)}.
	\end{align*}
	We write $\CP{k+\a}{\ov{\Omega}_T}$ for the space of all functions on $\ov{\Omega}_T$ such that $D^r_tD_\beta u$ is continuously defined on $\ov{\Omega}_T$ for all $2r+|\beta|\leq k$ and $|u|_\CP{k+\a}{\ov{\Omega}_T}$ is bounded. 
	
	For compact $M_0$ with boundary $\partial M_0$, we define $M_{0,T}=M_0\times[0,T)$. Considering a finite cover of coordinate patches $U_i$ with cutoff functions $\chi^i$ on $M_0$, we define the H\"older norm of a function $f$ as the maximum of the H\"older norms of $f\chi^i$ on the coordinate patches. In this way we may define H\"older spaces $\CE{k}{\a}{M_0}$ on $M_0$ and $\CP{k+\a}{\ov{M}_{0,T}}$. 
	
	Let $\Gamma$ be the space of continuous sections of the normal bundle, and define $\Gamma_T$ to be time dependent continuous sections for $t\in [0,T)$. Identically to above, using a covering of $M_0$ by a finite number of simply connected coordinate patches and trivialisations of the normal bundle of $NM_0$ we may define H\"older norms on sections of the normal bundle to be the sum over the norms over the trivialisations (see for \cite[Section 2.2]{Pulemotov} for similar constructions). 
	
	In this way we define the (elliptic) H\"older space of $k+\a$ differentiable sections of the normal bundle of $M_0$, denoted $\CEN{k}{\a}$ with H\"older norm $|-|_\CENn{k}{\a}$. Similarly we define the parabolic H\"older space of time dependent sections which are $k+\a$ times differentiable in space and $\frac{k+\a}{2}$ differentiable in time, which we denote $\CPN{k+\a}$ with norms $|-|_\CPNn{k+\a}$. We will denote by $N\partial M_0$ the pullback bundle of $NM_0$ to $\partial M_0$ by the inclusion mapping. Using the same idea, we denote time dependent sections of $N\partial M_0$ which are $k+\a$ times differentiable in space and $\frac{k+\a}{2}$ differentiable in time by $\CPNb{k+\a}$ with norm $|-|_{\CPNbn{k+\a}}$.
	
	\section{Estimates for Linear Parabolic Systems with a Mixed Boundary Condition}\label{app-para}
	For $u\in\CPN{2+\a}$, we now study the linear parabolic system which we write in coordinates as
	\[(\mathcal{L}u)^k = u_t^k - a^{ij}(x,t)\np_{ij} u^k -b^{ki}_l(x,t) \np_iu^l - c^k_l(x,t)u^l\]
	or as linear mappings as
	\[\mathcal{L}u = u_t - a^{ij}(x,t)\np_{ij} u -b(\np u) - c(u)\]
	with boundary operators 
	\[\mathcal{D}u = u-\ip{\Voo(x,t)}{u}_0\Voo(x,t)\]
	and
	\[\mathcal{N}u =\beta(\np u)+e(u)\]
	where $\beta$ and $e$ are linear mappings, so in coordinates $\beta(\np u)=\beta_k^i(x,t)\np_i u^k$ and $e(u)=e_l(x,t)u^l$. On the above we will assume that $|\Voo|_0=1$, $\partial M_0$ and $\Voo$ are smooth and $a, b, c, \beta, e \in \CP \a {\ov M_t}$ in the sense that in the system of localisations as determined in Appendix \ref{Hoeldersection}, they are bounded in $C^{\a;\frac{\a}{2}}$. We also require that this system is \emph{uniformly parabolic}, that is, for all $\xi\in T_pM_0$, 
	\begin{equation}\l|\xi|^2\leq a^{ij}\xi_i\xi_j\leq \Lambda|\xi|^2\ .\label{ssystemparabolicity}
	\end{equation}
	and that $\mathcal{N}$ satisfies a \emph{uniform obliqueness condition in direction $\Voo$}, that is, there exists a uniform constant $\chi>0$ such that, if $\mu_0$ is the outward unit vector to $M_0$ then 
	\begin{equation}
	\beta(\mu_0\otimes \Voo)=\beta^k_i\mu^i_0{\Voo}_k \geq \chi>0\ .
	\label{ssystemobliqueness}
	\end{equation}
	
	Specifically we will consider the system
	\begin{equation}
	\begin{cases}
	\mathcal{L}u=f&\text{ on } M_{0}\times[0,T)\\
	\mathcal{D}u=0&\text{ on } \partial M_{0}\times[0,T)\\
	\mathcal{N}u=\Phi&\text{ on } \partial M_0\times[0,T)\\
	u(\cdot, 0)=\varphi(\cdot)&
	\end{cases}\label{linearsystem}
	\end{equation}
	We will also assume that the data for (\ref{linearsystem}) satisfies compatibility conditions to various orders, which are determined iteratively, as on \cite[pages 319--320]{LSU}. 
	
	We note that we may choose a finite number of local trivialisations covering $M_0$ such that the base of each trivialisation is an open simply connected coordinate patch $U$ with $\ov{U}\cap\partial M_0=S$ such that either $S=\emptyset$ or $S$ is a simply connected portion of $\partial M_0$. Furthermore we may choose coordinates on these patches so that $S$ is given by $x^n=0$ and $\mu_0 = -\pard{}{x^n}$ and so that over $S$, $\Voo=\frac{\pard{}{y^n}}{|\pard{}{y^n}|_0}$ near the boundary. In these coordinates (\ref{linearsystem}) may be written
	\begin{equation}
	\begin{cases}
	u_t^k = a^{ij}(x,t)D^2_{ij} u^k +b^{ki}_l(x,t) D_iu^l + c^k_l(x,t)u^l+f^k&\text{ on } U_T\\
	u^I=0&\text{ on } NS_T\\
	D_nu^n\beta^n_n(x,t)+\beta^I_n(x,t)D_Iu^n+\beta^i_J(x,t)D_iu^J+e_k(x,t)u^k=\Phi&\text{ on } NS_T\\
	u^k(\cdot, 0)=\varphi^k(\cdot)&
	\end{cases}\label{linearsystemboundary}
	\end{equation}
	where now $\beta_n^n>\chi$. 
		\begin{proposition}\label{linearsystemholder}
		Suppose that the coefficients of $\mathcal{L}$ are in $\CPN{\a}$, the coefficients of $\mathcal{N}$ in $\CPNb{1+\a}$ and the coefficients of $\mathcal{D}$ are in $\CPNb{\a}$. Suppose $\mathcal{L}$ satisfies (\ref{ssystemparabolicity}) and $\mathcal{N}$ satisfies (\ref{ssystemobliqueness}) and  up to the $1^\text{st}$ and the $0^\text{th}$ compatibility conditions are satisfied on $\mathcal{D}$ and $\mathcal{N}$ respectively. Suppose that $f\in \CPN{1+\a}$, $\Phi\in\CPN{1+\a}$, $\varphi\in\CPN{2+\a}$. Then, any solution $u\in\CPN{2+\a}$ to (\ref{linearsystem}) satisfies
		\begin{align*}
		|u|_{{2+\a};T}&\leq C\left(|u|_{0;T}+|f|_{\a;T}+|\varphi|_{2,\a}+|\Phi|_{{1+\a};\partial;T}\right)\ .
		\end{align*}
	\end{proposition}
	\begin{proof}
		We work in the coordinates of (\ref{linearsystemboundary}). We take open simply connected $U''\subset U'\subset U\subset M_0$ such that $\partial U \setminus\partial M_0$, $\partial U' \setminus\partial M_0$, $\partial U'' \setminus\partial M_0$ are a positive distance apart. We define $U_T=U\times[0,T)$, $U'_T=U'\times[0,T)$, $U''_T=U''\times[0,T)$. We will denote $(2+\a)$- H\"older norms restricted to these parabolic domains by $|\cdot|_{2+\a; U; T}$, $|\cdot|_{2+\a; U';T}$, $|\cdot|_{2+\a; U'';T}$  respectively (and similar for other norms). Applying local Schauder estimates (\cite[Theorem IV.10.1, page 351-352]{LSU}) for the Dirichlet problem yields
		\[|u^I|_{2+\a;U';T}\leq C(|u|_{1+\a;U;T}+|f|_{\a;U;T}+|\varphi|_{2+\a;U})\ .\]
		Applying Schauder estimates to the Neumann problem given by $u^n$ we have
		\begin{align*}
		|u^n|_{2+\a;U'';T}&\leq C(|u^I|_{2+\a;U';T}+|u|_{1+\a;U';T}+|f|_{\a;U';T}+|\varphi|_{2+\a;U}+|\Phi|_{1+\a;\partial U \cap \partial M,T})\\
		&\leq C(|u|_{1+\a;U;T}+|f|_{\a;U;T}+|\varphi|_{2+\a;U}+|\Phi|_{1+\a;\partial M\cap\partial U;T})\ .
		\end{align*}
		We may get similar estimates on the interior, and patching them together gives
		\begin{align*}
		|u|_{\CPNn{2+\a}}&\leq C\left(|u|_\CPNn{1+\a}+|f|_\CPNn{\a}+|\varphi|_\CENn 2\a+|\Phi|_{\CPNbn{1+\a}}\right)\ .
		\end{align*}
		Ehrling's Lemma now yields the claimed estimate.
	\end{proof}
	
	The following is now a simple application of standard PDE theory.
	\begin{proposition}\label{linearsystemexists}
		Suppose that (\ref{linearsystem}) is as in Proposition \ref{linearsystemholder}. Then there exists a solution $u\in \CPN{2+\a}$ to (\ref{linearsystem}).
	\end{proposition}
	\begin{proof}[Proof Sketch]
		This follows exactly as in \cite[Lemma 2.6]{Pulemotov} and \cite[Section IV.7]{LSU}
		
		We start by assuming that $b=0$, $c=0$, $e=0$, and $\beta(a\otimes b) = \ip{\mu_0}{a}\ip{\nu_0}{b}\tilde{\beta}(x)$. This implies that in the coordinates as in (\ref{linearsystemboundary}) the system is totally decoupled (and $\beta_n^n$ is the only nonzero component of $\beta$). On any simply connected open local patch $U$ as above, we may therefore locally solve (by imposing extra Dirichlet boundary conditions on $\partial U_T\setminus S_T$ for $1\leq I \leq n-1$ and Neumann boundary conditions of $u^n$). We may then use cutoff functions to get an approximate solution to \eqref{linearsystem} by patching together local solutions using cutoff functions, as in \cite[Lemma 2.6]{Pulemotov} and \cite[Section IV.7]{LSU}. Then, by restricting the time interval to $T<\e$ (where $\e$ depends only on the coefficients of (\ref{linearsystem}) the error between our approximate solution becomes small, and (again, as in \cite[Lemma 2.6]{Pulemotov} and \cite[Section IV.7]{LSU}) this may be used to produce left and right inverses to the linear system, and so demonstrate the existence of a solution of (\ref{linearsystem}) for $t<\e$. 
		
		The H\"older estimates of Proposition \ref{linearsystemholder} and Lemma \ref{simpleC0} below imply that we may now apply the method of continuity to ensure the existence of a solution in the case we do not make the above assumptions on the coefficients of $\mathcal{L}$, $\mathcal{D}$, $\mathcal{N}$.
		
		As uniform H\"older estimates hold, repeatedly applying the above short time existence, we may extend this solution to all of the time interval $[0,T)$. 
	\end{proof}
\begin{lemma}\label{simpleC0}
		Suppose that $\mathcal{L}$, $\mathcal{D}$, $\mathcal{N}$,  $\varphi$, $\Phi$, $f$, are as in Proposition \ref{linearsystemholder} and suppose that $u\in\CPN{2+\a}$ is a solution of \eqref{linearsystem}. Then there exists a constant $C$ depending only on $M_0$, $T$ and the coefficients of the equations  such that
		\[|u(x,\cdot)-\varphi(\cdot)|<C(|f|_{0;T}+|\Phi|_{1;\partial;T})\sqrt{t}\ .\]
	\end{lemma}
	\begin{proof}
		Due to the assumptions on differentiability, we may (wlog) assume that $\varphi=0$ and look for a suitable bound on $u$. The main technicality here is reducing estimates on $|u|$ to a standard PDE problem.
		
		At the boundary we define the normal vector field $\tilde{\beta} := \beta(\mu_0)$. %=\sum_{i=1}^n \beta(\mu_0\oplus E_i)E_i$ where $E_i$ is a local orthonormal basis such that $\nu_0=E_n$. 
		Suppose that $\beta$ is in  $\CPNb{2+\a}$. We extend $\tilde{\beta}$ at time $t=0$ so that it satisfies compatibility conditions at $t=0$ and then solve the Dirichlet problem
		\[
		\begin{cases}
		\left(\ddt{}+a^{ij} \np^2_{ij}\right)\tilde{\beta}=0 & \text{ on }M_0\times[0,T)\\
		\tilde{\beta}(x,t) = \tilde{\beta}(x,0) & \text{ on }\p M_0\times[0,T)\\
		\end{cases}
		\]
		with this initial data. This gives a solution $\tilde{\beta}\in\CPN{2+\a}$. A priori, $\beta$ is only in $\CPNb{1+\a}$, but importantly in our estimates we will only use that $\tilde{\beta}\in\CPN{1}$, and so by approximation the full lemma will be achieved. We may extend $\nu_0$ to be a smooth normal vector field in a collar region of $\partial M_0$. Due to the $\CPN{1}$ bound we know that by restricting the collar region further, $\ip{\tilde{\beta}}{\nu_0}>\frac{1}{2}\chi$. Therefore, by choosing a suitable cutoff function $\gamma$, there exists a $\Lambda=\Lambda(\chi)$ such that the scalar product $p$ on $NM_0$ defined by
		\[p(X,Y) = \Lambda\left[\ip{X}{Y} - \gamma \ip{X}{\nu_0}\ip{Y}{\nu_0}\right]+\gamma\ip{\tilde{\beta}}{X}\ip{\tilde{\beta}}{Y}\]
		is positive definite. 
		
		Working in coordinates and writing $\ip{u}{v}_p = u^ip_{ij}v^j$ we have that
		\begin{flalign*}
		\left(\ddt{} -a^{ij}D_{ij}\right)|u|_p^2&= - 2D_i u^k a^{ij}D_j u^lp_{kl}+2\ip{b(Du)+c(u)+f}{u}_p\\
		&\qquad -2D_i(p_{kl})a^{ij}D_ju^ku^l + u^ku^l\left(\ddt{} -a^{ij}D_{ij}\right)p_{kl}\\
		&\leq C_\mathcal{L}(|u|_p^2+|f|^2_{0;T})
		\end{flalign*}
		where we used the uniform parabolicity of (\ref{linearsystem}), the fact that $p$ is positive definite and Young's inequality on the last line so that $C_\mathcal{L}$ depends  on the coefficients of $\mathcal{L}$, $p$ and its derivatives. Note however that $C_\mathcal{L}$ does not depend on more than the first space derivatives of $\tilde{\beta}$.
		
		At the boundary, as $u^n$ is the only nonzero component of $u$ we have that
		\[|u|_p^2 = (\beta_n^n)^2(u_n)^2\]
		and 
		\begin{flalign*}
		-\mu_0(|u|^2_p)+u^ku^l\mu_0(p_{kl}) &=D_n|u|^2_p-u^ku^lD_n(p_{kl})=2D_nu^ip_{ij}u^j= 2u^n D_nu^ip_{in}= 2\beta_n^nu^nD_nu^i\beta^n_i \\
		&= 2\beta_n^nu^n(-D_Iu^n \beta^I_n-e_nu^n+\Phi) \ .
		\end{flalign*}
		Assuming that we are at a nonzero boundary maximum of $|u|_p$, we have that $D_Iu^n=0$ and so at such a point,
		\begin{flalign*}
		\mu_0(|u|^2_p) &< C_\partial (|u|^2_p+t|\Phi|^2_{1;\partial;T}) \ ,
		\end{flalign*}
		where $C_\partial$ depends on $e$, $\Phi$, $\chi$ and the first derivative of $p_{ij}$ (here we have used that $|\Phi(x, t) - \Phi(x,0)|=|\Phi(x,t)|<|\Phi|_{1;\partial;T}\sqrt{t}$).
		
		Let $\rho$ be a smoothing of the distance to the boundary function (as in Lemma \ref{rho}) so that $\n \rho = -\mu_0$ at the boundary. We estimate
		\[\left(\ddt{} - a^{ij}D^2_{ij}\right)\rho = - a^{ij}D^2_{ij}\rho <C_\rho, \qquad |\n \rho|<C_{\rho}, \qquad |\rho|<C_\rho\ .\]
		
		We set $v=(|u|^2_p+t|\Phi|^2_{1;\partial;T})e^{C_\partial\rho}$ and note that at a boundary maximum of $v$, we are also at a maximum over $\partial M_0$ of $|u|_p$ and so we have that $\mu_0\left(v\right) < 0$. Therefore, $v$ does not attain it's maximum at the boundary.
		
		At a positive maximum of $v$
		\begin{flalign*}
		\left(\ddt{} - a^{ij}D^2_{ij}\right)v&\leq e^{C_\partial\rho}\left[C_\mathcal{L}(|u|^2_p+|f|^2_{0;T})+|\Phi|^2_{1;\partial;T}+(|u|_p^2+t|\Phi|^2_{1;\partial;T})(C_\partial C_\rho - C_\partial^2\n_i \rho a^{ij}\n_j\rho)\right.\\
		&\qquad\qquad\left. - 2C_\partial a^{ij}D_i |u|^2_p D_j \rho\right]\\
		&\leq C(C_\mathcal{L},C_\partial)\left[v+|f|^2_{0;T}+|\Phi|^2_{1;\partial;T}\right]\ ,
		\end{flalign*}
		where we estimate the last term using the fact that at a maximum $D_i |u|_p^2=-C_\partial (|u|_p^2+t|\Phi|^2_{1;\partial;T})D_i\rho$. As $M_0$ is compact and $\rho<R$, standard maximum principle methods now imply
		\[v\leq C(|f|^2_{0;T}+|\Phi|^2_{1;\partial;T})te^{Ct}\ .\]
	\end{proof}

\end{appendix}

\bibliographystyle{amsplain}
\bibliography{LagBoundary}		

\providecommand{\bysame}{\leavevmode\hbox to3em{\hrulefill}\thinspace}
\providecommand{\MR}{\relax\ifhmode\unskip\space\fi MR }
\providecommand{\MRhref}[2]{%
  \href{http://www.ams.org/mathscinet-getitem?mr=#1}{#2}
}
\providecommand{\href}[2]{#2}
\begin{thebibliography}{10}

\bibitem{Buckland}
J.~A. Buckland, \emph{Mean curvature flow with free boundary on smooth
  hypersurfaces}, Journal f{\"u}r die Reine und Angewandte Mathematik
  \textbf{586} (2005), 71--91.

\bibitem{Butscher_2003}
A.~Butscher, \emph{Deformations of minimal {L}agrangian submanifolds with
  boundary}, Proceedings of the American Mathematical Society \textbf{131}
  (2003), no.~06, 1953--1965.

\bibitem{Butscher_2004}
\bysame, \emph{Regularizing a singular special {L}agrangian variety},
  Communications in Analysis and Geometry \textbf{12} (2004), no.~4, 733--791.

\bibitem{EckerMinkowskiDBC}
K.~Ecker, \emph{Interior estimates and longtime solutions for mean curvature
  flow of noncompact spacelike hypersurfaces in {M}inkowski space}, Journal of
  Differential Geometry \textbf{45} (1997), 481--498.

\bibitem{EckerHuiskenInteriorEstimates}
K.~Ecker and G.~Huisken, \emph{Interior estimates for hypersurfaces moving by
  mean curvature}, Inventiones mathematicae \textbf{105} (1991), 547--569.

\bibitem{EdelenBrakke}
N.~Edelen, \emph{The free-boundary {B}rakke flow}, ArXiv preprint, to appear in
  Journal f\"ur die {R}eine und {A}ngewandte {M}athematik.

\bibitem{Edelen}
\bysame, \emph{Convexity estimates for mean curvature flow with free boundary},
  Advances in Mathematics \textbf{294} (2016), 1--36.

\bibitem{Evans2018}
C.~G. Evans, J.~D. Lotay, and F.~Schulze, \emph{Remarks on the self-shrinking
  {C}lifford torus}, 2019, to appear in Journal f{\"u}r die {R}eine und
  {A}ngewandte {M}athematik, doi:10.1515/crelle-2019-0015.

\bibitem{fukaya_oh_ohta_ono_2010}
K.~Fukaya, Y.-G. Oh, H.~Ohta, and K.~Ono, \emph{{L}agrangian intersection floer
  theory}, AMS/IP Studies in Advanced Mathematics, 2010.

\bibitem{Gerhardt}
C.~Gerhardt, \emph{Global regularity of the solutions to the capillarity
  problem}, Annali della Scuola Normale Superiore Pisa, Classe di Scienze
  $4^{\text{e}}$ s\'erie \textbf{3} (1976), 157--175.

\bibitem{Panagiotis}
P.~Gianniotis, \emph{The {R}icci flow on manifolds with boundary}, Journal of
  Differential Geometry \textbf{104} (2016), no.~2, 291--324.

\bibitem{Groh2007}
K.~Groh, M.~Schwarz, K.~Smoczyk, and K.~Zehmisch, \emph{Mean curvature flow of
  monotone {L}agrangian submanifolds}, Math. Z. \textbf{257} (2007), no.~2,
  295--327. \MR{2324804}

\bibitem{Huiskengraph}
G.~Huisken, \emph{Non-parametric mean curvature evolution with boundary
  conditions}, Journal of Differential Equations \textbf{77} (1989), 369--378.

\bibitem{joyce_2015}
D.~Joyce, \emph{Conjectures on {B}ridgeland stability for {F}ukaya categories
  of {C}alabi-–{Y}au manifolds, special {L}agrangians, and {L}agrangian mean
  curvature flow}, EMS Surveys in Mathematical Sciences \textbf{2} (2015),
  no.~1, 1–62.

\bibitem{LSU}
O.~A. Lady{\v{z}}henskaya, V.~A. Solonikov, and N.~N. Uraltseva, \emph{Linear
  and quasi-linear equations of parabolic type}, vol.~23, Translations of
  Mathematical Monographs, 1968.

\bibitem{LambertTorus}
B.~Lambert, \emph{The constant angle problem for mean curvature flow inside
  rotational tori}, Mathematical Research Letters \textbf{21} (2014), no.~3,
  537--551.

\bibitem{LambertConstruction}
\bysame, \emph{Construction of maximal hypersurfaces with boundary conditions},
  Manuscripta Mathematica \textbf{153} (2017), 431--454.

\bibitem{Lieberman}
G.~M. Lieberman, \emph{Second order parabolic differential equations}, World
  Scientific, 1996.

\bibitem{LiraWanderley}
J.~H. Lira and G.~A. Wanderley, \emph{Mean curvature flow of {K}illing graphs},
  Transactions of the American Mathematical Society \textbf{367} (2015),
  4703--4726.

\bibitem{mantegazza_novaga_alessandra_schulze_2018}
C.~Mantegazza, M.~Novaga, P.~Alessandra, and F.~Schulze, \emph{Evolution of
  networks with multiple junctions}, 2018, ArXiv preprint, arXiv:1611.08254.

\bibitem{MSIneq}
J.~H. Michael and L.~M. Simon, \emph{{S}obolev and mean value inequalities on
  generalised submanifolds in $\mathbb{R}^n$}, Communications on Pure and
  Applied Mathematics \textbf{26} (1973), 361--379.

\bibitem{Neves2013}
A.~Neves, \emph{Finite time singularities for {L}agrangian mean curvature
  flow}, Ann. of Math. (2) \textbf{177} (2013), no.~3, 1029--1076. \MR{3034293}

\bibitem{Petersen_2018}
Peter Petersen, \emph{{R}iemannian geometry}, Springer International
  Publishing, 2018.

\bibitem{PriwitzerDirichlet}
B.~Priwitzer, \emph{Mean curvature flow with {D}irichlet boundary conditions in
  {R}iemannian manifolds with symmetries}, Annals of Global Analysis and
  Geometry \textbf{23} (2003), 157--171.

\bibitem{Pulemotov}
A.~Pulemotov, \emph{Quasilinear parabolic equations and the {R}icci flow on
  manifolds with boundary}, Journal fur die reine und angewandte Mathematik
  \textbf{683} (2013), 97--118.

\bibitem{savas-halilaj_smoczyk_2019}
A.~Savas-Halilaj and K.~Smoczyk, \emph{{L}agrangian mean curvature flow of
  {W}hitney spheres}, Geometry \& Topology \textbf{23} (2019), no.~2,
  1057--1084.

\bibitem{seidel_2008}
P.~Seidel, \emph{{F}ukaya categories and {P}icard–-{L}efschetz theory},
  E.M.S., Zurich, 2008.

\bibitem{Smoczyk1996}
K.~Smoczyk, \emph{A canonical way to deform a {{L}agrangian} submanifold},
  ArXiv preprint, arXiv:dg-ga/9605005.

\bibitem{Smoczyk2000}
\bysame, \emph{The {L}agrangian mean curvature flow ({D}er {L}agrangesche
  mittlere kr\"ummungsfluss)}, habilitation, Universit\"at Leipzig, 2000.

\bibitem{Smoczyk2002}
\bysame, \emph{Angle theorems for the {{L}agrangian} mean curvature flow},
  Mathematische Zeitschrift \textbf{240} (2002), no.~4, 849--883.

\bibitem{Stahlsecond}
A.~Stahl, \emph{Convergence of solutions to the mean curvature flow with a
  {N}eumann boundary condition}, Calculus of Variations and Partial
  Differential Equations \textbf{4} (1996), 421--441.

\bibitem{Stahlfirst}
\bysame, \emph{Regularity estimates for solutions to the mean curvature flow
  with a {N}eumann boundary condition}, Calculus of Variations and Partial
  Differential Equations \textbf{4} (1996), 385--407.

\bibitem{Thomas2002}
R.~P. Thomas and S.-T. Yau, \emph{Special {{L}agrangians}, stable bundles and
  mean curvature flow}, Communications in Analysis and Geometry \textbf{10}
  (2002), no.~5, 1075--1113.

\bibitem{ThorpeDirichlet}
B.~S. Thorpe, \emph{A regularity theorem for graphic spacelike mean curvature
  flows}, Pacific Journal of Mathematics \textbf{255} (2012), no.~2, 463--487.

\bibitem{Weidemaier}
P.~Weidemaier, \emph{Local existence for parabolic problems with fully
  nonlinear boundary conditions; an $l^p$-approach}, Ann. Mat. Pura Appl.
  \textbf{160} (1991), no.~4, 207--222.

\bibitem{WheelersDirichletNeumann}
G.~Wheeler and V.~M. Wheeler, \emph{Mean curvature flow with free boundary
  outside a hypersphere}, Trans. Amer. Math. Soc. \textbf{369} (2017),
  8319--8342.

\bibitem{Wheelerhalfspace}
V.~M. Wheeler, \emph{Mean curvature flow of entire graphs in a half-space with
  a free boundary}, Journal f{\"u}r die reine und angewandte Mathematik
  \textbf{690} (2014), 115--131.

\bibitem{WheelerRotSym}
\bysame, \emph{Non-parametric radially symmetric mean curvature flow with free
  boundary}, Mathematische Zeitschrift \textbf{276} (2014), no.~1, 281--298.

\bibitem{Wood}
A.~Wood, \emph{Singularities of equivariant {L}agrangian mean curvature flow},
  2019, ArXiv preprint, arXiv:1910.06122.

\end{thebibliography}

\end{document}